\documentclass{amsart}

\usepackage{geometry}               
\usepackage{graphicx}
\usepackage{amssymb}
\usepackage[all]{xy}
\usepackage{tikz}
\usetikzlibrary{arrows}
\usepackage{mathrsfs}
\usepackage{rotating}
\usepackage{enumitem,subfigure}

\definecolor{myurlcolor}{rgb}{0,0,0.7}
\usepackage{hyperref}
\hypersetup{colorlinks,
linkcolor=myurlcolor,
citecolor=myurlcolor,
urlcolor=myurlcolor}

\newtheorem{theorem}{Theorem}
\newtheorem{lemma}[theorem]{Lemma}

\newtheorem{proposition}[theorem]{Proposition}
\newtheorem{corollary}[theorem]{Corollary}
\newtheorem{defn}[theorem]{Definition}

\newtheorem*{theorem*}{Theorem}
\newtheorem*{definition*}{Definition}
\newtheorem*{corollary*}{Corollary}
\newtheorem*{proposition*}{Proposition}
\newtheorem*{example*}{Example}

    \newtheorem{innercustomthm}{Theorem}
\newenvironment{customthm}[1]
  {\renewcommand\theinnercustomthm{#1}\innercustomthm}
  {\endinnercustomthm}
\newcommand{\Top}{\mathrm{Top}}
\newcommand{\Fin}{\mathrm{Fin}}
\newcommand{\Set}{\mathrm{Set}}
\newcommand{\Vect}{\mathrm{Vect}}
\newcommand{\NTop}{\Top^\N}
\newcommand{\Op}{\mathrm{Op}}

\newcommand{\Com}{\mathrm{Com}}

\newcommand{\Assoc}{\mathrm{Assoc}}
\newcommand{\Phyl}{\mathrm{Phyl}}

\newcommand{\Tree}{\mathrm{Tree}}
\newcommand{\PTree}{\mathrm{PTree}}

\newcommand{\R}{\mathbb{R}}
\newcommand{\N}{\mathbb{N}}

\newcommand{\define}[1]{{\bf \boldmath{#1}}}
\newcommand{\T}{\mathscr{T}}
\newcommand{\maps}{\colon}

\newcommand{\End}{\mathrm{End}}

\newcommand{\iso}{\cong}
\newcommand{\tensor}{\otimes}

\newcommand{\op}{\mathrm{op}}

\newcommand{\CAT}{\mathrm{CAT}}

\newcommand{\inc}{\mathrm{in}}
\newcommand{\wco}{W}

\newcommand{\Coend}{\mathrm{Coend}}

\newcommand\ellipsefoci[4]{
  \path[#1] let \p1=(#2), \p2=(#3), \p3=($(\p1)!.5!(\p2)$)
  in \pgfextra{
    \pgfmathsetmacro{\angle}{atan2(\x2-\x1,\y2-\y1)}
    \pgfmathsetmacro{\focal}{veclen(\x2-\x1,\y2-\y1)/2/1cm}
    \pgfmathsetmacro{\lentotcm}{\focal*2*#4}
    \pgfmathsetmacro{\axeone}{(\lentotcm - 2 * \focal)/2+\focal}
    \pgfmathsetmacro{\axetwo}{sqrt((\lentotcm/2)*(\lentotcm/2)-\focal*\focal}
  }
  (\p3) ellipse[x radius=\axeone cm,y radius=\axetwo cm, rotate=\angle];
}

\usetikzlibrary{calc}

\title{Operads and Phylogenetic Trees}

\date{July 28, 2017}                     

\begin{document}

\author{John Baez}
\address{Department of Mathematics, University of California, Riverside CA 92521, USA \\ and Centre for Quantum Technologies, National University of Singapore, 
         Singapore 117543}
\email{baez@math.ucr.edu}
\author{Nina Otter}
\address{Mathematical Institute, University of Oxford, Oxford OX2 6GG, UK}
\email{otter@maths.ox.ac.uk}

\begin{abstract} 
We construct an operad $\Phyl$ whose operations are the edge-labelled trees used in phylogenetics. This operad is the coproduct of $\Com$, the operad for commutative semigroups, and $[0,\infty)$, the operad with unary operations corresponding to nonnegative real numbers, where composition is addition.  We show that there is a homeomorphism between the space of $n$-ary operations of $\Phyl$ and $\T_n\times [0,\infty)^{n+1}$, where $\T_n$ is the space of metric $n$-trees introduced by Billera, Holmes and Vogtmann. Furthermore, we show that the Markov models used to reconstruct phylogenetic trees from genome data give coalgebras of $\Phyl$.   These always extend to coalgebras of the larger operad $\Com + [0,\infty]$, since Markov processes on finite sets converge to an equilibrium as time approaches infinity.  We show that for any operad $O$, its coproduct with $[0,\infty]$ contains the operad $\wco(O)$ constucted by Boardman and Vogt.  To prove these results, we explicitly describe the coproduct of operads in terms of labelled trees.
 \end{abstract}

\maketitle

 \section{Introduction}
\label{sec:intro}
 
Trees are important, not only in mathematics, but also biology.  The most important is the `tree of life' relating all organisms that have ever lived on Earth.  Darwin drew this sketch of it in 1837:

\vskip 1em

\begin{center}
\href{http://en.wikipedia.org/wiki/Tree_of_life_\%28biology\%29\#Darwin.27s_tree_of_life}{ 
\raisebox{-0.0 em}{\includegraphics[scale = 0.25]{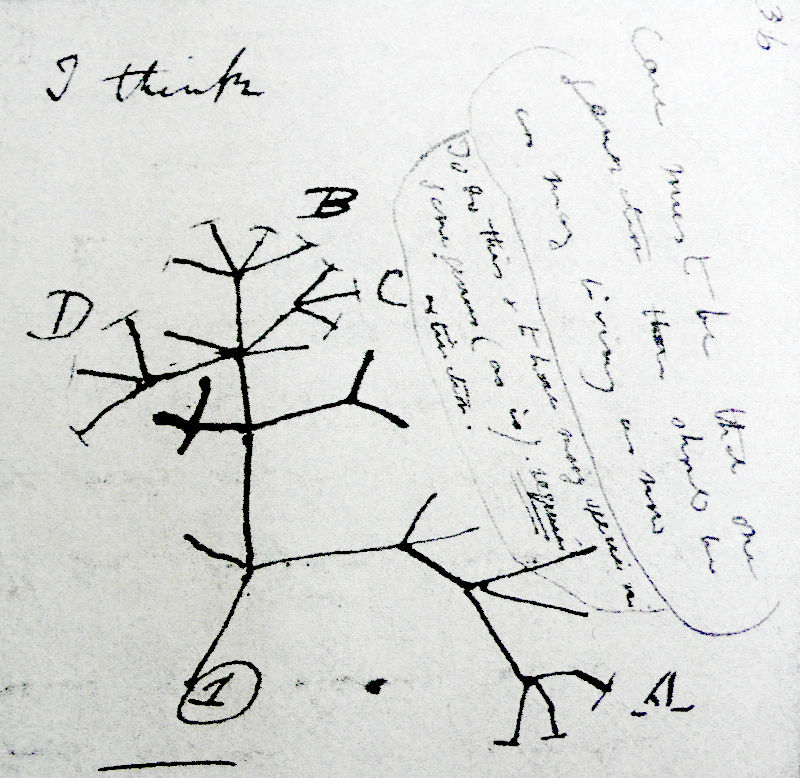}} 
}
\end{center}

\vskip 1em
He wrote about it in \textsl{On the Origin of Species} \cite{Dar}:

\begin{quote}
The affinities of all the beings of the same class have sometimes been represented by a great tree. I believe this simile largely speaks the truth. The green and budding twigs may represent existing species; and those produced during former years may represent the long succession of extinct species. At each period of growth all the growing twigs have tried to branch out on all sides, and to overtop and kill the surrounding twigs and branches, in the same manner as species and groups of species have at all times overmastered other species in the great battle for life. 
\end{quote}

Now we know that the tree of life is not really a tree in the mathematical
sense \cite{Doo}.  One reason is `endosymbiosis': the incorporation
of one organism together with its genetic material into another, as
probably happened with the mitochondria in our cells and also the
plastids that hold chlorophyll in plants.  Another is `horizontal gene
transfer': the passing of genetic material from one organism to
another, which happens frequently with bacteria.  So, the tree of life
is really a thicket, as shown in this figure \cite{Sme}:

\vskip 1em
\begin{center}
\href{http://en.wikipedia.org/wiki/Horizontal_gene_transfer}{ 
\includegraphics[scale = 1.5]{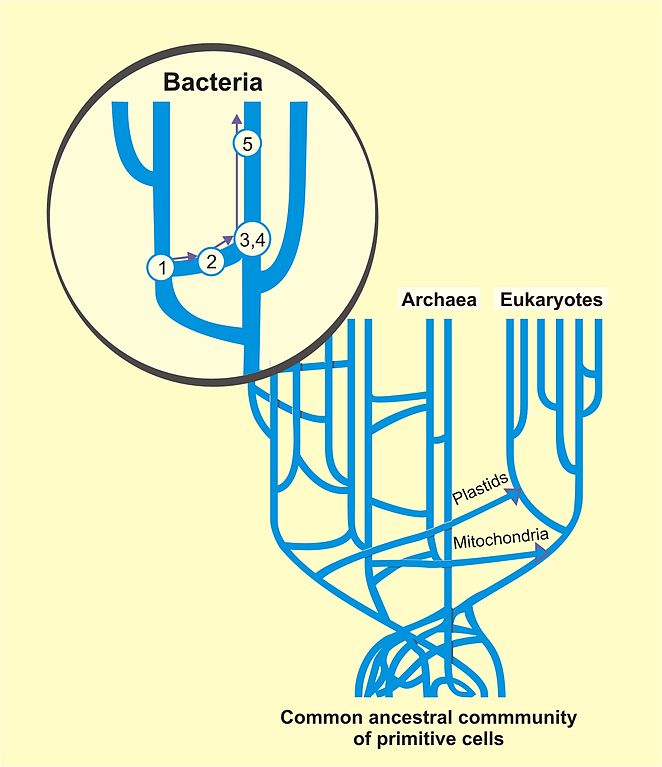}
}
\end{center}

\vskip 1em
In addition, the concept of `species' is imprecise and hotly debated \cite{Hey}.
Nonetheless, a tree with species as branches is a widely used approximation to the 
complex reality of evolution, especially for animals and plants in the last few hundred million years. Thus, biologists who try to infer phylogenetic trees from present-day genetic data often use simple models where:

\begin{itemize}
\item 
the genotype of each species follows a random walk, but  
\item 
species branch in two at various times.   
\end{itemize}

These are called `Markov models'.  The simplest Markov model for DNA evolution is the Jukes--Cantor model \cite{JC}.  Consider one or more pieces of DNA having a total of $N$ base pairs.  We can think of this as a string of letters chosen from the set \{A,T,C,G\}:

\begin{center}
   $\cdots$ ATCGATTGAGCTCTAGCG $\cdots$
\end{center}

\noindent
As time passes, the Jukes--Cantor model says the DNA changes randomly, with each base pair having the same constant rate of randomly flipping to any other.  So, we get a Markov process on the set 
\[X = \{\textrm{A,T,C,G}\}{}^N \]
However, a species can also split in two.  So, given current-day genetic data from
various species, biologists try to infer the most probable tree where, starting
from a common ancestor, the DNA in question undergoes a random walk most of the time but branches in two at certain times.   

To formalize this, we can define a concept of `phylogenetic tree'.   Our work is based on the definition of Billera, Holmes and Vogtmann \cite{BHV}, though we use a slightly different definition, for reasons that will soon become clear.  For us, a phylogenetic tree is a rooted tree with leaves labelled by numbers $1,2, \dots, n$ and edges labelled by `times' or, geometrically speaking, `lengths' in $[0, \infty)$.    We require that:
\begin{itemize}
\item the length of every edge is positive, except perhaps for `external edges': that is, edges incident to the leaves or root;
\item there are no 1-ary vertices.
\end{itemize}
For example, here is a phylogenetic tree with 5 leaves:
\[
\begin{tikzpicture}
\node at (0,1) {$\bullet$}; 
\node at (2,3){$\bullet$};
\node at (-0.65,3) {$\ell_1$};
\node at (0.25,3) {$\ell_2$};
\node at (1.2,3.5) {$\ell_3$};
\node at (1.8,3.7) {$\ell_4$};
\node at (2.9,3.5) {$\ell_5$};
\node at (1.1,1.7){$\ell_7$};
\node at (0.3,0.5) {$\ell_6$};

\node at (-1.5,4.3){$3$};
\node at (0,4.3){$1$};
\node at (1,4.3){$4$};
\node at (2,4.3){$5$};
\node at (3.1,4.3){$2$};
\node at (0,-0.3){$0$};

\path[-,font=\scriptsize]
(0,0) edge (0,1)
(0,1) edge (-1.5,4)
(0,1) edge (0,4)
(0,1) edge (1,2)
(2,3) edge (1,4)
(2,3) edge (2,4)
(1,2) edge (3,4);
\end{tikzpicture}
\]
where $\ell_1, \dots, \ell_6 \ge 0$ but we demand that $\ell_7 > 0$.    We draw the vertices as dots.  We do not count the leaves and the root as vertices, and we label the root with the number $0$.  We cannot collapse edges of length zero that end at  leaves, since doing so would eliminate those leaves.   Also note that the embedding of the tree in the plane is irrelevant, so this counts as the same phylogenetic tree:
\[
\begin{tikzpicture}
\node at (0,1) {$\bullet$}; 
\node at (2,3){$\bullet$};
\node at (-0.65,3) {$\ell_2$};
\node at (0.25,3) {$\ell_1$};
\node at (1.2,3.5) {$\ell_3$};
\node at (1.8,3.7) {$\ell_4$};
\node at (2.9,3.5) {$\ell_5$};
\node at (1.1,1.7){$\ell_7$};
\node at (0.3,0.5) {$\ell_6$};

\node at (-1.5,4.3){$1$};
\node at (0,4.3){$3$};
\node at (1,4.3){$4$};
\node at (2,4.3){$5$};
\node at (3.1,4.3){$2$};
\node at (0,-0.3){$0$};

\path[-,font=\scriptsize]
(0,0) edge (0,1)
(0,1) edge (-1.5,4)
(0,1) edge (0,4)
(0,1) edge (1,2)
(2,3) edge (1,4)
(2,3) edge (2,4)
(1,2) edge (3,4);
\end{tikzpicture}
\]

While the phylogenetic trees that we consider here are rooted, `unrooted' trees, i.e.\ trees without a specified root, are also used in phylogenetics \cite[Chap.\ 3]{BS}. Biologists use such trees to represent uncertainty about  the direction in which the evolution occurred among the species. 

In applications to biology, we are often interested in trees where the total distance from the root to the leaf is the same for every  leaf, since all species have evolved for the same time from their common ancestor. These are mathematically interesting as well, because then the distance between any two leaves defines an ultrametric on the set of leaves \cite{RTV}.  However, more general phylogenetic trees are also interesting---and they become essential when we construct an operad whose operations are phylogenetic trees.

Let $\Phyl_n$ be the set of phylogenetic trees with $n$ leaves.  This has 
a natural topology, which we explain in Section \ref{A:topology of phyl}.  For example, here is a continuous path in $\Phyl_4$ where we only change the length of one internal edge, reducing it until it becomes zero and we can collapse it:
\[
\begin{tikzpicture}[scale=0.5]
\node (1) at (0,1){$\bullet$};
\node (2) at (2,2){$\bullet$};
\node (3) at (-2,2){$\bullet$}; 
\node at (0,-1.5){$0$};
\node at (-3,4.5){$1$};
\node at (-1,4.5){$2$};
\node at (1,4.5){$3$};
\node at (3,4.5){$4$};
\node at (1.3,1.2){$1$};
\node at (-1.3,1.2){$1$};
\node at (-3.1,3){$0.6$};
\node at (-0.9,3){$0.6$};
\node at (3.1,3){$0.6$};
\node at (0.9,3){$0.6$};
\node at (-0.8,-0.1){$1.2$};
\node at (5,1.5){$\leadsto$};
\path[-,font=\scriptsize]
(0,-1) edge (0,1)
(0,1) edge (2,2)
(0,1) edge (-2,2)
(-2,2) edge (-3,4)
(-2,2) edge (-1,4)
(2,2) edge (1,4)
(2,2) edge (3,4);
\end{tikzpicture}
\begin{tikzpicture}[scale=0.5]
\node (1) at (0,1){$\bullet$};
\node (2) at (1.5,1.5){$\bullet$};
\node (3) at (-2,2){$\bullet$}; 
\node at (0,-1.5){$0$};
\node at (-3,4.5){$1$};
\node at (-1,4.5){$2$};
\node at (1,4.5){$3$};
\node at (3,4.5){$4$};
\node at (1.05,0.9){$0.5$};
\node at (-1.3,1.2){$1$};
\node at (-3.1,3){$0.6$};
\node at (-0.9,3){$0.6$};
\node at (3.1,3){$0.6$};
\node at (0.6,3){$0.6$};
\node at (-0.8,-0.1){$1.2$};
\node at (5,1.5){$\leadsto$};
\path[-,font=\scriptsize]
(0,-1) edge (0,1)
(0,1) edge (1.5,1.5)
(0,1) edge (-2,2)
(-2,2) edge (-3,4)
(-2,2) edge (-1,4)
(1.5,1.5) edge (1,4)
(1.5,1.5) edge (3,4);
\end{tikzpicture}
\begin{tikzpicture}[scale=0.5]
\node (1) at (0,1){$\bullet$};
\node (3) at (-2,2){$\bullet$}; 
\node at (0,-1.5){$0$};
\node at (-3,4.5){$1$};
\node at (-1,4.5){$2$};
\node at (1.5,4.5){$3$};
\node at (3.5,4.5){$4$};
\node at (-1.3,1.2){$1$};
\node at (-3.1,3){$0.6$};
\node at (-0.9,3){$0.6$};
\node at (3.1,3){$0.6$};
\node at (0.3,3){$0.6$};
\node at (-0.8,-0.1){$1.2$};
\path[-,font=\scriptsize]
(0,-1) edge (0,1)
(0,1) edge (-2,2)
(-2,2) edge (-3,4)
(-2,2) edge (-1,4)
(0,1) edge (1.5,4)
(0,1) edge (3.5,4);
\end{tikzpicture}
\]

Phylogenetic trees reconstructed by biologists are typically binary.  When a phylogenetic
tree appears to have higher arity, sometimes we merely lack sufficient data to resolve
a higher-arity branching into a number of binary ones \cite{PG}.  With the topology we are using on $\Phyl_n$, binary trees form an open dense set of $\Phyl_n$, except for $\Phyl_1$.   However, trees of higher arity are still important, because paths, paths of
paths, etc.\ in $\Phyl_n$ are often forced to pass through trees of higher arity.

Billera, Holmes and Vogtmann \cite{BHV} focused their attention on the set $\T_n$ of phylogenetic trees where lengths of the external edges---edges incident to the root and leaves---are fixed to a constant value.  They endow $\T_n$ with a metric, which induces a topology on $\T_n$, and we show that for $n \ne 1$ there is a homeomorphism
\[                 \Phyl_n \iso \T_n \times [0,\infty)^{n+1} , \]
where the data in $[0,\infty)^{n+1}$ describe the lengths of the external edges in a
general phylogenetic tree.

In algebraic topology, trees are often used to describe the composition of 
$n$-ary operations.   This is formalized in the theory of operads \cite{May72}.
An `operad' is an algebraic stucture where for each natural number $n=0,1,2,\dots$ we have a set $O_n$ whose elements are considered as abstract $n$-ary operations, not necessarily operating on anything yet.  An element $f \in O_n$ can be depicted as a planar tree with one vertex and $n$ labelled leaves:
\[
\begin{tikzpicture}
\node (1) at (-1,1){}; 
\node at (-1,1.3){$2$};
\node (2) at (0,1){};
\node at (0,1.3){$1$};
\node (3) at (1,1){};
\node at (1,1.3){$3$};
\node (4) at (0,0){$\bullet$};
\node at (0.5,0){$f$};
\node (5) at (0,-1){};
\node at (0,-1.3){$0$};
\path[-,font=\scriptsize]
(-1,1) edge (0,0)
(0,1) edge (0,0)
(1,1) edge (0,0)
(0,0) edge (0,-1);
\end{tikzpicture}
\]
We can compose these operations in a tree-like way to get new operations:
\[
\begin{tikzpicture}[scale = 0.8]
\node at (-2,1){}; 
\node at (-1,1){};
\node at (0,1){};
\node at (0,0){};
\node at (2,1){};
\node at (3,1){};
\node at (-2,1.3){$3$}; 

\node at (-1,1.3){$4$};
\node at (0,1.3){$1$};
\node at (1,1.3){$6$};
\node at (2,1.3){$2$};
\node at (3,1.3){$5$};
\node at (1,-2.3){$0$};
\node at (5,1.3){$3$}; 
\node at (6,1.3){$4$};
\node at (7,1.3){$1$};
\node at (8,1.3){$6$};
\node at (9,1.3){$2$};
\node at (10,1.3){$5$};
\node at (7.5,-2.3){$0$};
\node at (-1,0){$\bullet$};
\node at (-0.3,0){$g_1$};
\node at (1,0){$\bullet$};
\node at (1.4,0){$g_2$};
\node at (2.5,0){$\bullet$};
\node at (3,0){$g_3$};
\node at (1,-1){$\bullet$};
\node at (1.5,-1){$f$};
\node at (7.5,-1){$\bullet$};
\node at (9.2,-1){$f \circ (g_1, g_2, g_3)$};
\node at (4.5,-1){$=$};
\path[-,font=\scriptsize]
(-2,1) edge (-1,0)
(-1,1) edge (-1,0)
(0,1) edge (-1,0)
(2,1) edge (2.5,0)
(1,1) edge (1,0)
(3,1) edge (2.5,0)
(-1,0) edge (1,-1)
(1,0) edge (1,-1)
(2.5,0) edge (1,-1)
(1,-1) edge (1,-2)
(5,1) edge (7.5,-1)
(6,1) edge (7.5,-1)
(7,1) edge (7.5,-1)
(8,1) edge (7.5,-1)
(9,1) edge (7.5,-1)
(10,1) edge (7.5,-1)
(7.5,-1) edge (7.5,-2);
\end{tikzpicture}
\]
and an associative law holds, making this sort of composite unambiguous:
\[
\begin{tikzpicture}[scale=1.1]
\node at (-2,1){$\bullet$}; 
\node at (-1.7,1){$h_1$};
\node at (-1,1){$\bullet$};
\node at (-0.7,1){$h_2$};
\node at (0,1){$\bullet$};
\node at (0.3,1){$h_3$};
\node at (1,1){$\bullet$};
\node at (1.3,1){$h_4$};
\node at (2,1){$\bullet$};
\node at (2.3,1){$h_5$};
\node at (3,1){$\bullet$};
\node at (3.3,1){$h_6$};
\node at (-1,0){$\bullet$};
\node at (-0.3,0){$g_1$};
\node at (1,0){$\bullet$};
\node at (1.4,0){$g_2$};
\node at (2.5,0){$\bullet$};
\node at (3,0){$g_3$};
\node at (1,-1){$\bullet$};
\node at (1.5,-1){$f$};
\node at (-2.5,2.3){$4$};
\node at (-2,2.3){$1$}; 
\node at (-1.5,2.3){$3$}; 
\node at (-1,2.3){$9$}; 
\node at (-0.2,2.3){$8$}; 
\node at (0.2,2.3){$2$}; 
\node at (0.8,2.3){$6$}; 
\node at (1.2,2.3){$5$}; 
\node at (3,2.3){$7$}; 
\node at (1,-2.3){$0$};
\path[-,font=\scriptsize]
(-2.5,2) edge (-2,1)
(-2,2) edge (-2,1)
(-1.5,2) edge (-2,1)
(-1,2) edge (-1,1)
(-0.2,2) edge (0,1)
(0.2,2) edge (0,1)
(0.8,2) edge (1,1)
(1.2,2) edge (1,1)
(3,2) edge (3,1)
(-2,1) edge (-1,0)
(-1,1) edge (-1,0)
(0,1) edge (-1,0)
(2,1) edge (2.5,0)
(1,1) edge (1,0)
(3,1) edge (2.5,0)
(-1,0) edge (1,-1)
(1,0) edge (1,-1)
(2.5,0) edge (1,-1)
(1,-1) edge (1,-2);
\end{tikzpicture}
\]
\noindent
There are various kinds of operads, but in this paper our operads will always be `unital', having an operation $1 \in O_1$ that acts as an identity for composition.   They will also be `symmetric', meaning there is an action of the symmetric group $S_n$ on each set $O_n$, compatible with composition.  Further, they will be `topological', meaning that each set $O_n$ is a topological space, with composition and permutations acting as continuous maps.

In Section \ref{sec:trees} we prove that there is an operad $\Phyl$, the `phylogenetic operad', whose space of $n$-ary operations is $\Phyl_n$.  This raises a number of questions:
\begin{itemize}
\item What is the mathematical nature of this operad?
\item  How is it related to `Markov processes with branching'? 
\item How is it related to known operads in topology?
\end{itemize}
Briefly, the answer is that $\Phyl$ is the coproduct of $\Com$, the operad for commutative topological semigroups, and $[0,\infty)$, the operad having only unary 
operations, one for each $t \in [0,\infty)$.    The first describes branching, the 
second describes Markov processes.   Moreover, $\Phyl$ is closely related to the Boardmann--Vogt $W$ construction applied to $\Com$.  This is a construction
that Boardmann and Vogt applied to another operad in order to obtain an operad whose algebras are loop spaces \cite{BV}.

To understand all this in more detail, first recall that the \textit{raison d'\^etre} 
of operads is to have `algebras'.  The most traditional sort of algebra of an operad
$O$ is a topological space $X$ on which each operation $f \in O_n$ acts as a 
continuous map 
\[            \alpha(f) \maps X^n \to  X  \]
obeying some conditions: composition, the identity, and the permutation
group actions are preserved, and $\alpha(f)$ depends continuously on $f$.  
The idea is that the abstract operations in $O$ are realized as actual operations on
the space $X$.

In this paper we instead need algebras of a linear sort.  Such an algebra of $O$ is a 
finite-dimensional real vector space $V$ on which each operation $f \in O_n$ 
acts as a multilinear map
\[            \alpha(f) \maps V^n \to V  \] 
obeying the same list of conditions.  We can also think of $\alpha(f)$ as a linear
map 
\[            \alpha(f) \maps V^{\otimes n} \to V  \]
where $V^{\otimes n}$ is the $n$th tensor power of $V$.  

We also need `coalgebras' of operads.  The point is that while ordinarily one thinks of an
operation $f \in O_n$ as having $n$ inputs and one output, a phylogenetic tree is better
thought of as having one input and $n$ outputs.  A coalgebra of $O$ is a finite-dimensional real vector space $V$ on which every operation $f \in O_n$ gives a linear map
\[        \alpha(f) \maps V \to V^{\otimes n}  \]
obeying the same conditions as an algebra, but `turned around'.   More precisely, one can define algebras of an operad $O$ in any category $C$ enriched over topological spaces, and a coalgebra of $O$ in $C$ is simply an algebra of $O$ in $C^\op$.

The main point of this paper is that the phylogenetic operad has interesting coalgebras,
which correspond to how phylogenetic trees are actually used to describe 
branching Markov processes in biology.  But to understand this, we need to start
by looking at coalgebras of two operads from which the phylogenetic operad is built.

By abuse of notation, we will use $[0,\infty)$ as the name for the operad having only unary operations, one for each $t \in [0,\infty)$, with composition of operations given by addition.  \label{defn:1-param semigroup} A coalgebra of $[0,\infty)$ is a finite-dimensional real vector space $V$ together with for each $t \in [0,\infty)$ a linear map  
\[             \alpha(t) \maps V \to V  \]
such that:
\begin{itemize}
\item $\alpha(s+t) = \alpha(s) \alpha(t)$ for all $s,t \in [0,\infty)$,
\item $\alpha(0) = 1_V$,
\item $\alpha(t)$ depends continuously on $t$, where the space of linear operators from $V$ to itself is given its usual topology as a finite-dimensional real vector space.  
\end{itemize}
Analysts call such a thing a `continuous one-parameter semigroup' of operators on $V$, though category theorists might prefer to call it a continuous one-parameter monoid.

Given a finite set $X$, a `Markov process' or `continuous-time Markov chain' on $X$ is a continuous one-parameter semigroup of operators on $\R^X$ such that if $f \in \R^X$ is a probability distribution on $X$, so is $\alpha(t) f$ for all $t \in [0,\infty)$.  Equivalently, if we think of $\alpha(t)$ as an $X \times X$ matrix of real numbers, we demand that its entries be nonnegative and each column sum to 1. Such a matrix is called `stochastic'.  If $X$ is a set of possible sequences of base pairs, a Markov process on $X$ describes the random changes of DNA with the passage of time.  Any Markov process on $X$ makes $\R^X$ into a coalgebra of $[0,\infty)$.

This handles the Markov process aspect of DNA evolution; what about the branching?  
For this we use $\Com$, the unique operad with one $n$-ary operation for 
each $n > 0$.  Algebras of $\Com$ are not-necessarily-unital commutative algebras:  
there is only one way to multiply $n$ elements for $n > 0$.   

For us what matters most is that coalgebras of $\Com$ are finite-dimensional cocommutative coalgebras, not necessarily with counit.  The real-valued functions on a finite set form a commutative algebra with pointwise operations, so the dual of this vector space is a cocommutative coalgebra.  Since the set gives a basis for this vector space, we can identify
this vector space with its dual.  Thus, if $X$ is a finite set, there is a cocommutative coalgebra whose underlying vector space is $\R^X$.   The unique $n$-ary operation of $\Com$ acts as the linear map
\[  \Delta_n   \maps  \R^X \to \underbrace{\R^X  \tensor \cdots  \tensor \R^X}_{n \; \textrm{times}}  \iso \R^{X^n} \]
where
\[     \Delta_n (f)(x_1, \dots, x_n) = 
\left\{ \begin{array}{cl}  f(x) & \textrm{if } x_1 = \cdots = x_n = x  \\ \\
                                       0   & \textrm{otherwise}   
\end{array} \right.   \]
This map describes the `$n$-fold duplication' of a probability distribution $f$
on the set $X$ of possible genes, pictured as follows:
\[
\begin{tikzpicture}[scale=0.5]
\node (2) at (0,1){$\bullet$};
\node at (-1,2.5){$1$};
\node at (0,2.5){$2$};
\node at (1,2.5){$3$};
\node at (0,-0.5){$0$};
\path[-,font=\scriptsize]
(-1,2) edge (0,1)
(0,2) edge (0,1)
(1,2) edge (0,1)
(0,1) edge (0,0);
\end{tikzpicture}
\]

Next, we wish to describe how to combine the operads $[0,\infty)$ and $\Com$ to
obtain the phylogenetic operad.  Any pair of operads $O$ and $O'$ has a 
coproduct $O + O'$.  The definition of coproduct gives an easy way to understand
the algebras of $O + O'$.  Such an algebra is simply an object that is both an algebra of $O$ and an algebra of $O'$, with no compatibility conditions imposed.
One can also give an explicit construction of $O + O'$.  When $O'$ has only unary operations,
the $n$-ary operations of $O + O'$ are certain equivalence classes of trees with 
leaves labelled $\{1,\dots, n\}$, vertices labelled by
operations in $O$, and edges labelled by operations in $O'$.    

Given this, it should come as no surprise that the operad $\Phyl$ is the coproduct $\Com + [0,\infty)$.  In fact, we shall take this as a definition.  Starting from this definition, we work backwards to show that the operations of $\Phyl$ correspond to phylogenetic trees.  We prove this in Theorem \ref{thm:bijection}.  The definition of coproduct determines a topology on the spaces $\Phyl_n$, and it is a nontrivial fact that with this topology we have $\Phyl_n \cong \T_n \times [0,\infty)^{n+1}$ for $n > 1$, where $\T_n$ has the topology defined by Billera, Holmes and Vogtmann.   We prove this in Theorem \ref{thm:homeo}. 

Using the definition of the phylogenetic operad as a coproduct, it is clear that given any Markov process on a finite set $X$, the vector space $\R^X$ naturally becomes a coalgebra of this operad.  The reason is that, as we have seen, $\R^X$ is automatically a coalgebra of $\Com$,  and the Markov process makes it into a coalgebra of $[0,\infty)$.   Thus, by the universal property of a coproduct, it becomes a coalgebra of $\Phyl \iso \Com + [0,\infty)$.   We prove this in Theorem \ref{thm:coalgebra_of_phyl}.

Proving these theorems requires a detailed understanding of the operations in a coproduct of operads.  To reach this understanding, we study the relation between an operad $O$ and its underlying collection $U(O)$, where a `collection' is simply a sequence of topological spaces, one for each natural number.  We explicitly describe the free operad $F(C)$ on a collection $C$ in Theorem \ref{thm:C-tree-2}, and describe the counit $\epsilon_O \maps F(U(O)) \to O$ in Theorem \ref{thm:counit}.  Using these results, we describe the operations in a coproduct of operads $O + O'$ in Theorem \ref{thm:coproduct}.   In Theorem \ref{thm:generalized_phyl_trees} we show how this description simplifies when $O'$ has only unary operations.   

Since operads arose in algebraic topology, it is interesting to consider how the phylogenetic operad connects to ideas from that subject.   Boardmann and Vogt 
\cite{BV} defined a construction on operads, the `$W$ construction', which when applied
to the operad for spaces with an associative multiplication gives an operad for loop spaces.  The operad $\Phyl$ has an interesting relation to $W(\Com)$.  To see this, 
define addition on $[0,\infty]$ in the obvious way, where 
\[         \infty + t = t + \infty = \infty   \]
Then $[0,\infty]$ becomes a commutative topological monoid, so we obtain 
an operad with only unary operations, one for each $t \in [0,\infty]$, where
composition is addition.  By abuse of notation, let us call this operad $[0,\infty]$.

Boardmann and Vogt's $W$ construction involves trees with edges having lengths 
in $[0,1]$, but we can equivalently use $[0,\infty]$.  Leinster \cite{Lei}
observed that for any \emph{nonsymmetric} topological operad $O$, Boardmann and Vogt's operad $W(O)$ is closely related to $O + [0,\infty]$.   Here we make this observation precise in the symmetric case.  Operations in $\Com + [0,\infty]$ are just like phylogenetic trees except that edges may have length $\infty$.  Moreover, for any operad $O$, the operad $W(O)$ is a non-unital suboperad of $O + [0,\infty]$.  An operation of $O + [0,\infty]$ lies in $W(O)$ if and only if all the external edges of the corresponding tree have length $\infty$.  We prove this in Theorem \ref{thm:W-construction}.

Berger and Moerdijk \cite{BM2} showed that if $S_n$ acts freely on $O_n$ and $O_1$ is well-pointed,  $W(O)$ is a cofibrant replacement for $O$.  This is true for $O = \Assoc$, the operad whose algebras are topological semigroups.  This cofibrancy is why Boardmann and Vogt could use $W(\Assoc)$ as an operad for loop spaces.  But $S_n$ does not act freely on $\Com_n$, and $W(\Com)$ is not a cofibrant replacement for $\Com$.  So, it is \emph{not} an operad for infinite loop spaces. 

Nonetheless, the larger operad $\Com + [0,\infty]$, a compactification of $\Phyl = \Com + [0,\infty)$, is somewhat interesting. The reason is that any Markov process $\alpha \maps [0,\infty) \to \End(\R^X)$ approaches a limit as $t \to \infty$.  Indeed, $\alpha$ extends uniquely to a homomorphism from the topological monoid $[0,\infty] $ to $\End(\R^X)$.   Thus, given a Markov process on a finite set $X$, the vector space 
$\R^X$ naturally becomes a coalgebra of $\Com + [0,\infty]$.  We prove this in 
Theorem \ref{thm:extension}.

\section{Trees and the phylogenetic operad}
\label{sec:trees}

For a graph theorist, a rooted planar tree is something like this:
\[
\begin{tikzpicture}[scale=0.5]
\node (1) at (0,0){$\bullet$};
\node (3) at (1,1){$\bullet$};
\node (2) at (-2,2){$\bullet$}; 
\node (4) at (2,2){$\bullet$};
\node (5) at (0,2){$\bullet$};

\path[-,font=\scriptsize]
(0,0) edge (1,1)
(0,0) edge (-2,2)
(1,1) edge (2,2)
(1,1) edge (0,2);
\end{tikzpicture}
\]
with a vertex called the `root' at the bottom and vertices called `leaves' at top.  Sometimes these trees are drawn upside-down.  But more importantly, ever since the pioneering work of Boardmann and Vogt \cite{BV},
operad theorists have used trees of a different sort:
\[
\begin{tikzpicture}[scale=0.5]
\node (1) at (0,1){$\bullet$};
\node (2) at (1,2){$\bullet$};
\path[-,font=\scriptsize]
(0,0) edge (0,1)
(0,1) edge (1,2)
(0,1) edge (-2,3)
(1,2) edge (2,3)
(1,2) edge (0,3);
\end{tikzpicture}
\]
These have `input edges' coming in from above and an `output edge' leaving the root from below.  These `external edges' are incident to a vertex at only one end; the other end trails off into nothingness.   So, a tree of this type is not a graph of the various kinds most commonly used in graph theory; rather, it is of the kind discussed in Appendix C.4 of Loday and Vallette's book \cite{LV}.

We want a notion of tree that is suitable for operad theory yet tailored to working with phylogenetic trees.  After painstaking thought, our choice is this:

\begin{defn}
For any natural number $n = 0, 1, 2, \dots$, an \define{$n$-tree} is a quadruple $T=(V,E,s,t)$ where:
\begin{itemize} 
\item $V$ is a finite set;
\item $E$ is a finite non-empty set  whose elements are called \define{edges};
\item $s \maps E \to V\sqcup \{1,\dots, n\}$ and $t \maps E \to V
\sqcup\{0\}$ are maps sending any edge to its \define{source} and \define{target}, respectively.
\end{itemize}
Given $u,v\in V \sqcup\{0, 1,\dots, n\}$, we write $u \stackrel{e}{\longrightarrow} v$ if $e\in E$ has $s(e)=u$ and $t(e)=v$.  

This data is required to satisfy the following conditions:
\begin{itemize}
\item $s \maps E \to V\sqcup \{1,\dots, n\}$ is a bijection; 
\item there exists exactly one $e\in E$ such that $t(e)=0$;
\item for any $v \in V\sqcup\{1,\dots , n\} $ there exists a
\define{directed edge path} from $v$ to $0$: that is, a sequence of edges $e_0, \dots, e_n$ and vertices $v_1 , \dots , v_n$ such that
\[     v \stackrel{e_0}{\longrightarrow} v_1 , \; v_1 \stackrel{e_1}{\longrightarrow} v_2, \; \dots, \; v_n \stackrel{e_n}{\longrightarrow} 0  .\]
\end{itemize}
\end{defn}

We draw $n$-trees following a convention where the source of any edge is at
its top end, while its target is at the bottom.   Here are two 3-trees:
\[
\begin{tikzpicture}[scale=0.5]
\node (1) at (0,1){$\bullet$};
\node (2) at (1,2){$\bullet$};
\node at (-2,3.5){$1$};
\node at (0,3.5){$2$};
\node at (2,3.5){$3$};
\node at (0,-0.5){$0$};
\path[-,font=\scriptsize]
(0,0) edge (0,1)
(0,1) edge (1,2)
(0,1) edge (-2,3)
(1,2) edge (2,3)
(1,2) edge (0,3);
\end{tikzpicture} 
\qquad \qquad
\begin{tikzpicture}[scale=0.5]
\node (1) at (0,1){$\bullet$};
\node at (-2,3.5){$2$};
\node at (0,3.5){$3$};
\node at (2,3.5){$1$};
\node at (0,-0.5){$0$};
\path[-,font=\scriptsize]
(0,0) edge (0,1)
(0,1) edge (0,3)
(0,1) edge (-2,3)
(0,1) edge (2,3);
\end{tikzpicture} 
\]
We draw the elements of $V$ as dots, but not the elements $0,1, \dots n$.   A graph theorist would call all the elements of $V \sqcup \{0\} \sqcup \{1, \dots, n\}$ vertices.  We, however, reserve the term \define{vertex} for an element of $V$.  We call $0$ the \define{root}, and call $1, \dots, n$ the \define{leaves}.    

We define the \define{arity} of a vertex $v\in V$ to be the cardinality of the preimage 
$t^{-1}(v)$.  We call elements of this preimage the \define{children} of $v$.  Note that for us children are edges, not vertices.  

We define a \define{terminus} to be a vertex of arity zero.  Here is a 3-tree with two  termini:
\[
\begin{tikzpicture}[scale=0.5]
\node (1) at (0,1){$\bullet$};
\node (2) at (1,2){$\bullet$};
\node at (-0.5,2){$\bullet$};
\node at (1.5,3){$\bullet$};
\node at (-4,4.5){$1$};
\node at (0,4.5){$3$};
\node at (4,4.5){$2$};
\node at (0,-0.5){$0$};
\path[-,font=\scriptsize]
(0,0) edge (0,1)
(0,1) edge (1,2)
(0,1) edge (-0.5,2)
(0,1) edge (-4,4)
(1,2) edge (1.5,3)
(1,2) edge (4,4)
(1,2) edge (0,4);
\end{tikzpicture}
\]
Termini are important for studying operads with 0-ary operations---or in biology, extinctions, where a species dies out.

We can also have $n$-trees with vertices of arity 1:

\[
\begin{tikzpicture}[scale=0.5]
\node (1) at (0,1.6){$\bullet$};
\node at (0,-0.5){$0$};
\node at (0,3.5){$1$};
\path[-,font=\scriptsize]
(0,0) edge (0,3);
\end{tikzpicture} \qquad \qquad \quad
\begin{tikzpicture}[scale=0.5]
\node (1) at (0,0.8){$\bullet$};
\node (2) at (0,1.6){$\bullet$};
\node at (-1,3.5){$2$};
\node at (1,3.5){$1$};
\node at (0,-0.5){$0$};
\path[-,font=\scriptsize]
(0,0) edge (0,0.8)
(0,0.6) edge (0,1.6)
(0,1.6) edge (-1,3)
(0,1.6) edge (1,3);
\end{tikzpicture} \qquad \qquad \quad
\begin{tikzpicture}[scale=0.5]
\node (1) at (0,1){$\bullet$};
\node (2) at (1,2){$\bullet$};
\node at (0.5,1.5){$\bullet$};
\node at (-1,2){$\bullet$};
\node at (-2,3.5){$3$};
\node at (0,3.5){$1$};
\node at (2,3.5){$2$};
\node at (0,-0.5){$0$};
\path[-,font=\scriptsize]
(0,0) edge (0,1)
(0,1) edge (1,2)
(0,1) edge (-2,3)
(1,2) edge (2,3)
(1,2) edge (0,3);
\end{tikzpicture} 
\]
Vertices with arity 1 are important for describing operads with 1-ary operations.  In biology it is not very interesting to think about a species that splits into just one species.  However,  we will use an operad with one 1-ary operation for each number $t \in [0,\infty)$ to equip phylogenetic trees with lengths for edges.  So, we will need to think about $n$-trees with 1-ary vertices.

Finally, we warn the reader that there exist 0-trees---that is, trees with no leaves:
\[
\begin{tikzpicture}[scale=0.5]
\node (2) at (0,1){$\bullet$};
\node at (0,-0.5){$0$};
\path[-,font=\scriptsize]
(0,0) edge (0,1);
\end{tikzpicture} 
\qquad \qquad
\begin{tikzpicture}[scale=0.5]
\node at (0,1){$\bullet$};
\node at (0,2.2){$\bullet$};
\node at (0,-0.5){$0$};
\path[-,font=\scriptsize]
(0,0) edge (0,1)
(0,1) edge (0,2.2);
\end{tikzpicture}
\qquad
\begin{tikzpicture}[scale=0.5]
\node at (0,.7){$\cdots$};
\node at (0,-1){};
\end{tikzpicture}
\]
since we interpret the set $\{1,\dots, n\}$ to be the empty set when $n = 0$.  There is also a tree with no vertices, which is a 1-tree:
\[
\begin{tikzpicture}[scale=0.5]
\node (2) at (0,1.5){$1$};
\node at (0,-0.5){$0$};
\path[-,font=\scriptsize]
(0,0) edge (0,1);
\end{tikzpicture} 
\]

Moerdijk and Weiss \cite{MW} have described a category $\Omega$ whose objects are essentially the same as $n$-trees (for arbitrary $n$), and this has been further developed by Weber \cite{Web}.  For this the authors need to define morphisms between $n$-trees. We shall only need isomorphisms, which are easier to define:

\begin{defn}\label{defn:iso n-tree}
An \define{isomorphism of $n$-trees}
$f \maps (V,E,s,t) \to (V',E',s',t')$ consists of:
\begin{itemize}
\item a bijection $f_0 \maps V \sqcup\{0, 1,\dots, n\} \to V' \sqcup \{0,1,\dots, n\}$,
\item a bijection $f_1 \maps E \to E'$
\end{itemize}
such that
\begin{itemize} 
\item $f_0$ is the identity on $\{0,1,\dots, n\}$,
\item $f_0 s = s' f_1$,
\item $f_0 t = t' f_1$.
\end{itemize}
\end{defn}
\noindent
In simple terms, two $n$-trees are isomorphic if one is obtained from the other by renaming the vertices and edges.

\begin{defn}
We call an $n$-tree with just one vertex a \define{corolla}.
\end{defn}
\noindent
For each $n\ge 0$ there is, up to isomorphism, a unique $n$-tree that is a corolla:
\[
\begin{tikzpicture}[scale=0.5]
\node (2) at (0,1){$\bullet$};
\node at (0,-0.5){$0$};
\path[-,font=\scriptsize]
(0,0) edge (0,1);
\end{tikzpicture}
\qquad\;\;
\begin{tikzpicture}[scale=0.5]
\node (2) at (0,1){$\bullet$};
\node at (0,2.5){$1$};
\node at (0,-0.5){$0$};
\path[-,font=\scriptsize]
(0,0) edge (0,1)
(0,1) edge (0,2);
\end{tikzpicture}
\qquad\;
\begin{tikzpicture}[scale=0.5]
\node (2) at (0,1){$\bullet$};
\node at (-0.5,2.5){$1$};
\node at (0.5,2.5){$2$};
\node at (0,-0.5){$0$};
\path[-,font=\scriptsize]
(-0.5,2) edge (0,1)
(0.5,2) edge (0,1)
(0,1) edge (0,0);
\end{tikzpicture}
\qquad
\begin{tikzpicture}[scale=0.5]
\node (2) at (0,1){$\bullet$};
\node at (-1,2.5){$1$};
\node at (0,2.5){$2$};
\node at (1,2.5){$3$};
\node at (0,-0.5){$0$};
\path[-,font=\scriptsize]
(-1,2) edge (0,1)
(0,2) edge (0,1)
(1,2) edge (0,1)
(0,1) edge (0,0);
\end{tikzpicture}
\qquad
\begin{tikzpicture}[scale=0.5]
\node at (0,.7){$\cdots$};
\node at (0,-1){};
\end{tikzpicture}
\]

\begin{defn}
A \define{planar} $n$-tree is an $n$-tree in which each vertex is equipped with a linear order on the set of its children.   A \define{planar tree} is a planar $n$-tree for any $n = 0, 1, 2, \dots$.
\end{defn}

\begin{defn} 
An \define{isomorphism of planar $n$-trees} is an isomorphism of $n$-trees $f \maps (V,E,s,t) \to (V',E',s',t')$ that preserves this linear ordering on the children of each vertex. 
\end{defn}

\noindent
We can draw any planar $n$-tree on the plane in such a way that the children of a vertex are listed in increasing order from left to right.  With this convention, the following two planar 3-trees are not isomorphic, even though they are isomorphic as 3-trees:
\[
\begin{tikzpicture}[scale=0.5]
\node (1) at (0,1){$\bullet$};
\node (2) at (1,2){$\bullet$};
\node at (-2,3.5){$1$};
\node at (0,3.5){$2$};
\node at (2,3.5){$3$};
\node at (0,-0.5){$0$};
\path[-,font=\scriptsize]
(0,0) edge (0,1)
(0,1) edge (1,2)
(0,1) edge (-2,3)
(1,2) edge (2,3)
(1,2) edge (0,3);
\end{tikzpicture} 
\qquad \qquad
\begin{tikzpicture}[scale=0.5]
\node (1) at (0,1){$\bullet$};
\node (2) at (1,2){$\bullet$};
\node at (-2,3.5){$1$};
\node at (0,3.5){$3$};
\node at (2,3.5){$2$};
\node at (0,-0.5){$0$};
\path[-,font=\scriptsize]
(0,0) edge (0,1)
(0,1) edge (1,2)
(0,1) edge (-2,3)
(1,2) edge (2,3)
(1,2) edge (0,3);
\end{tikzpicture} 
\]

We are now ready to define a phylogenetic tree:

\begin{defn}
An $n$-tree together with a map $\ell \maps E\to [0,\infty)$ is called an \define{$n$-tree with lengths}.  For any $e\in E$ we call $\ell(e)$ the \define{length} of $e$.
\end{defn}

\begin{defn}
A \define{phylogenetic $n$-tree} is an isomorphism class of $n$-trees with lengths obeying these rules:
\begin{enumerate}
\item the length of every edge is positive, except perhaps for edges incident to a leaf or the root;
\item there are no 0-ary or 1-ary vertices.
\end{enumerate}
A \define{phylogenetic tree} is a phylogenetic $n$-tree for some $n \ge 1$.
\end{defn}

\noindent
In this definition we require that there are no $0$-ary vertices, or in other words, there are no extinctions.  This restriction may seem odd, but it reflects common practice: biologists often use phylogenetic trees to describe evolutionary relationships between currently existing species, ignoring extinct species \cite{B}.  Furthermore, the space of phylogenetic $n$-trees is finite-dimensional with this restriction, but infinite-dimensional without it, since without it an $n$-tree could have arbitrarily many edges labelled by lengths.

Taking isomorphism classes of $n$-trees (see Definition \ref{defn:iso n-tree}) means that the names of the vertices and edges are irrelevant. So, for example, this is a phylogenetic tree:
\[
\begin{tikzpicture}[scale = 0.7]
\node at (-0.5,0){$\bullet$};
\node at (1,-1){$\bullet$};
\node at (-1.5,1.3){$2$};
\node at (-1.4,0.4){$0$};
\node at (0.8,1.3){$1$};
\node at (0.5,0.4){$0.2$};
\node at (2.5,1.3){$3$};
\node at (2.4,0){$2.7$};
\node at (-0.3,-0.7){$1.4$};
\node at (1.5,-1.6){$1.3$};
\node at (1,-2.5){$0$};                   
\path[-,font=\scriptsize]
(-1.5,1) edge (-0.5,0)
(0.5,1) edge (-0.5,0)
(-0.5,0) edge (1,-1)
(2.5,1) edge (1,-1)
(1,-1) edge (1,-2);
\end{tikzpicture}
\]

This tree with lengths is not a phylogenetic tree, because it violates rule 1:
\[
\begin{tikzpicture}[scale = 0.7]
\node at (-0.5,0){$\bullet$};
\node at (1,-1){$\bullet$};
\node at (-1.5,1.4){$2$};
\node at (-1.5,0.5){$0$};
\node at (0.8,1.4){$1$};
\node at (0.6,0.5){$0.2$};
\node at (2.5,1.4){$3$};
\node at (2.5,0){$2.7$};
\node at (-0.1,-0.7){$0$};
\node at (1.5,-1.6){$1.3$};
\node at (1,-2.5){$0$};                   
\path[-,font=\scriptsize]
(-1.5,1) edge (-0.5,0)
(0.5,1) edge (-0.5,0)
(-0.5,0) edge (1,-1)
(2.5,1) edge (1,-1)
(1,-1) edge (1,-2);
\end{tikzpicture}
\]
In terms of biology, the idea is that if a species splits and then \emph{immediately} splits again, we would describe this using a single ternary vertex instead of two binary ones.

The following tree with lengths is not a phylogenetic tree, because it has a $0$-ary vertex:
\[
\begin{tikzpicture}[scale = 0.7]
\node at (-0.5,0){$\bullet$};
\node at (1,-1){$\bullet$};
\node at (-1.5,1.4){$2$};
\node at (-1.5,0.5){$0$};
\node at (0.8,1.4){$1$};
\node at (0.6,0.5){$0.2$};
\node at (2.2,0.4){$\bullet$};
\node at (2.1,-0.4){$2.7$};
\node at (-0.3,-0.7){$1.4$};
\node at (1.5,-1.6){$1.3$};
\node at (1,-2.5){$0$};                   
\path[-,font=\scriptsize]
(-1.5,1) edge (-0.5,0)
(0.5,1) edge (-0.5,0)
(-0.5,0) edge (1,-1)
(2.2,0.4) edge (1,-1)
(1,-1) edge (1,-2);
\end{tikzpicture}
\]
In terms of biology, the terminus here describes an extinction.  We might allow trees with termini if we were building phylogenetric trees using DNA from extinct species.  Doing this would simply require that we replace $\Com$ with $\Com_+$, the operad with a single $n$-ary operation for each $n \ge 0$.   The 0-ary operation in $\Com_+$ would describe an extinction.

Finally, this tree with lengths is not a phylogenetic tree since it has a $1$-ary vertex:
\[
\begin{tikzpicture}[scale = 0.7]
\node at (-0.5,0){$\bullet$};
\node at (1,-1){$\bullet$};
\node at (-0.5,1.4){$2$};
\node at (-1,0.5){$0.3$};
\node at (2.5,1.4){$1$};
\node at (2.5,0){$2.7$};
\node at (-0.3,-0.7){$1.4$};
\node at (1.5,-1.6){$1.3$};
\node at (1,-2.5){$0$};                   
\path[-,font=\scriptsize]
(-0.5,1) edge (-0.5,0)
(-0.5,0) edge (1,-1)
(2.5,1) edge (1,-1)
(1,-1) edge (1,-2);
\end{tikzpicture}
\]
In terms of biology, we do not want to discuss the process of a species splitting into  \emph{just one} species.  In particular, this is not a phylogenetic tree:
\[
\begin{tikzpicture}
\node at (0,1){$\bullet$};
\node at (0,2.4){$1$};
\node at (0.2,1.5){$0$};
\node at (0.2,0.5){$0$};
\node at (0,-0.4){$0$};
\path[-,font=\scriptsize]
(0,0) edge (0,1)
(0,1) edge (0,2);
\end{tikzpicture}
\]
but this is:
\[
\begin{tikzpicture}
\node (2) at (0,1.4){$1$};
\node at (0.2,0.5){$0$};
\node at (0,-0.4){$0$};
\path[-,font=\scriptsize]
(0,0) edge (0,1);
\end{tikzpicture} 
\]
and we shall need it, to serve as the identity operation in the phylogenetic operad.

We next build the phylogenetic operad as the coproduct of two operads already discussed in Section \ref{sec:intro}.  The first is $\Com$.  This is the unique operad with one $n$-ary operation $f_n$ for each $n > 0$ and no 0-ary operations.  The second is $[0,\infty)$.  This is the unique operad with only unary operations whose space of unary operations is the set of nonnegative real numbers, topologized in the standard way, with composition defined to be addition. 

\begin{defn} 
The \define{phylogenetic operad} $\Phyl$ is the coproduct $\Com + [0,\infty)$.
\end{defn}

Our first main result, Theorem \ref{thm:bijection}, gives an explicit description of the
operations of $\Phyl$:

\begin{theorem} 
\label{thm:bijection}
The $n$-ary operations in the phylogenetic operad are in 
one-to-one correspondence with phylogenetic $n$-trees.
\end{theorem}

The statement here is somewhat inadequate, since we really have a specific bijection in mind.
The task of making this theorem precise and proving it occupies Sections \ref{sec:free}--\ref{sec:coproduct_unary}.   In Section \ref{sec:free} we give a description of the operations in a free operad.  In Section \ref{sec:coproduct} we use this to describe operations in a coproduct of operads.  Finally, in Section \ref{sec:coproduct_unary} we give a simpler description of the operations in a coproduct of operads $O+O'$ when $O'$ has only unary operations.   We show that the $n$-ary operations of $O+O'$ are certain equivalence classes of planar rooted trees having $n$ leaves, with edges labelled by the unary operations of $O'$ and $k$-ary vertices labelled by $k$-ary operations in $O$.   We state this fact more precisely and prove it in Lemma \ref{thm:generalized_phyl_trees}.

Applying this lemma to the coproduct $\Com+[0,\infty)$, we see that the $n$-ary operations in this operad are equivalence classes of planar rooted trees having $n$ leaves, with edges labelled by numbers in $[0,\infty)$ but with unlabelled vertices, since there is a unique $k$-ary operation in 
$\Com$ for each $k \ge 1$.  This gives Theorem \ref{thm:bijection}, which we prove
at the very end of Section \ref{sec:coproduct_unary}.

Henceforth we shall use the bijection given in the proof of Theorem \ref{thm:bijection} to identify $n$-ary operations of $\Phyl = \Com + [0,\infty)$ with phylogenetic $n$-trees.    Since $\Com + [0,\infty)$ is a topological operad, this bijection puts a topology on the set of phylogenetic $n$-trees.  From now on, we freely use $\Phyl_n$ to mean either the set of phylogenetic $n$-trees with this topology or the space of $n$-ary operations of $\Com + [0,\infty)$.  

In their work on phylogenetic trees, Billera, Holmes and Vogtmann \cite{BHV} studied a space closely related to $\Phyl_n$, which they call the space of `metric $n$-trees'.    They give a definition equivalent to this one:

\begin{defn} \label{defn:metric tree}
A \define{metric $n$-tree} is an isomorphism class of $n$-trees with lengths obeying
these rules:
\begin{enumerate}
\item the length of every internal edge is positive;
\item the length of every external edge is zero;
\item there are no 0-ary or 1-ary vertices.
\end{enumerate}
We denote the set of metric $n$-trees by \define{$\T_n$}.
\end{defn}

We note that by the last item in the definition of a metric $n$-tree there are no $0$-trees, and there is exactly one $1$-tree, namely the trivial tree with its unique edge labelled by zero.

Billera, Holmes and Vogtmann do not label the external edges with lengths.  But this is equivalent to labelling them all with length zero.  More importantly, these authors give the space $\T_n$ a metric. To do this, they show that $\T_n$ may be constructed by gluing standard Euclidean orthants together.   They then define the distance between two points in the same orthant as the Euclidean distance, and the distance between two points in two different orthants as the minimum of the lengths of all paths between them that consist of finitely many straight line segments in orthants.  This metric makes $\T_n$ into a space with well-behaved geodesics, called a $\CAT(0)$-space \cite[Sec.\ II.1]{BH}. 

The space $\T_n$ is contractible; we get a contracting homotopy by rescaling the lengths of all internal edges in a way that sends them to zero.  Nonetheless, its topology is very interesting.  For example, Billera, Holmes and Vogtmann note that $\T_4$ is the cone on the Petersen graph:
\begin{center}
\includegraphics[scale = 0.2]{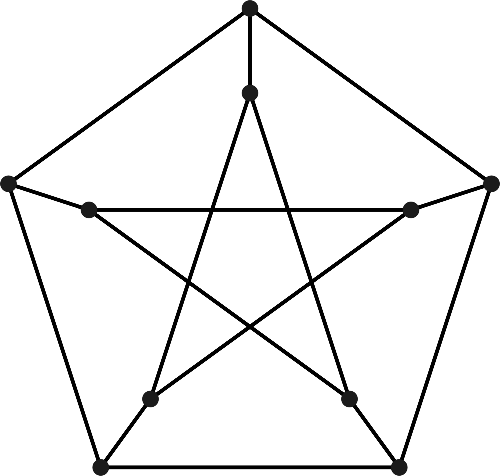}
\end{center}
\noindent
The 15 edges of the Petersen graph correspond to the 15 binary trees with 
four labelled leaves.   The cone on any edge of the Petersen graph is a quadrant,
where the two coordinates are the lengths of the two internal edges of the given
binary tree:
\[
\hspace*{-10em}
\raisebox{0em}{
\begin{tikzpicture}
\node at (0,-0.02){$\bullet$};
\node at (0,5){$\bullet$};
\node at (5,-0.02){$\bullet$};
\node at (5,5){$\bullet$};
\node at (-0.8,5){$(0,1)$};
\node at (4.3,5){$(1,1)$};
\node at (5,-0.5){$(1,0)$};
\node at (-0.5,-0.5){$(0,0)$};
\draw [ultra thick] (0,0) -- (0,6);
\draw [ultra thick] (6,0) -- (0,0);
\draw [blue, ->] (1,0.9) -- (0.2,0.2);
\draw [blue, ->] (1,5.2) -- (0.2,5.05);
\draw [blue, ->] (6,5.2) -- (5.2,5.05);
\draw [blue, ->] (6,0.9) -- (5.2,0.2);
\end{tikzpicture}
}
\hspace*{0em}
\raisebox{14em}{
\scalebox{0.6}{
\begin{tikzpicture}[scale=0.5]
\node (1) at (0,1){$\bullet$};
\node (2) at (2,2){$\bullet$};
\node (3) at (-2,2){$\bullet$}; 
\node at (0,-1.5){$0$};
\node at (-3,4.5){$1$};
\node at (-1,4.5){$2$};
\node at (1,4.5){$3$};
\node at (3,4.5){$4$};
\node at (1.3,1.2){$1$};
\node at (-1.3,1.2){$1$};
\path[-,font=\scriptsize]
(0,-1) edge (0,1)
(0,1) edge (2,2)
(0,1) edge (-2,2)
(-2,2) edge (-3,4)
(-2,2) edge (-1,4)
(2,2) edge (1,4)
(2,2) edge (3,4);
\end{tikzpicture}
}
}
\hspace*{-8em}
\raisebox{3em}{
\scalebox{0.6}{
\begin{tikzpicture}[scale=0.5]
\node (1) at (0,1){$\bullet$};
\node (3) at (-2,2){$\bullet$}; 
\node at (0,-1.5){$0$};
\node at (-3,4.5){$1$};
\node at (-1,4.5){$2$};
\node at (1.5,4.5){$3$};
\node at (3.5,4.5){$4$};
\node at (-1.3,1.2){$1$};
\path[-,font=\scriptsize]
(0,-1) edge (0,1)
(0,1) edge (-2,2)
(-2,2) edge (-3,4)
(-2,2) edge (-1,4)
(0,1) edge (1.5,4)
(0,1) edge (3.5,4);
\end{tikzpicture}
}
}
\hspace*{-21em}
\raisebox{14em}{
\scalebox{0.6}{
\begin{tikzpicture}[scale=0.5]
\node (1) at (0,1){$\bullet$};
\node (3) at (2,2){$\bullet$}; 
\node at (0,-1.5){$0$};
\node at (3,4.5){$4$};
\node at (1,4.5){$3$};
\node at (-1.5,4.5){$2$};
\node at (-3.5,4.5){$1$};
\node at (1.3,1.2){$1$};
\path[-,font=\scriptsize]
(0,-1) edge (0,1)
(0,1) edge (2,2)
(2,2) edge (3,4)
(2,2) edge (1,4)
(0,1) edge (-1.5,4)
(0,1) edge (-3.5,4);
\end{tikzpicture}
}
}
\hspace*{-8em}
\raisebox{3em}{
\scalebox{0.6}{
\begin{tikzpicture}[scale=0.5]
\node (1) at (0,1){$\bullet$};
\node at (0,-1.5){$0$};
\node at (3,4.5){$4$};
\node at (1,4.5){$3$};
\node at (-1.5,4.5){$2$};
\node at (-3.5,4.5){$1$};
\path[-,font=\scriptsize]
(0,-1) edge (0,1)
(0,1) edge (0,1)
(0,1) edge (3,4)
(0,1) edge (1,4)
(0,1) edge (-1.5,4)
(0,1) edge (-3.5,4);
\end{tikzpicture}
}
}
\]
At the boundary of such a quadrant we have phylogenetic trees where the length of one internal edge equals zero; such trees are no longer binary.  In the Petersen graph, three edges meet at each vertex. Thus, the cone on any vertex of the Petersen graph is a ray at which three quadrants of the above form meet along their boundaries.  Similarly, the cone on any pentagon in the Petersen graph consists of five quadrants glued together along their boundaries.  The  corresponding binary trees are those appearing in the famous Stasheff pentagon \cite{Sta}.  Billera, Holmes and Vogtmann explain how to generalize this to any $n$, obtaining a relation between the spaces $\T_n$ and associahedra.

Our second main result relates $\Phyl_n$ to $\T_n$. 

\begin{theorem} 
\label{thm:homeo}
For every $n \ne 1$ there is a homeomorphism 
\[           \Phyl_n \cong \T_n  \times [0,\infty)^{n+1}  ,\]
and $\Phyl_1 \cong \T_1 \times [0,\infty)$.
\end{theorem}  

If we compare the definition of a phylogenetic $n$-tree to the definition of a metric $n$-tree, we see that the theorem holds vacuously for $n=0$, since there are no metric $0$-trees and no phylogenetic $0$-trees.  A phylogenetic $1$-tree has just one edge, which is labelled by a non-negative real number, while there is only one metric $0$-tree, namely the tree with one edge labelled by zero. Thus there is a bijection between  $\Phyl_1$ and   $\T_1  \times [0,\infty)$. More interestingly, for $n > 1$ a phylogenetic tree gives a metric $n$-tree together with an $(n+1)$-tuple of numbers in $[0,\infty)$, namely the lengths labelling the  external edges of the phylogenetic tree.  On the other hand, a metric $n$-tree together with an $(n+1)$-tuple of numbers in $[0,\infty)$ gives a phylogenetic tree with these numbers labelling its external edges.  We thus have a  bijection between $\Phyl_n$ and  $\T_n  \times [0,\infty)^{n+1}$. 

Checking that this bijection is a homeomorphism requires further work.  The operad $\Phyl$ is defined as a coproduct of operads, so the topology on $\Phyl_n$ is determined rather implicitly by  the universal property of the coproduct.   We describe an explicit basis for the topology on $\Phyl_n$ in  Appendix \ref{A:topology of phyl}, and use this to prove Theorem \ref{thm:homeo}.

\section{Branching Markov processes}
\label{sec:branching}

In Section \ref{sec:intro} we sketched how a Markov process on a finite set gives a coalgebra of the phylogenetic operad.   Here we give the details.   We also prove that any coalgebra of $\Phyl = \Com + [0,\infty)$ arising this way extends uniquely to a coalgebra of a larger operad $\Com + [0,\infty]$.  This expresses the fact that Markov processes on finite sets always converge as time approaches infinity.

We begin with a general definition of algebras and coalgebras for an operad $O$.
Let $\Top$ be a convenient category of topological spaces, such as compactly
generated Hausdorff spaces, and suppose $C$ is a symmetric monoidal category enriched over $\Top$.  Then for any object $V \in C$ there is an operad $\End(V)$, the \define{endomorphism operad} of $V$, with
\[           \End(V)_n = \hom_C(V^{\otimes n}, V) .\]
An \define{algebra} of $O$ in $C$ is an operad morphism $\alpha \maps O \to \End(V)$.  In other words, $\alpha$ sends operations $f \in O_n$ to maps 
\[    \alpha(f) \maps V^{\otimes n} \to V \]
in a continuous manner, preserving composition, the identity, and the permutation group actions.  

A \define{coalgebra} of $O$ in $C$ is an algebra of $O$ in the opposite category, $C^{\op}$.   Equivalently, it is an operad morphism from $O$ to the \define{coendomorphism operad} $\Coend(V)$, where
\[        \Coend(V)_n = \hom_C(V, V^{\otimes n})  \]
Given a coalgebra of $O$, any operation $f \in O_n$ is mapped to a morphism
\[           \alpha(f) \maps V \to V^{\otimes n} . \]
We say $\alpha(f)$ \define{coacts} on $V$.

In this section we only need coalgebras in $\Fin\Vect$, the category of finite-dimensional
real vector spaces and linear maps, made into a symmetric monoidal category
with its usual tensor product, and enriched over $\Top$ in the usual way.  So, by `coalgebra', we shall mean one in $\Fin\Vect$.  

Recall from Section \ref{sec:intro} that a coalgebra of the operad $[0,\infty)$ is known to analysts as a \define{continuous 1-parameter semigroup}.   Concretely, such a coalgebra consists of
a finite-dimensional real vector space $V$ together with linear maps
\[             \alpha(t) \maps V \to V  \]
for each $t \in [0,\infty)$, such that:
\begin{itemize}
\item $\alpha(s+t) = \alpha(s) \alpha(t)$ for all $s,t \in [0,\infty)$,
\item $\alpha(0) = 1_V$,
\item $\alpha(t)$ depends continuously on $t$.
\end{itemize}
When $V = \R^X$ for a finite set we have the following result:

\begin{theorem} 
\label{thm:coalgebra_of_phyl}
Given a finite set $X$ and a continuous 1-parameter semigroup $\alpha \maps [0,\infty) \to \End(\R^X)$, there is a unique way of making
$\R^X$ into a coalgebra of $\Phyl =  \Com+ [0,\infty)$ such that:
\begin{enumerate}
\item Each unary operation $t \in [0,\infty)$ coacts on $\R^X$ as
$\alpha(t) \maps \R^X \to \R^X$.
\item  The unique binary operation in $\Com$ coacts on $\R^X$ as the linear
map
\[  \Delta \maps  \R^X \to \R^X  \tensor \R^X  \]
where
\[     \Delta (f)(x_1, x_2) = 
\left\{ \begin{array}{cl}  f(x) & \textrm{if } x_1 = x_2 = x  \\ \\
                                       0   & \textrm{otherwise.}   
\end{array} \right.   \]
\end{enumerate}
\end{theorem}

\begin{proof} 
Because all $n$-ary operations in $\Com$ for $n > 1$
are composites of the unique binary operation,
item (2) forces the unique $n$-ary operation to coact as the linear map
\[  \Delta_n   \maps  \R^X \to \R^X  \tensor \cdots  \tensor \R^X  \iso \R^{X^n} \]
where
\[     \Delta_n (f)(x_1, \dots, x_n) = 
\left\{ \begin{array}{cl}  f(x) & \textrm{if } x_1 = \cdots = x_n = x  \\ \\
                                       0   & \textrm{otherwise}   
\end{array} \right.   \]
It is easy to check that this formula indeed makes $\R^X$ into a coalgebra of
$\Com$.  It also becomes a coalgebra of the operad $[0,\infty)$ via item (1).
By the universal property of the coproduct of operads, these coalgebra structures uniquely determine a way of making $\R^X$ into a coalgebra of the coproduct $\Com + [0,\infty)$.
\end{proof}

Among continuous 1-parameter semigroups on $\R^X$, Markov processes have a special property: they always approach an equilibrium.  More precisely:

\begin{lemma}
\label{lem:convergence}
If $X$ is a finite set and $\alpha \maps [0,\infty) \to \End(\R^X)$ is a Markov process, 
the operators $\alpha(t) \in \End(\R^X)$ converge to a fixed operator $P \in \End(\R^X)$ as $t \to \infty$.   Moreover, $P^2 = P$.
\end{lemma}

\begin{proof} 
This result should be well-known in the theory of Markov processes, but since we were unable to find an easy reference we include a proof which also serves as a quick introduction to Markov processes.

Suppose $\alpha \maps [0,\infty) \to \End(\R^X)$ is a Markov process.   Because $\alpha$ is a continuous one-parameter semigroup of operators on $\R^X$, it follows \cite[Thm.\ 2.9]{EN} that $\alpha(t)$ is differentiable, and if we set
\[          H = \left. \frac{d}{dt} \alpha(t) \right|_{t = 0} \]
then $\alpha(t) = \exp(tH)$.  

Since $\alpha(t)$ is stochastic for all $t  \ge 0$, its matrix entries must be nonnegative in the standard basis of $\R^X$, but its off-diagonal entries vanish at $t = 0$ since $\alpha(0) = 1$.  Thus, the off-diagonal entries of $H = \left. \frac{d}{dt} \alpha(t) \right|_{t = 0}$ must be nonnegative.  We can define a directed graph $\Gamma$ with $X$ as its set of nodes and an edge from $i \in X$ to $j \in X$ if and only if $H_{ji} > 0$.   We begin by assuming that this graph is \define{strongly connected}, meaning that there is a directed edge path from $i$ to $j$ for all $i, j \in X$.  In this case $H$ is \define{irreducible}: there is no way to bring the matrix $(H_{ij})$ into a block upper triangular form by permuting its rows and columns \cite[Sec.\ 8.1]{BC}. 

Thus, for sufficiently large $c > 0$, $H + c I$ will be irreducible and also have nonnegative matrix entries.  As a result, the Perron--Frobenius theorem applies \cite[Sec.\ 8.3]{BC}. This says that the matrix $H + c I$ has a real eigenvalue that is greater than the real part of all other eigenvalues, and a positive eigenvector $f \maps X \to (0,\infty)$ with this eigenvalue.   It follows that the same holds for $H$.  

Suppose $\lambda$ is the real eigenvalue of $H$ that is greater than the real part of all other eigenvalues, and let $f \maps X \to (0,\infty)$ be a function with 
$H f= \lambda f$.   We may assume $f$ is normalized to be a probability distribution.  Since $\exp(tH)$ is stochastic for all $t \ge 0$, 
\[   \exp(tH)f = \exp(t\lambda) f  \]
must also be a probability distribution for all $t \ge 0$.  It follows that $\lambda = 0$.  In particular, $H$ has zero as an eigenvalue.  Moreover, if we regard $H$ as a special case of an $n \times n$ matrix of complex numbers, then all the other---possibly complex---eigenvalues $\lambda_i$ of $H$ have $\mathrm{Re}(\lambda_i) < 0$. 

More generally, suppose the graph $\Gamma$ is not strongly connected.  Then we can partition
$X$ into \define{strongly connected components}: that is subsets $\{S_k\}_{k \in \Lambda}$ such that the restriction of $\Gamma$ to each subset is strongly 
connected.  Moreover, these strongly connected components are partially ordered
where $k \preccurlyeq \ell$ if and only if there exists a directed edge path from a node in $S_k$ to a node in $S_\ell$.  We can choose a linear ordering $\le$ for $\Lambda$ such that
$k \preccurlyeq \ell$ implies $k \le \ell$.  Thus, we can order the standard basis
for $\R^X$ in such a way that $H$ becomes a block upper triangular matrix,
with blocks corresponding to strongly connected components. 

By our analysis of the strongly connected case, each diagonal block of $H$ must have
zero as an eigenvalue, with all other eigenvalues having negative
real part.  Since $H$ is block upper triangular, it follows that the only possible eigenvalues of $H$, including complex eigenvalues, are zero and numbers with negative real part.  Using the Jordan normal form, it follows that for some invertible linear transformation $Q$, $QHQ^{-1}$ is a block diagonal sum of Jordan blocks:
\[  
\left(\begin{array}{cccccc}
\lambda & 1 & 0 & \cdots & 0 & 0  \\
0 & \lambda & 1 & \cdots & 0 & 0 \\
0 & 0 & \lambda & \cdots & 0 & 0 \\
\vdots & \ddots & \ddots  & \ddots & \vdots   &  \vdots \\
0 & 0 & 0 & \cdots & \lambda & 1 \\
0 & 0 & 0 & \cdots & 0 & \lambda
 \end{array} \right)
 \]
where $\lambda = 0$ or $\mathrm{Re}(\lambda) < 0$.  Exponentiating, we see that $Q\exp(tH)Q^{-1}$ is a block diagonal sum of square matrices of this form:
\[  
\left(\begin{array}{cccccc}
e^{t\lambda} & te^{t\lambda} &\displaystyle{ \frac{t^2}{2!} e^{t\lambda}} & \cdots & \displaystyle{ \frac{t^{k-1}}{(k-1)!} e^{t\lambda}}  \\
0 & e^{t\lambda} & te^{t\lambda} & \cdots &\displaystyle{ \frac{t^{k-2}}{(k-2)!} e^{t\lambda} }  \\
\vdots &\ddots & \ddots &\ddots &\vdots \\
0 & 0 & 0 & \cdots  & te^{t\lambda} \\
0 & 0 & 0 & \cdots  & e^{t\lambda}
 \end{array} \right)
 \]
As $t \to +\infty$, the above matrix converges to the identity if $\lambda = 0$ and to zero if $\mathrm{Re}(\lambda) < 0$.  Thus, it converges to an idempotent.  As a consequence, $\exp(tH)$ converges to an idempotent $P \in \End(\R^X)$ as $t \to +\infty$.
\end{proof}

We can make $[0,\infty]$ into a commutative monoid using addition, where we define
$\infty + t = t + \infty = \infty$ for all $t \in [0,\infty]$.  The set $[0,\infty]$ has a 
topology where it is homeomorphic to a closed interval, e.g.\ by requiring that
$\tan \maps [0,\pi/2] \to [0,\infty]$ is a homeomorphism.    With this topology $[0,\infty]$ becomes a topological monoid, and thus an operad with only unary operations.

Since $[0,\infty)$ is a suboperad of $[0,\infty]$ in an obvious way, the phylogenetic operad $\Phyl = \Com + [0,\infty)$ becomes a suboperad of $\Com + [0,\infty]$, thanks to Corollary \ref{cor:suboperad}.

\begin{theorem}
\label{thm:extension}
If $X$ is a finite set and $\alpha \maps [0,\infty) \to \End(\R^X)$ is a Markov process, 
$\R^X$ becomes a coalgebra of $\Com + [0,\infty]$ in a unique way extending its
structure as a coalgebra of $\Phyl = \Com + [0,\infty)$ described in Theorem \ref{thm:coalgebra_of_phyl}.
\end{theorem}

\begin{proof}
Thanks to the universal property of the coproduct, we only need to prove that $\R^X$ becomes a coalgebra of $[0,\infty]$ in unique way extending its structure as a coalgebra 
of $[0,\infty)$, where $t \in [0,\infty)$ acts as $\alpha(t)$.  Uniqueness follows from the continuity and fact that $[0,\infty)$ is dense in $[0,\infty]$: this forces us to take
\[       \alpha(\infty) = \lim_{t \to \infty} \alpha(t)  .\]
For existence, we first use Lemma \ref{lem:convergence} to note that the limit exists.  Then, to note that $\alpha \maps [0,\infty] \to \End(\R^X)$ thus defined is really a coalgebra
action, we note that for any $s \in [0,\infty)$
\[       \alpha(s) \alpha(\infty) = \alpha(s)  \lim_{t \to \infty} \alpha(t) =
 \lim_{t \to \infty} \alpha(s+t) = \alpha(\infty) \]
as required, and similarly $\alpha(\infty)\alpha(s) = \alpha(\infty)$, and also
\[   \alpha(\infty) \alpha(\infty) = 
\lim_{t \to \infty} \alpha(s)  \lim_{t \to \infty} \alpha(t) =
 \lim_{s,t \to \infty} \alpha(s+t) = \alpha(\infty).  \qedhere \]
\end{proof}

Combining this result and Lemma \ref{lem:convergence}, one
sees that $P = \alpha(\infty)$ is an idempotent ($P^2 = P$) that maps the set of probability distributions on $X$ onto the set of \define{equilibrium} probability
distributions, meaning those that are invariant under the time evolution given by 
$\alpha(t)$.

\section{Free operads}
\label{sec:free}

Theorem \ref{thm:bijection} claims that there is a bijection between phylogenetic trees and operations in $\Com + [0,\infty)$.   Constructing this bijection takes some work.   We need an explicit description of the operations in a  coproduct of operads---and for that, we need a description of the operations in a free operad.  We work out the details in the next three sections.

Readers who are eager to read about the relationship between the phylogenetic operad and the  $\wco$ construction can go directly to Section \ref{section:w}.  We have tried to make that section readable on its own, though logically it depends on all the material that comes before. 

 
In what follows, we use `operad' to mean a symmetric operad in the symmetric monoidal category  $\Top$.    Thus, our definition matches that of May \cite{May72} except that we allow more than one operation of arity $0$.  We shall show, among other things, that there is an operad $\PTree$ with isomorphism classes of planar $n$-trees as its $n$-ary operations.  Moreover $\PTree$ arises quite naturally from the theory of operads, as we now explain.

Every operad $O$ has an underlying `collection', which is simply the list of
spaces $O_n$, forgetting composition and the permutation group actions:

\begin{defn} 
A \define{collection} $C$ consists of topological spaces $\{C_n\}_{n \ge 0}$.  A \define{morphism} of collections $f \maps C \to C'$ consists of a continuous map $f_n \maps C_n \to C'_n$ for each $n \ge 0$.  
\end{defn}
\noindent
Collections and morphisms between them form a category.  This category is simply $\NTop$,
where $\N$ stands for the set of natural numbers, or the corresponding discrete category.

Let $\Op$ be the category consisting of operads and morphisms between them.  There is a forgetful functor $U \maps \Op \to \NTop$ sending any operad $O$ to the collection $\{O_n\}_{n \ge 0}$.  Moreover:

\begin{lemma}
\label{monadic}
The forgetful functor
\[         U \maps \Op \to \NTop  \]
is \define{monadic}, meaning that it has a left adjoint 
\[         F \maps \NTop  \to \Op  \]
giving rise to a monad $UF \maps \NTop \to \NTop$, and the comparison functor from $\Op$ to the category of algebras of this monad is an equivalence.
\end{lemma}

\begin{proof} This follows from Boardmann and Vogt's work on free algebras for
colored operads, using the fact that operads are themselves the algebras of a colored operad with one color for each arity $n \in \N$. The existence of a left adjoint for $U$ follows from Boardmann and Vogt's Theorem 2.24, and the monadicity of $U$ follows from their Proposition 2.33.  For the colored operad whose algebras are operads, see \cite{BD}.  This operad began life as a $\Set$-based rather than a topological operad, but we can reinterpret it as a topological operad whose spaces of operations  are discrete, and this has the same algebras in $\Top$.
\end{proof}

The operad whose $n$-ary operations are isomorphism
classes of planar $n$-trees has a simple description in terms of this adjunction: 
\[                 \PTree \iso FU(\Com_+)   \]
where $\Com_+$ is the operad whose algebras are commutative topological monoids.  More concretely, $\Com_+$ is the operad, unique up to isomorphism, whose space of $n$-ary operations is a one-element set for each $n \ge 0$.   More abstractly, $\Com_+$ is the terminal operad.   It thus arises naturally in operad theory---and thus, so does the concept of planar tree.  Indeed, the role of planar trees in operad theory well-known \cite[Sec.\ II.1.9]{MSS}, but we
deduce it from a more general statement in Corollary \ref{cor:PTree}.

As usual, the adjunction between operads and collections gives rise to an operad morphism called the \define{unit}
\[          \iota_C \maps C \to UF(C)  \]
for any operad $O$, and a morphism of collections called the \define{counit}\label{defn:counit}
\[          \epsilon_O \maps FU(O) \to O \]
for any collection $C$.  These are natural transformations.  Moreover, thanks to Lemma \ref{monadic}, any operad $O$ can be described as the coequalizer of this diagram:
\[        
\xymatrix{
FUFU(O)  \ar@<1ex>[rr]^{\epsilon_{FU(O)}}  \ar@<-1ex>[rr]_{FU(\epsilon_O)} &&
FU(O) \ar[r]^>>>>>{\epsilon} & O 
}
\]
This will give an explicit description of $O_n$ as a quotient of $FU(O)_n$ by an equivalence relation.   We start by describing the functor $F$.   For any collection $C$, the operations in $F(C)$ will be `$C$-trees':

\begin{defn}
For any collection $C$, a \define{$C$-labelled planar $n$-tree} is a planar $n$-tree for which each vertex with $k$ children is labelled by an element of $C_k$.
\end{defn}

\begin{defn}
\label{defn:isomorphism}
Given two $C$-labelled planar $n$-trees, we say they are \define{isomorphic} if there is an isomorphism of their underlying planar $n$-trees such that the labelling of each vertex in the first equals the labelling of the corresponding vertex in the second.
\end{defn}

\begin{defn}
\label{defn:CTree}
We define a \define{$C$-$n$-tree} to be an isomorphism class of $C$-labelled planar $n$-trees. We define a \define{$C$-tree} to be a $C$-$n$-tree for any $n = 0,1,2, \dots $. We denote the set of $C$-$n$-trees by \define{$C\Tree_n$}.
\end{defn}

To make $C$-trees into the operations of an operad, we must say how to compose
them.  Instead of fully general composition
\[   f \circ (g_1, \dots, g_n)  \]
it suffices to describe \define{partial composition}:
\[  f\circ_i g=f\circ (1,\dots,1, g, 1,\dots ,1) \qquad  f \in O_n, \; 1 \le i \le n \]
where $g$ appears in the $i$th position.  Knowing partial composites we can
recover all composites, and there is an alternative axiomatization for operads, 
equivalent to the standard one, using partial composition \cite[Def.\ 12]{May97}.  

Partial composition of $C$-trees will be defined using `grafting'.  The rough idea is that in the partial composite $T' \circ_i T$, we glue the root of $T$ to the $i$th  leaf of $T'$.  Then we delete the resulting vertex and combine the two edges incident to it into a single edge.  Here a picture is worth a thousand words:
\[
\begin{tikzpicture}[scale=0.5]
\node at (0,1) {$\bullet$};
\node at (-0.6,1) {$f$};
\node at (0,-0.5) {$0$};
\node at (-1,2.5) {$2$};
\node at (1,2.5) {$1$};
\node at (2,1) {$\circ_1$};
\path[-,font=\scriptsize]
(0,0) edge (0,1)
(0,1) edge (1,2)
(0,1) edge (-1,2);
\end{tikzpicture}
\begin{tikzpicture}[scale=0.5]
\node at (0,1) {$\bullet$};
\node at (-0.6,1) {$g$};
\node at (0,-0.5) {$0$};
\node at (-1,2.5) {$1$};
\node at (1,2.5) {$2$};
\node at (3,1) {$=$};
\path[-,font=\scriptsize]
(0,0) edge (0,1)
(0,1) edge (1,2)
(0,1) edge (-1,2);
\end{tikzpicture}
\qquad
\begin{tikzpicture}[scale=0.5]
\node at (0,1) {$\bullet$};
\node at (-0.6,1) {$f$};
\node at (1,2) {$\bullet$};
\node at (0.5,2) {$g$};
\node at (0,-0.5) {$0$};
\node at (-2,3.5) {$3$};
\node at (0,3.5) {$1$};
\node at (2,3.5) {$2$};
\path[-,font=\scriptsize]
(0,0) edge (0,1)
(0,1) edge (1,2)
(0,1) edge (-2,3)
(1,2) edge (0,3)
(1,2) edge (2,3);
\end{tikzpicture}
\]
\[
\begin{tikzpicture}[scale=0.5]
\node at (0,1) {$\bullet$};
\node at (-0.6,1) {$f$};
\node at (0,-0.5) {$0$};
\node at (-1,2.5) {$2$};
\node at (1,2.5) {$1$};
\node at (2,1) {$\circ_2$};
\path[-,font=\scriptsize]
(0,0) edge (0,1)
(0,1) edge (1,2)
(0,1) edge (-1,2);
\end{tikzpicture}
\begin{tikzpicture}[scale=0.5]
\node at (0,1) {$\bullet$};
\node at (-0.5,1) {$g$};
\node at (0,-0.5) {$0$};
\node at (-1,2.5) {$1$};
\node at (1,2.5) {$2$};
\node at (2,1) {$=$};
\path[-,font=\scriptsize]
(0,0) edge (0,1)
(0,1) edge (1,2)
(0,1) edge (-1,2);
\end{tikzpicture}
\qquad
\begin{tikzpicture}[scale=0.5]
\node at (0,1) {$\bullet$};
\node at (-0.6,1) {$f$};
\node at (-1,2) {$\bullet$};
\node at (-1.5,2) {$g$};
\node at (0,-0.5) {$0$};
\node at (-2,3.5) {$2$};
\node at (0,3.5) {$3$};
\node at (2,3.5) {$1$};
\path[-,font=\scriptsize]
(0,0) edge (0,1)
(0,1) edge (2,3)
(0,1) edge (-1,2)
(-1,2) edge (-2,3)
(-1,2) edge (0,3);
\end{tikzpicture}
\]
The subtlest issue is the labelling of leaves in the tree obtained from grafting.  For a formal definition, we start with grafting for planar trees:

\begin{defn}
Consider a planar $n$-tree $T=(V,E,s,t)$ and a planar $m$-tree $T'=(V',E',s',t')$. For any $1\leq i \leq m$ we define the \define{grafting of $T$ onto $T'$ along $i$} to be the planar $(n+m-1)$-tree  $T'\circ_i T=(\tilde{V},\tilde{E},\tilde{s},\tilde{t})$ where
\begin{itemize}
\item $\tilde{V}=V\sqcup V'$
\item $\tilde{E}=\Bigl(E\setminus\{e_0\}\Bigr)\sqcup\Bigl( E'\setminus\{e_i\}\Bigr)\sqcup \{x\}$, where $e_0$ is the edge of  $T$ with $t(e_0)=0$ and $e_i$ is the edge of $T'$ such that $s'(e_i)=i$
\item $\tilde{s} \maps \tilde{E}\rightarrow \tilde{V}$ is defined by
\[e\mapsto \begin{cases}
s(e) & \text{ if $e\in E$ and $s(e)\in V$}\\
s'(e) & \text{ if $e\in E'$ and $s'(e)\in V'$}\\
s'(e) & \text{ if $e\in E'$ and $1\leq s'(e)\leq i-1$}\\
s(e)+i-1 & \text{ if $e\in E$ and $1\leq s(e)\leq n$}\\
s'(e)+n-1 & \text{ if $e\in E'$ and $i+1\leq s'(e)\leq m$}\\
s(e_0) & \text{ if $e=x$}
\end{cases}\]
\item $\tilde{t}\maps \tilde{E}\rightarrow \tilde{V}$ is defined by
\[ e\mapsto \;\begin{cases} t(e) & \text{ if $e\in E$}\\
t'(e) & \text{ if $e\in E'$}\\
t(e_i) & \text{ if $e=x$}
\end{cases}
\]
\end{itemize}
If in $T$ the order of the children of $t(e_i)$ is $e_1<\dots <e_{i-1}< e_i<e_{i+1}<\dots <e_r$, then the order of its children in $T\circ_i T'$ is $e_1<\dots <e_{i-1}< x<e_{i+1}<\dots <e_r$. The order of the children of all other vertices is unchanged.  

We say that edge $e_0$ is \define{identified} with edge $e_i$.
\end{defn}

Next we define grafting for $C$-labelled planar trees. Suppose we have two $C$-labelled planar trees whose underlying planar trees are $T = (V,E,s,t)$ and $T ' = (V',E',s',t')$.   Then we can make $T \circ_i T'$ into a $C$-labelled planar tree as follows: 
its set of vertices is $V \sqcup V'$, so we label the vertices in $V$ using the 
labelling of $T$, and label those in $V'$ using the labelling of $T'$.  

Grafting is well-defined on isomorphism classes.  We thus obtain partial composition operations for $C$-trees. To make $C\Tree_n$ into the $n$-ary operations of an operad, we also need to give it a right action of the permutation group $S_n$.  We do this by permuting the labels of leaves:

\begin{defn}
Given a $C$-labelled planar $n$-tree $T=(V,E,s,t)$ and a permutation $\sigma\in S_n$, we define the $C$-labelled planar $n$-tree $T\cdot \sigma$ to have the underlying planar $n$-tree $(V,E,s\cdot \sigma, t)$ with same $C$-labelling, where $s\cdot \sigma\maps E\rightarrow V\sqcup \{1,\dots , n\}$ is given by
\[  (s\cdot \sigma)(e) = \left\{ \begin{array}{cl} s(e) & \textrm{if } s(e)\in V \\
\sigma^{-1}(s(e)) & \textrm{otherwise.}  \end{array} \right. \]
We call this \define{relabelling of leaves}.
\end{defn}

This operation defines a right action of the symmetric group $S_n$ on the set of planar $C$-labelled $n$-trees.  One can check that this is well-defined on isomorphism classes, 
so it descends to an action of $S_n$ on the set of $C$-$n$-trees.

For example, if $\sigma \in S_3$ is the cyclic permutation $\left(\begin{array}{ccc}
1 & 2 & 3 \\ 2 & 3 & 1 \end{array}\right)$, we have
\[  
\begin{tikzpicture}[scale=0.6]
\node at (0,0) {$\bullet$};
\node at (-1,1) {$\bullet$};
\node at (0,1) {$\bullet$};
\node at (0,-1) {$\bullet$};
\node at (0,-2.3) {$0$};
\node at (-1.4,1) {$h$};
\node at (-1,2.3) {$2$};
\node at (1,2.3) {$3$};
\node at (2,2.3) {$1$};
\node at (-0.5,-1) {$f$};
\node at (-0.5,0) {$g$};
\node at (-0.4,1) {$i$};
\node at (2.5,0) {$\cdot\,\sigma$};
\node at (4,0) {$=$};
\node at (7,0) {$\bullet$};
\node at (6,1) {$\bullet$};
\node at (7,1) {$\bullet$};
\node at (7,-1) {$\bullet$};
\node at (7,-2.3) {$0$};
\node at (5.6,1) {$h$};
\node at (6,2.3) {$1$};
\node at (8,2.3) {$2$};
\node at (9,2.3) {$3$};
\node at (6.5,-1) {$f$};
\node at (6.5,0) {$g$};
\node at (6.6,1) {$i$};
\path[-,font=\scriptsize]
(0,-1) edge (0,0)
(0,0) edge (-1,1)
(0,0) edge (0,1)
(0,1) edge (-1,2)
(0,1) edge (1,2)
(0,0) edge (2,2)
(0,-2) edge(0,-1)
(7,-1) edge (7,0)
(7,0) edge (6,1)
(7,0) edge (7,1)
(7,1) edge (6,2)
(7,1) edge (8,2)
(7,0) edge (9,2)
(7,-2) edge(7,-1);
\end{tikzpicture}
\]

\begin{lemma}
\label{C-tree}
 Let $C$ be a collection.  There is an operad \define{$C\Tree$} such that:
\begin{itemize}
\item $C\Tree_n$ is the set of $C$-trees with $n$ leaves;
\item composition is given by grafting of trees;
\item the unit is given by the isomorphism class of the tree with no vertices;
\item the permutation group $S_n$ acts on $C\Tree_n$ by relabelling leaves.
\end{itemize}
\end{lemma}

\begin{proof} 
This follows via a straightforward verification of the operad axioms written
in terms of partial composition \cite[Def.\ 12]{May97}.
\end{proof}

Next we show that $C\Tree$ is the free operad on the collection $C$.  There is
a morphism of collections 
\[ \iota \maps C \to U(C\Tree) \]
that sends any element $f\in C_n$ to the isomorphism class of the corolla with its $n$ leaves ordered so that $1< \cdots < n$, and with its vertex labelled by 
$f$.  For example, if $f \in C_3$, then
\[  
\begin{tikzpicture}[scale=0.5]
\node (2) at (0,1){$\bullet$};
\node at (0.7,1){$f$};
\node at (-1,2.5){$1$};
\node at (0,2.5){$2$};
\node at (1,2.5){$3$};
\node at (0,-0.5){$0$};
\node at (-4,1){$\iota(f)$};
\node at (-2.5,1){$=$};
\path[-,font=\scriptsize]
(-1,2) edge (0,1)
(0,2) edge (0,1)
(1,2) edge (0,1)
(0,1) edge (0,0);
\end{tikzpicture} 
\]
where the picture shows the isomorphism class of the corolla with 3 leaves ordered so that $1 < 2 < 3$.   We claim that $\iota$ exhibits $C\Tree$ as the free
operad on $C$.  In other words:

\begin{theorem} 
\label{thm:C-tree-2}
Let $C$ be a collection.  For any operad $O$ and any morphism of collections $\phi \maps C \to U(O)$, there exists a unique operad morphism $\overline{\phi} \maps 
C\Tree \to O$ making this triangle commute:
\[  
\xymatrix@R=3em@C=2.5em{
C \ar[d]_\iota  \ar[dr]^\phi   \\
U(C\Tree) \ar[r]_{U(\overline{\phi})} & U(O) 
}
\]
Thus, $C\Tree$ is the free operad on $C$.
\end{theorem}

\begin{proof} 
The morphism $\overline{\phi} \maps C\Tree \to O$ making the above triangle
commute is clearly unique, because every operation in $C\Tree$ is
obtained from operations of the form $\iota(f)$ by composition and permutations.  
The issue is to show that an operad morphism $\overline{\phi}$ making the triangle 
commute actually exists.

For any morphism of collections $\psi\maps C\to D$, we can define a map $\psi_\star$ from the set of $C$-trees to the set of $D$-trees, mapping any $C$-tree $T$ to the $D$-tree obtained from $T$ by substituting the label $f$ of any vertex of $T$ by $\psi(f)$.  This gives an operad morphism $\psi_\star \maps C\Tree \to D\Tree$.  In particular, starting from $\phi \maps C \to U(O)$ we obtain an operad morphism
\[        \phi_\star \maps C\Tree \to U(O)\Tree  .\]
We shall construct an operad morphism
\[        \epsilon_O \maps U(O)\Tree \to O  \]
with the following property: $\epsilon_O$ maps the isomorphism class of the corolla with its $n$ leaves ordered so that $1< \cdots < n$ and its vertex labelled by 
$f \in U(O)_n$ to the corresponding operation $f \in O_n$.   It will follow that the
composite
\[     \overline{\phi} = \epsilon_O \phi_\star  \]
makes the triangle commute.

We begin by saying what it means to `contract' an edge of a planar tree.  Before giving the definition, we give an example.  Contracting the edge $e$ in the planar tree at left, we obtain the planar tree at right:
\[
\begin{tikzpicture}[scale = 0.9]
\node at (-3.5,3.2) {$4$};
\node at (0,-0.5) {$\bullet$};
\node at (3.5,3.2) {$2$};
\node at (0,1) {$\bullet$};
\node at (-2,3.2) {$3$};
\node at (2,3.2) {$1$};
\node at (-0.3,-0.5) {$$};
\node at (-0.5,1) {$t(e)$};
\node at (0.7,1.4) {$e$};
\node at (1,2) {$\bullet$};
\node at (0,-1.2) {$0$};
\node at (0.5,2.5) {$\bullet$};
\node at (1.5,2) {$s(e)$};
\node at (4,1) {$\rightsquigarrow$};

\node at (5,3.2) {$4$};
\node at (8,0) {$\bullet$};
\node at (11,3.2) {$2$};
\node at (8,1.5) {$\bullet$};
\node at (6.5,3.2) {$3$};
\node at (8,2.5) {$\bullet$};
\node at (9.5,3.2) {$1$};
\node at (7.7,0) {$$};
\node at (7.7,1.3) {$x$};
\node at (7.7,2.5) {$$};
\node at (8,-1.2) {$0$};

\path[-,font=\scriptsize]
(0,-1) edge (0,0)
(0,-0.5) edge (-3.5,3)
(0,-0.5) edge (3.5,3)
(0,-0.5) edge (0,1)
(0,1) edge (-2,3)
(0,1) edge (1,2)
(1,2) edge (0.5,2.5)
(1,2) edge (2,3)

(8,-1) edge (8,0)
(8,0) edge (5,3)
(8,0) edge (11,3)
(8,0) edge (8,2)
(8,1.5) edge (6.5,3)
(8,1.5) edge (8,2.5)
(8,1.5) edge (9.5,3);
\end{tikzpicture}
\]
In the resulting tree, the vertices $s(e)$ and $t(e)$ are gone: they have coalesced to form a new vertex $x$.

\begin{defn}
Given an $n$-tree $T=(V,E,s,t)$, we define \define{$\inc(v)$}
to be the set of children of the vertex $v \in V$.
\end{defn}

\begin{defn}
Given an $n$-tree $T=(V,E,s,t)$,
we call an edge $e\in E$ \define{internal} if its source and target both lie in $V$.
We call the edges that are not internal \define{external}. 
\end{defn}

\begin{defn}
Consider a planar $n$-tree $T=(V,E,s,t)$ with an internal edge $e$. 
We define the planar $n$-tree $T / e =(V_e,E_e,s_e,t_e)$, called 
\define{$T$ with its edge $e$ contracted}, as follows:
\begin{itemize}
\item the vertex set $V_e$ is given by $\bigl(V - \{s(e),t(e)\}\bigr)\sqcup\{x\}$; 
\item the edge set $E_e$ is given by $E -\{e\}$;
\item The maps $s_e$ and $t_e$ are defined as follows:\\
\[ \begin{array}{ccl} s_e(e') &=& \begin{cases}
s(e') & \text{ if } s(e')\ne t(e)\\
x & \text{ otherwise}
\end{cases} 
\\
\\
t_e(e') &=& \begin{cases}
t(e') &\text{ if $t(e')\ne t(e)$ and $t(e')\ne s(e) $}\\ 
x & \text{ otherwise}
\end{cases}
\end{array}
\]
\end{itemize}
The order on the children of a vertex in $V_e$ is defined as it was in $T$ if that vertex
lies in $V$.  For the new vertex $x$, the order is defined as follows.  The vertex $t(e)$ has $k_1>0$ children by construction, while if $s(e)$ if has no children then $x$ has none, so we do not need to define an order on its children. Therefore suppose that  $s(e)$ has $k_2>0$ children, and further that $e$ is the $i$th child of $t(e)$.  The planar structure on $T$ induces order-preserving bijections 
\[  \phi_1\maps \inc(t(e))\to [k_1], \qquad \phi_2\maps \inc(s(e))\to [k_2] \] 
where $[n]$ is the set $\{1,\dots, n\}$ with its standard linear ordering.  
Using these we define a bijection
\[
\phi_1\circ_i \phi_2\maps \inc(t(e))\sqcup \inc(s(e)) \setminus\{e\}\to [k_1+k_2-1]
\] 
as follows:
\[
\phi(y)=\begin{cases}
\phi_1(y) & \text{ if $y\in \inc(t(e))$ and $1\leq \phi_1(y)\leq i-1$}\\
\phi_2(y)+i-1 & \text{ if $y\in \inc(s(e))$}\\
\phi_1(y)+k_1-1 &\text{ if $y\in \inc(t(e))$ and $\phi_1(y)>i$.}
\end{cases}
\]

This induces a linear order on $\inc (x)$.
\end{defn}

More generally, we can define contraction for $U(O)$-labelled planar trees for any operad $O$.  Suppose $T$ is a $U(O)$-labelled planar tree with an internal edge $e$.  We define the $U(O)$-labelled planar tree $T / e$ as follows.  Its underlying planar tree is the underlying planar tree of $T$ with its edge $e$ contracted.  We label all the vertices
other than new vertex $x$ just as in $T$.  As for $x$, suppose that the vertices $t(e)$ and $s(e)$ are labelled by the operations $f\in O_k$ and $g \in O_{\ell}$, respectively,
and suppose that $e$ is the $i$th child of $t(e)$.   Then we label the vertex $x$
by the operation $f\circ_i g$. This yields a $U(O)$-tree: we have $f\circ_ig\in O_{k + \ell-1}$, and by definition $x$ has $k+\ell-1$ children.

Contraction is well-defined on isomorphism classes, so we can define contraction
for $U(O)$-trees.  For example, if we contract the edge between $f$ and $g$ in the $U(O)$-tree at left, we get the one at right:
\[
\begin{tikzpicture}[scale = 0.9]
\node at (-3.5,3.2) {$4$};
\node at (0,-0.5) {$\bullet$};
\node at (3.5,3.2) {$2$};
\node at (0,1) {$\bullet$};
\node at (-2,3.2) {$3$};
\node at (1,2) {$\bullet$};
\node at (2,3.2) {$1$};
\node at (-0.3,-0.5) {$e$};
\node at (-0.3,0.9) {$f$};
\node at (0.2,2.5) {$h$};
\node at (0,-1.2) {$0$};
\node at (0.5,2.5) {$\bullet$};
\node at (1.25,2) {$g$};
\node at (4,1) {$\rightsquigarrow$};

\node at (5,3.2) {$4$};
\node at (8,0) {$\bullet$};
\node at (11,3.2) {$2$};
\node at (8,1.5) {$\bullet$};
\node at (6.5,3.2) {$3$};
\node at (8,2.5) {$\bullet$};
\node at (9.5,3.2) {$1$};
\node at (7.7,0) {$e$};
\node at (7.5,1.3) {$f \circ_2 g$};
\node at (7.7,2.5) {$h$};
\node at (8,-1.2) {$0$};

\path[-,font=\scriptsize]
(0,-1) edge (0,0)
(0,-0.5) edge (-3.5,3)
(0,-0.5) edge (3.5,3)
(0,-0.5) edge (0,1)
(0,1) edge (-2,3)
(0,1) edge (1,2)
(1,2) edge (0.5,2.5)
(1,2) edge (2,3)

(8,-1) edge (8,0)
(8,0) edge (5,3)
(8,0) edge (11,3)
(8,0) edge (8,2)
(8,1.5) edge (6.5,3)
(8,1.5) edge (8,2.5)
(8,1.5) edge (9.5,3);
\end{tikzpicture}
\]

 Iterating this operation, we can assign to any $U(O)$-tree $T$ with $n$ leaves a unique $U(O)$-tree which is a corolla with $n$ leaves and with the unique vertex labelled by the composite of all the operations in $O$ labelling vertices of $T$. This assignment does not depend on the order in which we contract the internal edges, since the composition in $O$ is associative.  We denote the label of the vertex of this corolla by $\epsilon_O(T)$.  

We claim that the resulting map 
\[               \epsilon_O \maps U(O)\Tree \to O  \]
is an operad morphism.  To show this, the only nontrivial task is to show that
\[          \epsilon_O(T' \circ_i T) = \epsilon_O(T') \circ_i \epsilon_O(T)  \]
when $T$ and $T'$ are $U(O)$-trees.  To do this, we note that instead of contracting all the
internal edges of a tree, we could contract only those in a subtree.   Here we borrow
a definition from Fresse \cite[A 1.5]{Fre}:

\begin{defn}\label{defn:subtree}
A \define{subtree} $S=(V_S,E_S,\inc_S,s_S,t_S)$ of a planar $n$-tree $T=(V,E,s,t)$ is given by:
\begin{itemize}
\item a set of vertices $V_S\subseteq V$,
\item a set of edges $E_S\subseteq E$,
\item a set $\inc_S \subseteq V\sqcup \{1,\dots , n\}$ such that $\inc_S \cap V_S =\emptyset$,
\item an element $0_S\in V$ such that $0_S \notin V_S$ and such that there is a unique edge $e_0$ in $E_S$ with $t(e_0)=0_S$,
\item $s_S = s|_{E_S}$ and $t_S = t|_{E_S}$. 
\end{itemize}
This data satisfies the following requirement: an edge $e$ is in $E_S$ if and only if $t(e)\in V_S\sqcup \{0_S\}$ if and only if $s(e)\in V_S \sqcup \inc_S$.
\end{defn}
The last requirement in the definition ensures that in a subtree $S$ there is a unique directed edge path from any vertex to $0_S$, and also that if a vertex is in $V_S$ then  all its children and the edge with this vertex as its source are in $E_S$. Furthermore, a subtree is completely determined by its set of vertices or its set of edges, as noted by Fresse \cite[A 1.6]{Fre}.
We also note that we have modified Fresse's definition slightly, to ensure that all trees are subtrees of themselves.  Fresse requires the set of vertices $V_S$ to be non-empty so that trivial trees are not allowed to be subtrees.  Thanks to the  last requirement in the definition, our modification  allows trivial trees to be subtrees only of themselves.

The definition of subtree can be generalized to $C$-labelled planar trees and further to $C$-trees, where the subtree
inherits its labels from the original tree.  For example, given this $C$-3-tree:
\[ \begin{tikzpicture}[scale=0.6]
\node at (-1,1) {$\bullet$};
\node at (0,0) {$\bullet$};
\node at (-0.6,0) {$$};
\node at (-1.3,1) {$$};
\node at (-2,2.3) {$2$};
\node at (0,2.3) {$1$};
\node at (2,2.3) {$3$};
\node at (-0.5,0) {$f$};
\node at (-1.5,1) {$g$};
\node at (0,-1.3) {$0$};
\path[-,font=\scriptsize]
(0,0) edge (-1,1)
(0,0) edge (0,-1)
(0,0) edge (2,2)
(-1,1) edge (0,2)
(-1,1) edge (-2,2);
\end{tikzpicture}
\] 
this is a $C$-subtree:
\[ 
\begin{tikzpicture}[scale = 0.6]
\node at (-0.7,-0.3) {$0_S$};
\node at (-1,1) {$\bullet$};
\node at (-1.3,1) {$$};
\node at (-2,2.3) {$2$};
\node at (0,2.3) {$1$};
\node at (-1.5,1) {$g$};
\path[-,font=\scriptsize]
(-1,0) edge (-1,1)
(-1,1) edge (0,2)
(-1,1) edge (-2,2);
\end{tikzpicture}
\] 
while this is not: 
\[
\begin{tikzpicture}[scale = 0.6]
\node at (0,-1.3) {$0_S$};
\node at (0,0) {$\bullet$};
\node at (-0.5,0) {$f$};
\node at (0,1.3) {$3$};
\path[-,font=\scriptsize]
(0,0) edge (0,-1)
(0,0) edge (0,1);
\end{tikzpicture}
\]
because an edge that is a child of the vertex labelled $f$ is not included.

Given a tree $T$ with a subtree $S$, call an edge $e\in E_S$ \define{internal to} $S$ if $s(e)$ and $t(e)$ lie in $V_S$, and \define{external to} $S$ otherwise.  As noted by Fresse \cite{Fre}, contracting the edges internal to a subtree $S$ of a $U(O)$-tree we obtain another $U(O)$-tree.  We call this operation \define{contraction of the subtree $S$}. For example, we can contract the subtree containing the vertices in the green ellipse at left, and obtain the $U(O)$-tree at right:
\[
\begin{tikzpicture}
\node (1) at (0,0) {$\bullet$};
\node at (0.3,0) {$f$};
\node at (-1,1) {$\bullet$};
\node at (-1.3,1) {$g$};
\node (2) at (1,1) {$\bullet$};
\node at (1.3,1) {$h$};
\node at (0.5,2) {$\bullet$};
\node at (0.5,2.3) {$k$};
\node at (0,-1.2) {$0$};
\node at (-1.5,2.7) {$2$};
\node at (-0.5,2.7) {$1$};
\node at (1,2.7) {$4$};
\node at (1.5,2.7) {$3$};
\path[-,font=\scriptsize]
(0,0) edge (0,-1)
(0,0) edge (-1,1)
(0,0) edge (1,1)
(1,1) edge (0.5,2)
(1,1) edge (1,2.5)
(1,1) edge (1.5,2.5)
(-1,1) edge (-1.5,2.5)
(-1,1) edge (-0.5,2.5);
  \ellipsefoci{draw,green}{0,0}{2,2.8}{1.2}
\end{tikzpicture}\qquad
\raisebox{4em}{
\begin{tikzpicture}
\node at (0,5) {$\rightsquigarrow$};
\end{tikzpicture}}\qquad
\begin{tikzpicture}
\node (1) at (0,0) {$\bullet$};
\node at (1,0) {$f\circ_2(h \circ_1 k)$};
\node at (-1,1) {$\bullet$};
\node at (-1.3,1) {$g$};
\node at (0,-1.2) {$0$};
\node at (-1.5,2.7) {$2$};
\node at (-0.5,2.7) {$1$};
\node at (0.5,2.7) {$4$};
\node at (1.5,2.7) {$3$};
\path[-,font=\scriptsize]
(0,0) edge (0,-1)
(0,0) edge (-1,1)
(0,0) edge (0.5,2.5)
(0,0) edge (1.5,2.5)
(-1,1) edge (-1.5,2.5)
(-1,1) edge (-0.5,2.5);
\end{tikzpicture}
\]

Now we can check that 
\[          \epsilon_O(T' \circ_i T) = \epsilon_O(T') \circ_i \epsilon_O(T)  \]
when $T$ and $T'$ are $U(O)$-trees.  At left, we first graft $T$ onto $T'$ and 
then contract the resulting $U(O)$-tree.  Thanks to the associativity of operadic composition, this is the same as grafting $T$ onto $T'$, then contracting the subtree $T$ of $T' \circ_i T$, and then contracting the resulting $U(O)$-tree.  But this is the same
as contracting $T'$, contracting $T$, and then composing the operations
in $O$ that label the two resulting corollas.  This is the expression at right. This completes the proof of Theorem \ref{thm:C-tree-2}.
\end{proof}

The simplest case of Theorem \ref{thm:C-tree-2} is when $C = U(\Com_+)$, where $\Com_+$
is the terminal operad:

\begin{corollary}
\label{cor:PTree}
Let $P\Tree$ be the operad whose $n$-ary operations are isomorphism classes of
planar $n$-trees, with composition defined by grafting and permutation group actions
given by relabelling leaves.  Then $P\Tree \cong FU(\Com_+)$.  
\end{corollary}

\begin{proof}  
Since $\Com_+$ has just one operation of each arity, there is always
just one way to label vertices of a planar tree by operations of $\Com_+$.  
Thus, an operation $U(\Com_+)\Tree$ can be naturaly identified with an 
isomorphism class of planar trees, and by Lemma \ref{C-tree} we have
$U(\Com_+)\Tree \cong P\Tree$ as operads.  
The result then follows from Theorem \ref{thm:C-tree-2}.
\end{proof}

More generally, we make the following definition, closely tied to Definition \ref{defn:CTree}:

\begin{defn}
\label{defn:OTree}
For any operad $O$, we define an \define{$O$-$n$-tree} to be a $U(O)$-$n$-tree, where $U(O)$ is the underlying collection of $O$.  We define an \define{$O$-tree} to be a
$U(O)$-tree.
\end{defn}

Thus, an $O$-tree is an operation in $FU(O)$.   Recall that $\Com_+$ is the 
terminal operad, so there is a unique operad morphism $!_O \maps O \to \Com_+$.
This in turn gives a morphism
\[        FU(!_O)) \maps FU(O) \to FU(\Com_+)
\iso \PTree  \]
sending each $O$-tree to the isomorphism class of its underlying 
planar $n$-tree.   For example:
\[
\begin{tikzpicture}[scale = 0.7]
\node at (-3.5,0){$FU(!_O) \maps$};
\node at (-1,0){$\bullet$};
\node at (-0.3,0){$g_1$};
\node at (1,0){$\bullet$};
\node at (1.4,0){$g_2$};
\node at (2.5,0){$\bullet$};
\node at (3,0){$g_3$};
\node at (1,-1){$\bullet$};
\node at (1.5,-1){$f$};
\node at (-1.5,1.4){$3$};
\node at (-0.5,1.4){$2$};
\node at (2.5,1.4){$1$};
\node at (4.5,0){$\mapsto$};
\node at (1,-2.5){$0$};                    
\path[-,font=\scriptsize]
(-1.5,1) edge (-1,0)
(-0.5,1) edge (-1,0)
(2.5,1) edge (2.5,0)
(-1,0) edge (1,-1)
(1,0) edge (1,-1)
(2.5,0) edge (1,-1)
(1,-1) edge (1,-2);
\end{tikzpicture}
\qquad
\begin{tikzpicture}[scale = 0.7]
\node at (-1,0){$\bullet$};
\node at (1,0){$\bullet$};
\node at (2.5,0){$\bullet$};
\node at (1,-1){$\bullet$};
\node at (-1.5,1.4){$3$};
\node at (-0.5,1.4){$2$};
\node at (2.5,1.4){$1$};
\node at (1,-2.5){$0$};                    
\path[-,font=\scriptsize]
(-1.5,1) edge (-1,0)
(-0.5,1) edge (-1,0)
(2.5,1) edge (2.5,0)
(-1,0) edge (1,-1)
(1,0) edge (1,-1)
(2.5,0) edge (1,-1)
(1,-1) edge (1,-2);
\end{tikzpicture}
\]
This clarifies the special role of planar $n$-trees in the theory of operads.

On the other hand, the counit of the adjunction between operads and collections
\[             \epsilon_O \maps FU(O) \to O  \]
maps each $O$-tree to an operation in $O$.   For example:
\[
\begin{tikzpicture}[scale = 0.7]
\node at (-2.5,0){$\epsilon_O \maps$};
\node at (-1,0){$\bullet$};
\node at (-0.3,0){$g_1$};
\node at (1,0){$\bullet$};
\node at (1.4,0){$g_2$};
\node at (2.5,0){$\bullet$};
\node at (3,0){$g_3$};
\node at (1,-1){$\bullet$};
\node at (1.5,-1){$f$};
\node at (-1.5,1.4){$3$};
\node at (-0.5,1.4){$2$};
\node at (2.5,1.4){$1$};
\node at (4.5,0){$\mapsto$};
\node at (1,-2.5){$0$};                    
\path[-,font=\scriptsize]
(-1.5,1) edge (-1,0)
(-0.5,1) edge (-1,0)
(2.5,1) edge (2.5,0)
(-1,0) edge (1,-1)
(1,0) edge (1,-1)
(2.5,0) edge (1,-1)
(1,-1) edge (1,-2);
\end{tikzpicture}
\qquad
\begin{tikzpicture}[scale = 0.7]
\node at (1,-1){$\bullet$};
\node at (3,-1){$f \circ (g_1, g_2, g_3)$};
\node at (-0.5,1.4){$3$};
\node at (1,1.4){$2$};
\node at (2.5,1.4){$1$};
\node at (1,-2.5){$0$};                    
\path[-,font=\scriptsize]
(-0.5,1) edge (1,-1)
(1,1) edge (1,-1)
(2.5,1) edge (1,-1)
(1,-1) edge (1,-2);
\end{tikzpicture}
\]
We can use this to describe operations in $O$ as equivalence classes of $O$-trees, in a way that will be useful later.

First note that we can act on a planar treee by permuting the children of a vertex. More precisely:

\begin{defn}\label{defn:perm subtree}
 Suppose $S$ is a subtree of a planar $n$-tree $T$, and that $S$ consists of a single vertex. Then the linear order on $\inc_S$ gives an order-preserving isomorphism $f\maps \inc_S\to [k]$ for some $k\geq 0$. Define the \define{permutation of $S$ by $\sigma$} to be the planar $n$-tree \define{$S\cdot \sigma$} with same underlying $k$-tree as $S$ and linear order on $\inc_S$ given by $\sigma^{-1}\circ f$. 
 \end{defn}
 \noindent
 
This definition can be generalized in a straighforward way to $C$-labelled planar $n$-trees and further to $C$-$n$-trees. We are now ready to state our result:

\begin{theorem}
\label{thm:counit}
Let $O$ be an operad.  Then $\epsilon_O$ maps two $O$-trees to the same operation of $O$ if and only if we can go from one $O$-tree to the other by a finite sequence of the following moves:

\begin{enumerate}
\item \label{item:composition}
Given any $O$-tree, replace any subtree consisting of a vertex together with its children and their source vertices by its contraction.

\item \label{item:unit} For any $O$-tree,  replace any edge by a corolla with one vertex labelled by the identity $1 \in O_1$.

\item \label{item:symmetry} For any $O$-tree, replace any subtree $S$ given by exactly one vertex $v$ labelled by $f\cdot \sigma$, where $\sigma \in S_k$ and $f \in O_k$, by the subtree obtained by permuting $S$ by $\sigma$ and substituting the label of $v$ by $f$.

\end{enumerate}
\end{theorem}

The following is a move of type 1:
\[
\begin{tikzpicture}[scale = 0.8]
\node at (-2,1){}; 
\node at (-1,1){};
\node at (0,1){};
\node at (0,0){};
\node at (2,1){};
\node at (3,1){};
\node at (-1,0){$\bullet$};
\node at (-0.3,0){$g_1$};
\node at (1,0){$\bullet$};
\node at (1.4,0){$g_2$};
\node at (2.5,0){$\bullet$};
\node at (3,0){$g_3$};
\node at (1,-1){$\bullet$};
\node at (1.5,-1){$f$};
\node at (7.5,-1){$\bullet$};
\node at (9.3,-1){$f \circ (g_1, g_2, g_3)$};
\node at (4.5,-1){$\sim$};
\path[-,font=\scriptsize]
(-2,1) edge (-1,0)
(-1,1) edge (-1,0)
(0,1) edge (-1,0)
(2,1) edge (2.5,0)
(1,1) edge (1,0)
(3,1) edge (2.5,0)
(-1,0) edge (1,-1)
(1,0) edge (1,-1)
(2.5,0) edge (1,-1)
(1,-1) edge (1,-2)
(5,1) edge (7.5,-1)
(6,1) edge (7.5,-1)
(7,1) edge (7.5,-1)
(8,1) edge (7.5,-1)
(9,1) edge (7.5,-1)
(10,1) edge (7.5,-1)
(7.5,-1) edge (7.5,-2);
\end{tikzpicture}
\]
This is a move of type 2:
\[
\begin{tikzpicture}
\node (1) at (0,2){};
\node (2) at (0,1) {$\bullet$}; 
\node at (0.2,1) {$1$};
\node (3) at (0,0){};
\node at (2,1) {$\sim$};
\path[-,font=\scriptsize]
(0,0) edge (0,2)
(4,0) edge (4,2);
\end{tikzpicture}
\]
and for $\sigma = \left(\begin{array}{ccc} 1 & 2 & 3 \\ 2 & 1 & 3\end{array} \right)$, this is a move of type 3:
\[
\begin{tikzpicture}[scale=1, every node/.style={scale=1}]
\node at (-1,1) {$\bullet$};
\node at (0,0) {$\bullet$};
\node at (0,1) {$\bullet$};
\node at (-0.2,1) {$s$};
\node at (-1.3,1) {$g$};
\node at (-0.5,0) {$f\cdot \sigma $};
\node at (2.5,0) {$\sim$};
\node at (4,1) {$\bullet$};
\node at (5,0) {$\bullet$};
\node at (4.7,0) {$f$};
\node at (4.7,1) {$g$};
\node at (3.8,1) {$s$};
\node at (5,1) {$\bullet$};
\path[-,font=\scriptsize]
(-1,1) edge (-2,2)
(-1,1) edge (0,2)
(-1,1) edge (0,0)
(0,0) edge (0,1)
(0,0) edge (2,2)
(0,0) edge (0,-1)
(5,0) edge (4,1)
(5,0) edge (5,1)
(5,1) edge (4,2)
(5,1) edge (6,2)
(5,0) edge (7,2)
(5,0) edge (5,-1);
\end{tikzpicture}
\]

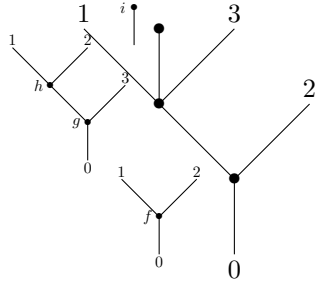
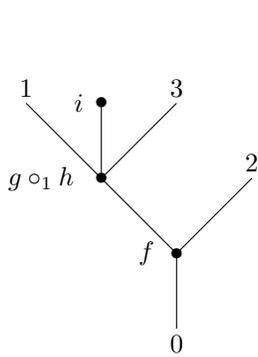
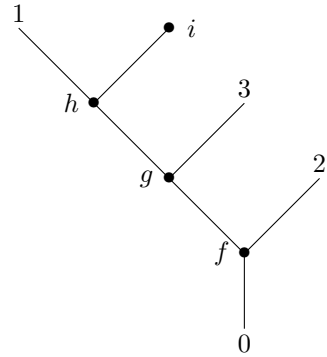
\begin{figure}[h!]
\centering
 \qquad
\subfigure[{An operation in $FUFU(O)_3$.}]{
\label{subf oper}
\centering
\begin{tikzpicture}
\node at (0,0) {$\bullet$};
\node at (-1,1) {$\bullet$};
\node at (-1,2) {$\bullet$};
\node at (-1.4,2.1){$
\begin{tikzpicture}[scale=0.5, every node/.style={scale=0.6}]
\node at (0,0) {$\bullet$};
\node at (-0.3,0) {$i$};
\path[-,font=\scriptsize]
(0,0) edge (0,-1);
\end{tikzpicture}$};
\node at (-2.2,1) {$\begin{tikzpicture}[scale=0.5, every node/.style={scale=0.6}]
\node at (-1,1) {$\bullet$};
\node at (0,0) {$\bullet$};
\node at (-0.3,0) {$g$};
\node at (-1.3,1) {$h$};
\node at (-2,2.2) {$1$};
\node at (0,2.2) {$2$};
\node at (1,1.2) {$3$};
\node at (0,-1.2) {$0$};
\path[-,font=\scriptsize]
(0,0) edge (-1,1)
(0,0) edge (0,-1)
(0,0) edge (1,1)
(-1,1) edge (0,2)
(-1,1) edge (-2,2);
\end{tikzpicture}$};
\node at (-1,-0.5) {$\begin{tikzpicture}[scale=0.5, every node/.style={scale=0.6}]
\node at (0,0) {$\bullet$};
\node at (-0.3,0) {$f$};
\node at (-1,1.2) {$1$};
\node at (1,1.2) {$2$};
\node at (0,-1.2) {$0$};
\path[-,font=\scriptsize]
(0,0) edge (-1,1)
(0,0) edge (0,-1)
(0,0) edge (1,1);
\end{tikzpicture}$};
\node at (-2,2.2) {$1$};
\node at (0,2.2) {$3$};
\node at (1,1.2) {$2$};
\node at (0,-1.2) {$0$};
\path[-,font=\scriptsize]
(0,0) edge (-1,1)
(0,0) edge (0,-1)
(0,0) edge (1,1)
(-1,1) edge (0,2)
(-1,1) edge (-2,2)
(-1,1) edge (-1,2);
\end{tikzpicture}
}\\
\subfigure[{
The morphism $FU\epsilon_O$ sends the operation in Figure~\ref{subf oper}  to the  $O$-tree obtained by applying $\epsilon_O$ to the vertex labels.}]
{
\centering
\begin{tikzpicture}
\node at (0,0) {$\bullet$};
\node at (-1,1) {$\bullet$};
\node at (-1,2) {$\bullet$};
\node at (-1.3,2) {$i$};
\node at (-1.8,1) {$g\circ_1h$};
\node at (-0.4,0) {$f$};
\node at (-2,2.2) {$1$};
\node at (0,2.2) {$3$};
\node at (1,1.2) {$2$};
\node at (0,-1.2) {$0$};
\path[-,font=\scriptsize]
(0,0) edge (-1,1)
(0,0) edge (0,-1)
(0,0) edge (1,1)
(-1,1) edge (0,2)
(-1,1) edge (-2,2)
(-1,1) edge (-1,2);
\end{tikzpicture}
}\hspace{2cm}
\subfigure[{
The morphism  $\epsilon_{FUO}$ sends the operation in Figure~\ref{subf oper} to the  tree obtained by  grafting the vertex labels.
}]{\centering
\begin{tikzpicture}
\node at (0,0) {$\bullet$};
\node at (-1,1) {$\bullet$};
\node at (-2,2) {$\bullet$};
\node at (-1,3) {$\bullet$};
\node at (-0.3,0) {$f$};
\node  at (-1.3,1) {$g$};
\node at (-2.3,2) {$h$};
\node at (-0.7,3) {$i$};
\node at (-3,3.2) {$1$};
\node at (0,2.2) {$3$};
\node at (1,1.2) {$2$};
\node at (0,-1.2) {$0$};
\path[-,font=\scriptsize]
(0,0) edge (-1,1)
(0,0) edge (0,-1)
(0,0) edge (1,1)
(-1,1) edge (0,2)
(-1,1) edge (-2,2)
(-2,2) edge (-1,3)
(-2,2) edge (-3,3);
\end{tikzpicture}
}
\caption{Example of an operation in $FUFU(O)_3$ and its image under $FU\epsilon_O$ and $\epsilon_{FUO}$.}\label{F:equivalence_ex}
\end{figure}

\begin{proof}

By Lemma \ref{monadic} we know that $(O,\epsilon_O)$ is the coequalizer of the following diagram:
\[        
\xymatrix{
FUFU(O)  \ar@<1ex>[rr]^{\epsilon_{FU(O)}}  \ar@<-1ex>[rr]_{FU(\epsilon_O)} &&
FU(O)
}
\]

Furthermore by Theorem \ref{thm:C-tree-2} we know that operations in $FUFU(O)$ are $FU(O)$-trees, so they are isomorphism classes of planar trees with vertices labelled by $O$-trees. 
We say that two $O$-trees $T,T'\in FU(O)_n$ are \define{equivalent} if and only if $T$ and $T'$ are related by the smallest equivalence relation with 
\[         \epsilon_{FUO}(H) \sim FU\epsilon_O(H)   \]
for all $H \in FUFU(O)_n$.  We give an example of an operation in $FUFU(O)$ and its image under the morphisms   $\epsilon_{FUO}$ and  $FU\epsilon_O$ in Figure \ref{F:equivalence_ex}. To prove the theorem, it suffices to show that two trees in $FU(O)_n$ are equivalent if and only if they differ by a finite sequences of moves  (\ref{item:composition})--(\ref{item:symmetry}).   For the `if' direction, it is enough to show that $T$ and $T'$ are equivalent if they differ by exactly one of these moves.

For moves of type (\ref{item:composition}) and (\ref{item:unit}) it is enough to consider trees with leaves labelled by $1<\dots <n$, since $\epsilon_{FUO}$ and  $FU\epsilon_O$ are equivariant. We call such trees \define{unpermuted}. Furthermore, if $T$ is an unpermuted tree and  $T'=T\cdot \sigma$ we say that $T'$ has \define{leaves permuted by $\sigma$}.

First suppose that $T'$ is obtained from $T$ by applying a move of type (\ref{item:composition}) in the forward direction. This means that $T$ contains a subtree $S$ with $k\leq n$ leaves and $T'$ is obtained from $T$ by contracting $S$. 
We denote  by $v$ the vertex of $T'$ corresponding to the contraction of $S$. 
Let $H$ be the $FU(O)$-tree whose underlying isomorphism class of planar $n$-trees is the same as that of $T'$ and 
whose vertex $v$ is labelled by $S$, while every other vertex $v_i$ is labelled by a corolla with unpermuted leaves and vertex labelled by the label of $v_i$ in $T'$.  In this case have $T = \epsilon_{FUO}(H)$ and $T' = FU\epsilon_O(H)$ as desired.

Next suppose that $T'$ is obtained from $T$ by applying move (\ref{item:unit}) in the forward direction. Let $v$ be the vertex of $T$ labelled by the identity that is replaced by an edge in $T'$. Then we let $H$ be the $FU(O)_n$-tree with underlying isomorphism class of planar $n$-trees that of $T$ and such that $v$ is labelled by the isomorphism class of the tree with no vertices. We again have $T = \epsilon_{FUO}(H)$ and $T' = FU\epsilon_O(H)$.

Finally, suppose that $T$ and $T'$ have arbitrary leaf labellings and that $T'$ is obtained from $T$ by a move of type (\ref{item:symmetry}). Let $H$ be the $FU(O)$-tree with underlying $n$-tree a corolla with unpermuted leaves and its only vertex labelled by $T$, and similarly let $H'$ be the  $FU(O)$-tree with underlying $n$-tree a  corolla with unpermuted leaves and its only vertex labelled by $T'$.  Then we clearly have that $T=\epsilon_{FU(O)}(H)$ and $T'=\epsilon_{FU(O)}(H')$. Furthemore, the equality $FU\epsilon_O(H)=FU\epsilon_O(H')$ shows that  $T$ and $T'$ are equivalent.

Conversely, we have to show that for any tree $H\in FUFU(O)_n$ we can go from $\epsilon_{FUO}(H)$ to $FU\epsilon_O(H)$ with a finite sequence of moves (\ref{item:composition})--(\ref{item:symmetry}).  We prove this by induction on the number $n$ of vertices of $H$ that are labelled by trees that are not unpermuted corollas. 

For $n=0$ we have $\epsilon_{FUO}(H)=FU\epsilon_O(H)$.  So, assume that $H$ has exactly $n+1$ vertices labelled by trees other than unpermuted corollas. Let  $v$ be one of these vertices and denote its label by $S$. First, suppose that $S$ is not the isomorphism class of the tree with no vertices. Let $\widetilde{H}$ be the $FU(O)$-tree obtained from $H$ by substituting the label of $v$ by a corolla with its vertex labelled by $\epsilon_O(S)$. Then we have $FU\epsilon_O(H)=FU\epsilon_O(\widetilde{H})$.  Let $T=\epsilon_{FU(O)}(H)$ and  $\widetilde{T}=\epsilon_{FU(O)}(\widetilde{H})$. The label of $v$ is sent in $T$ to a subtree with underlying tree that of $S$ and same labels on the vertices, while it is sent in $\widetilde{T}$ to a vertex labelled by $\epsilon_O(S)$. Thus we can go from $T$ to $\widetilde{T}$ with a move of type (\ref{item:composition}) or type (\ref{item:symmetry}) or both.  Next, assume that $S$ is the isomorphism class of the tree with no vertices. Then we let $\widetilde{H}$ be the $FU(O)_n$-tree obtained from $H$ by deleting $v$. In this case we have $\epsilon_{FU(O)}(H)=\epsilon_{FU(O)}(\widetilde{H})$, and we can go from $FU\epsilon_O(H)$ to $FU\epsilon_O(\widetilde{H})$ with a move of type (\ref{item:unit}) in the forward direction. 

The claim now follows by the induction hypothesis. This completes the proof of Theorem \ref{thm:counit}. \end{proof}

\section{Coproducts of operads}
\label{sec:coproduct}

We can use Theorem \ref{thm:C-tree-2} and Theorem \ref{thm:counit} to describe the coproduct of operads.  Given operads $O$ and $O'$, their coproduct is an operad $O+O'$.  Its algebras are easy to describe: by the universal property of the coproduct, an algebra of $O+O'$ is a topological space that is an algebra both of $O$ and $O'$.  Its collection of operations, on the other hand, is a bit complicated.  

Leinster \cite{Lei} has described the coproduct for non-symmetric operads in the category of sets.  To prove our result, we adapt his result to the operads we are considering: symmetric topological operads.  

To build $O+O'$, first note that there are epimorphisms
\[         \epsilon_O \maps FU(O) \to O , \qquad \epsilon_{O'} \maps FU(O') \to O' . \]
Taking their coproduct, we obtain an epimorphism
\[         \epsilon_O + \epsilon_{O'} \maps FU(O) + FU(O') \to O + O' . \]
On the other hand, left adjoints preserve coproducts, so we have a canonical isomorphism
\[   FU(O) + FU(O') \cong F(U(O) + U(O')) \]
This gives us, with a slight abuse of notation, an epimorphism
\[         \epsilon_O + \epsilon_{O'} \maps F(U(O) + U(O')) \to O + O' . \]
By Theorem \ref{thm:C-tree-2}, operations in $F(U(O) + U(O'))$ can
be seen as $U(O) + U(O')$-trees.   The epimorphism above thus lets us
describe operations of $O + O'$ as equivalence classes of $U(O) + U(O')$-trees.

What is the equivalence relation?  This is answered by the following result, which is based on Theorem \ref{thm:counit}.
To state the result we will need the following definition:
\begin{defn}
Let $O$ and $O'$ be operads. A \define{$O$-subtree} of a $U(O)+U(O')$ tree is a subtree having vertices labelled only by operations of $O$.
\end{defn}

\begin{theorem}
\label{thm:coproduct}
Let $O$ and $O'$ be operads.  Operations in $F(U(O) + U(O'))$ may be identified with $U(O)+U(O')$-trees.  Two $U(O)+U(O')$-trees map to the same operation of $O+O'$ via the operad morphism
\[         \epsilon_O + \epsilon_{O'} \maps F(U(O) + U(O')) \to O + O'  \]
if and only if we can go from one to the other by a finite sequence of the following moves:

\begin{enumerate}
\item \label{item:composition_2}
For any $U(O)+U(O')$-tree, we can replace any $O$-subtree by its contraction.

\item \label{item:unit_2} For any $U(O)+U(O')$-tree, we can replace any edge by a corolla with its vertex labelled by the identity $1 \in O_1$.

\item \label{item:symmetry_2} For any $U(O)+U(O')$-tree, we can replace any $O$-subtree given by exactly one vertex $v$ labelled by $f\cdot \sigma$, where $\sigma \in S_k$ and $f \in O_k$, by the subtree obtained from $S$ by permuting $S$ by $\sigma$ and substituting the label of $v$ by $f$.

\item \label{item:composition_3} The same as (\ref{item:composition_2}) with $O'$ instead of $O$.

\item \label{item:unit_3} The same as (\ref{item:unit_2}) with $O'$ instead of $O$.

\item \label{item:symmetry_3} The same as (\ref{item:symmetry_2}) with $O'$ instead of $O$.
\end{enumerate}
\end{theorem}

\begin{proof}
We know from Theorem \ref{thm:counit} that operations in $O+O'$  are equivalence classes of $U(O)+U(O')$-trees, while operations in $F(U(O)+U(O'))$ are $U(O)+U(O')$-trees by Theorem \ref{thm:C-tree-2}.

Since $\epsilon_O+\epsilon_{O'}$ is equivariant, we may prove the statement only for unpermuted $U(O)+U(O')$-trees.
Given a $O$-tree $T$, we have $\epsilon_O+\epsilon_{O'}(T)=i_O(\epsilon_O(T))$ where $i_O\colon O\to O+O'$ is part of the data of the coproduct. Similarly, we have $\epsilon_O+\epsilon_{O'}(T')=i_{O'}(\epsilon_{O'}(T'))$ for any $O'$-tree $T'$.
We can write any $U(O)+U(O')$-tree as a composition of  $O$- and $O'$-trees, and call such a collection of trees a \define{decomposition} of $T$.   
We write 
\[
 T=S_1\circ_{j_1}S_2\circ_{j_2}\dots \circ_{k-1}S_k
 \]
for such a decomposition, where $S_i$ is either a $O$-tree or a $O'$-tree.

We define a partial order on the set of decompositions of a tree: a decomposition $D$ is smaller then a decomposition $D'$ if by substituting none or a finite number  of trees in $D$ by their decomposition we obtain $D'$.  It is easy to see that every tree has maximum and minimum decompositions which are unique up to a permutation of the terms.  We then have 
\[
 \epsilon_O+\epsilon_{O'}(T)=\epsilon_{O+O'}(S_1)\circ_{j_1}\epsilon_{O+O'}(S_2)\circ_{j_2}\dots \circ_{k-1}\epsilon_{O+O'}(S_k)
 \] 
 \noindent 
 where each term $\epsilon_{O+O'}(S_i)$ is either $i_O\epsilon_O(S_i)$ or $i_{O'}\epsilon_{O'}(S_i)$ depending on whether $S_i$ is an $O$-tree or $O'$-tree.
  This does not depend on the decomposition of the tree because $\epsilon_O$ and $\epsilon_{O'}$ are operad morphisms.
 The morphism $\epsilon_O+\epsilon_{O'}$ sends two trees $T$ and $T'$ to the same equivalence class in $O+O'$ if and only if 
\[
\epsilon_{O+O'}(S_1)\circ_{j_1}\epsilon_{O+O'}(S_2)\circ_{j_2}\dots \circ_{j_{k-1}}\epsilon_{O+O'}(S_k)
=
\epsilon_{O+O'}(S'_1)\circ_{j'_1}\epsilon_{O+O'}(S'_2)\circ_{j'_2}\dots \circ_{j_{k'-1}}\epsilon_{O+O'}(S'_{k'})
\]
If we take the minimum decompositions of $T$ and $T'$, we necessarily have $k=k'$, and  $\epsilon_{O+O'}(S_p)=\epsilon_{O+O'}(S'_p)$ for all $1\leq p\leq k$ (if necessary take a permutation of the decomposition of one of the trees).  Since $i_O$ and $i_{O'}$ are monomorphisms, this is equivalent to either $\epsilon_O(S_p)=\epsilon_O(S'_p)$ or $\epsilon_{O'}(S_p)=\epsilon_{O'}(S_p')$. By Theorem \ref{thm:counit} we know that this is  the case if and only if we can go from $S_p$ to $S'_{p}$ with a finite sequence of moves \ref{item:composition_2}, \ref{item:unit_2}, \ref{item:symmetry_2}. We thus obtain the desired result.
\end{proof}

Theorem \ref{thm:coproduct} tells us what the operations of $(O+O')_n$ are, but what about its topology?   We know that there is an epimorphism  $\epsilon_O+\epsilon_{O'}$ from $F(U(O)+U(O'))$ to $O+O'$.  The topology of the spaces $F(U(O)+U(O'))_n$ underlying the free operad $F(U(O)+U(O'))$ is the finest topology making the maps $i_n\maps O_n\to F(U(O)+U(O'))_n$ and $i'_n\maps O'_n\to F(U(O)+U(O'))_n$ continuous.  It is also easy to show that the topology of $(O+O')_n$ is the finest topology  making the maps $(\epsilon_O+\epsilon_{O'})_n\maps F(U(O)+U(O'))_n\to O+O'_n$ continuous.  We will need this fact in Appendix \ref{A:topology of phyl}, where we study the topology of the space of phylogenetic trees. 

We can also use Theorem \ref{thm:coproduct} to show that for any operads $O$ and $O'$, any suboperad of $O$ is a suboperad of $O+O'$:

\begin{defn}
Given a pair of operads $P$ and $Q$, we say that $P$ is a \define{suboperad} of $Q$ if it is equipped with a monomorphism $\iota \maps P \to O$.
\end{defn}
 
\begin{corollary}
\label{cor:suboperad}
Let $O$ and $O'$ be operads. Then the morphisms $\iota_O \maps O\to O+O'$ and 
$\iota_{O'} \maps O \to O + O'$ that are part of the data of the coproduct are monomorphisms. 
As a consequence, any suboperad of $O$ becomes a suboperad of $O+O'$. 
\end{corollary}

\begin{proof}
We note that since limits of operads can be computed pointwise \cite[Prop.\ I.1.2.4]{Fre}, a morphism $\iota \maps P \to Q$ is monic if and only if all the maps $\iota_n \maps P_n \to Q_n$ 
is monic in $\Top$. Furthermore, the monomorphisms in $\Top$ are the continuous injections.  
Thus, we only need to prove that each map $\iota_{O,n} \maps O_n \to (O+O')_n$ is a continuous injection.  This is easy to see from the explicit description of $O + O'$ given in Theorem \ref{thm:coproduct}, together with the description of the topology on $O+O'$.
\end{proof}

\section{The coproduct of an operad and a unary operad} 
\label{sec:coproduct_unary}

We next give an explicit description of the operations of the coproduct $O+O'$ when $O'$ is an operad having only unary operations.

\begin{defn} Let $C$ be a collection and $M$ a set.  A \define{$(C,M)$-labelled planar
$n$-tree} is a $C$-$n$-tree together with a map $\ell \colon E\to M$ assigning a label in $M$ to each edge.  An \define{isomorphism} of $(C,M)$-labelled planar $n$-trees is an isomorphism of the underlying $C$-$n$-trees that preserves the labelling of edges.
\end{defn}

\begin{defn}
A \define{$(C,M)$-$n$-tree} is an isomorphism class of $(C,M)$-labelled planar $n$-trees. A \define{$(C,M)$-tree} is a $(C,M)$-$n$-tree for some $n$.
\end{defn}
\noindent
 We further make the following definition, closely linked to Definition \ref{defn:OTree}:
\begin{defn}
For any operad $O$ and any set $M$ we define an \define{$(O,M)$-$n$-tree} to be a $(U(O),M)$-$n$-tree. An \define{$(O,M)$-tree} is a $(U(O),M)$-tree.
\end{defn}

\noindent
The notion of subtree introduced in Definition \ref{defn:subtree} can be extended to $(C,M)$-trees, where the subtree inherits its labels from the original tree. More precisely, a \define{subtree} of a $(C,M)$-tree $(T,l)$ is a $(C,M)$-tree $(S,l|_{E_S})$ where $S$ is a subtree of $T$.
We can also extend the definition of permutation on subtrees to $(C,M)$-trees. Suppose that  $S$ consists of a single vertex and that $\inc_S$ consists of $k$ elements. The \define{permutation of $(S,l|_{E_S})$ by $\sigma\in S_k$} is the $(C,M)$-tree  $(S\cdot \sigma,l|_{E_S})$ where $S\cdot \sigma$ is the permutation of $S$ by $\sigma$ (see Definition \ref{defn:perm subtree}).\\

We are now ready to state the result that we need to prove Theorem \ref{thm:bijection}. For this we first give the following definition of equivalence relation:

\begin{defn}\label{defn:eq relation}
 Let $O$ and $O'$ be operads, and suppose that $O'$ only has unary operations. We say that two $(O,O'_1)$-trees are  \define{equivalent} if we can reach one from the other by a finite sequence of moves of this type: for any such tree $(T,\ell)$, we can replace any $(O,O'_1)$-subtree $(S,\ell|_{E_S})$ where $S$ has exactly one vertex $v$ labelled by $f \cdot \sigma$, with $\sigma \in S_k$ and $f \in O_k$, by the subtree obtained from $(S,\ell|_{E_S})$ by permuting $(S,\ell|_{E_S})$ by $\sigma$ and substituting the label of $v$ by $f$.
\end{defn}

\begin{theorem}\label{thm:generalized_phyl_trees}
Suppose $O$ and $O'$ are operads and $O'$ has only unary operations.  The operations of the coproduct $O+O'$ are in bijection with equivalence classes of $(O,O'_1)$-trees such that no unary vertex is labelled by $1_O$ and no internal edge is labelled by $1_{O'}$, where the equivalence relation is as in Definition~\ref{defn:eq relation}.
\end{theorem}

For example, if $\sigma = \left(\begin{array}{ccc} 1 & 2 & 3 \\ 2 & 1 & 3\end{array} \right)$:
\[
\begin{tikzpicture}[scale=1, every node/.style={scale=1}]
\node at (-1,1) {$\bullet$};
\node at (0,0) {$\bullet$};
\node at (0,1) {$\bullet$};
\node at (-0.2,1) {$s$};
\node at (0.2,0.6) {$\ell_2$};
\node at (-1.3,1) {$g$};
\node at (-1.4,1.7) {$\ell_5$};
\node at (-0.6,1.7) {$\ell_6$};
\node at (-0.5,0.7) {$\ell_1$};
\node at (-0.5,0) {$f\cdot \sigma $};
\node at (1.4,1.7) {$\ell_3$};
\node at (0.2,-0.5) {$\ell_4$};
\node at (2.5,0) {$\sim$};
\node at (4,1) {$\bullet$};
\node at (5,0) {$\bullet$};
\node at (4.5,0.7) {$\ell_2$};
\node at (5.2,0.6) {$\ell_1$};
\node at (6.4,1.7) {$\ell_3$};
\node at (5.2,-0.5) {$\ell_4$};
\node at (4.6,1.7) {$\ell_5$};
\node at (5.4,1.7) {$\ell_6$};
\node at (4.7,0) {$f$};
\node at (4.7,1) {$g$};
\node at (3.8,1) {$s$};
\node at (5,1) {$\bullet$};
\path[-,font=\scriptsize]
(-1,1) edge (-2,2)
(-1,1) edge (0,2)
(-1,1) edge (0,0)
(0,0) edge (0,1)
(0,0) edge (2,2)
(0,0) edge (0,-1)
(5,0) edge (4,1)
(5,0) edge (5,1)
(5,1) edge (4,2)
(5,1) edge (6,2)
(5,0) edge (7,2)
(5,0) edge (5,-1);
\end{tikzpicture}
\]

\begin{proof}
We use Theorem \ref{thm:coproduct}, which describes any operation of $O+O'$ as an equivalence class of $U(O) + U(O')$-trees.  There is an operad morphism
\[         \epsilon_O + \epsilon_{O'} \maps F(U(O) + U(O')) \to O+O' \]
sending $U(O) + U(O')$-trees to operations of $O+O'$.  This map is onto, and Theorem \ref{thm:coproduct} says when two $U(O) + U(O')$-trees are sent to the same operation of $O+O'$.  We can use this to describe operations of $O+O'$.

To begin, recall from Definition \ref{defn:CTree} that a $U(O) + U(O')$-tree is an isomorphism class of planar $n$-trees where each vertex with $k$ children is labelled by:
\begin{itemize}
\item an operation in $O_k$  if $k > 1$ or $k=0$;
\item  either an  operation in $O'_1$ or in $ O_1$ if $k = 1$.
\end{itemize}
We can draw these trees in a simpler way as follows.  If a vertex $v$ has one child and is labelled by $\ell \in O'_1$, we redraw it by using this operation $\ell$ to label the unique edge having $v$ as its source.   For example, consider the following $U(O) + U(O')$-tree, where we use $f_1,\dots , f_5$ to denote operations from $O$ and $\ell_1,\dots , \ell_4$ to denote operations in $O'_1$:
\[
\begin{tikzpicture}[scale=1.1]
\node at (-2,1){$\bullet$}; 
\node at (-2.25,1){$f_4$};
\node at (-1,1){$\bullet$};
\node at (-0.7,1){$1_{O'}$};
\node at (0,1){$\bullet$};
\node at (0.3,1){$f_1$};
\node at (1,1){$\bullet$};
\node at (1.3,1){$f_2$};
\node at (3,1){$\bullet$};
\node at (3.3,1){$\ell_1$};
\node at (-1,0){$\bullet$};
\node at (-1.3,0){$f_3$};
\node at (1,0){$\bullet$};
\node at (1.3,0){$1_{O'}$};
\node at (2.5,0){$\bullet$};
\node at (2.9,0){$\ell_2$};
\node at (1,-1){$\bullet$};
\node at (1.4,-1){$f_5$};
\node at (-2.5,2.3){$4$};
\node at (-2,2.3){$1$}; 
\node at (-1.5,2.3){$3$}; 
\node at (-1,2.3){$2$}; 
\node at (0,2.3){$8$}; 
\node at (0.8,2.3){$6$}; 
\node at (1.2,2.3){$5$}; 
\node at (3,2.3){$7$};
\node at (1,-2){$\bullet$};
\node at (1.3,-2){$\ell_3$};
\node at (1,-3.3){$0$};
\node at (0,-0.5){$\bullet$};
\node at (0.3,-0.3){$\ell_4$};

\path[-,font=\scriptsize]
(-2.5,2) edge (-2,1)
(-2,2) edge (-2,1)
(-1.5,2) edge (-2,1)
(-1,2) edge (-1,1)
(0,2) edge (0,1)
(0.8,2) edge (1,1)
(1.2,2) edge (1,1)
(3,2) edge (3,1)
(-2,1) edge (-1,0)
(-1,1) edge (-1,0)
(0,1) edge (-1,0)
(1,1) edge (1,0)
(3,1) edge (2.5,0)
(-1,0) edge (1,-1)
(1,0) edge (1,-1)
(2.5,0) edge (1,-1)
(1,-1) edge (1,-2)
(1,-2) edge (1,-3);
\end{tikzpicture}
\]
We redraw this as follows:
\[
\begin{tikzpicture}[scale=1.1]
\node at (-2,1){$\bullet$}; 
\node at (-2.25,1){$f_4$};
\node at (-1,1){$\bullet$};
\node at (-0.75,1.5){$1$};
\node at (0,1){$\bullet$};
\node at (0.3,1){$f_1$};
\node at (1,1){$\bullet$};
\node at (1.3,1){$f_2$};
\node at (3,1){$\bullet$};
\node at (3.2,1.5){$\ell_1$};
\node at (-1,0){$\bullet$};
\node at (-1.3,0){$f_3$};
\node at (1,0){$\bullet$};
\node at (1.25,0.45){$1$};
\node at (2.5,0){$\bullet$};
\node at (3.1,0.5){$\ell_2$};
\node at (1,-1){$\bullet$};
\node at (1.4,-1){$f_5$};
\node at (-2.5,2.3){$4$};
\node at (-2,2.3){$1$}; 
\node at (-1.5,2.3){$3$}; 
\node at (-1,2.3){$2$}; 
\node at (0,2.3){$8$}; 
\node at (0.8,2.3){$6$}; 
\node at (1.2,2.3){$5$}; 
\node at (3,2.3){$7$}; 
\node at (1,-2){$\bullet$};
\node at (1.2,-1.5){$\ell_3$};
\node at (1,-3.3){$0$};
\node at (0,-0.5){$\bullet$};
\node at (-0.25,-0.1){$\ell_4$};

\path[-,font=\scriptsize]
(-2.5,2) edge (-2,1)
(-2,2) edge (-2,1)
(-1.5,2) edge (-2,1)
(-1,2) edge (-1,1)
(0,2) edge (0,1)
(0.8,2) edge (1,1)
(1.2,2) edge (1,1)
(3,2) edge (3,1)
(-2,1) edge (-1,0)
(-1,1) edge (-1,0)
(0,1) edge (-1,0)
(1,1) edge (1,0)
(3,1) edge (2.5,0)
(-1,0) edge (1,-1)
(1,0) edge (1,-1)
(2.5,0) edge (1,-1)
(1,-1) edge (1,-2)
(1,-2) edge (1,-3);
\end{tikzpicture}
\]
Note we are writing $1$ for the operation $1_{O'}$ labelling the edges.  This redrawing process never loses information, so henceforth we draw $U(O) + U(O')$-trees in this simplified way.  

By Theorem \ref{thm:coproduct}, an operation in $O + O'$ is an equivalence class of $U(O) + U(O')$-trees.  In terms of our simplified style of drawing these trees,  the equivalence relation is generated by the following moves:

\begin{enumerate}
\item
Suppose a subtree consists of a vertex $v$ labelled by $f\in O_k$ with children $e_1,\dots e_k$, with each $s(e_i)$ having $n_i$ children of its own, and labelled by $f_i\in O_{n_i}$, and suppose that the edges $e_1,\dots , e_k$ are either not labelled or labelled by $1$. Then we can replace this subtree by the corolla with  $n_1 + \cdots + n_k$ leaves and it unique vertex labelled by $f\circ(f_1,\dots, f_k)$, and edge labels (if any) those of the original tree.  For example:
\[
\begin{tikzpicture}[scale = 0.8]
\node at (-1,0){$\bullet$};
\node at (-1.4,0){$f_1$};
\node at (1,0){$\bullet$};
\node at (1.4,0){$f_2$};
\node at (2.5,0){$\bullet$};
\node at (2.8,0){$f_3$};
\node at (1,-1){$\bullet$};
\node at (1.5,-1){$f$};
\node at (2.9,0.7) {$$};
\node at (0,-0.8) {$1$};
\node at (4.5,-0.5){$\leadsto$};
\path[-,font=\scriptsize]
(-1.5,1) edge (-1,0)
(-0.5,1) edge (-1,0)
(2.5,1) edge (2.5,0)
(-1,0) edge (1,-1)
(1,0) edge (1,-1)
(1,0) edge (1,1)
(1,0) edge (1.5,1)
(1,0) edge (0.5,1)
(2.5,0) edge (1,-1)
(1,-1) edge (1,-2);
\end{tikzpicture}
\qquad
\begin{tikzpicture}[scale = 0.8]
\node at (1,-1){$\bullet$};
\node at (2.8,-1){$f\circ(f_1,f_2,f_3)$};                
\path[-,font=\scriptsize]
(0.7,1) edge (1,-1)
(0.1,1) edge (1,-1)
(-0.5,1) edge (1,-1)
(1.3,1) edge (1,-1)
(1.9,1) edge (1,-1)
(2.5,1) edge (1,-1)
(1,-1) edge (1,-2);
\end{tikzpicture}
\]
(This picture does not show the whole tree, only the subtree being modified.  The sources of the edges at top can be vertices or leaves; the target of the edge at bottom can be a vertex or the root.)

\vskip 1em
\item Suppose the edges $e \maps u \to v$, $e' \maps v \to w$ are labelled by operations $\ell, \ell' \in O'_1$, and suppose that $v$ is the target of just one edge, namely $e$.  Then we can remove the vertex $v$, replace the edges $e$ and $e'$ by a single edge $f \maps u \to w$, and label this new edge by $\ell \circ\ell'$.  In pictures:
\[
\begin{tikzpicture}
\node at (-1,2){$$};
\node (1) at (0,2){$$};
\node at (0.2,1.5) {$\ell$};
\node (2) at (0,1) {$\bullet$}; 
\node (3) at (0,0){$\bullet$};
\node at (0.25,0.5) {$\ell'$};
\node (4) at (4,2){$$};
\node (5) at (4,0){$\bullet$};
\node at (2,1) {$\leadsto$};
\node at (4.5,1){$\ell \circ \ell'$};
\path[-,font=\scriptsize]
(0,2) edge (0,1)
(0,1) edge (0,0)
(4,2) edge (4,0);
\end{tikzpicture}
\]
(The source of the edge at top left can be a vertex or leaf; the target of the edge at bottom left can only be a vertex, not the root.)

\vskip 1em 
\item  Suppose a vertex $v$ is labelled by $1_O$. Then we can remove the label of the vertex:
\[
\begin{tikzpicture}
\node at (0.4,1) {$1_O$};
\node at (0,1) {$\bullet$}; 
\node at (4,1) {$\bullet$}; 
\node at (0,0){$$};
\node at (4,2){$$};
\node at (4,0){$$};
\node at (2,1) {$\leadsto$};
\path[-,font=\scriptsize]
(0,2) edge (0,1)
(0,1) edge (0,0)
(4,2) edge (4,0);
\end{tikzpicture}
\]
(The source of the edge at top can be a vertex or leaf; the target can be a vertex or the root.)
\vskip 1em
\item Suppose a vertex $v$ is unlabelled and is the target of just one edge $e \maps u \to v$ and suppose there is an edge $e' \maps v \to 0$.  Suppose $e$ is labelled by the identity $1$.  Then we can remove the vertex $v$ and replace the edges $e \maps u \to v$, $e' \maps v \to 0$ by a single edge $f \maps u \to 0$, which is unlabelled:
\[
\begin{tikzpicture}
\node at (0.4,1.5) {$1$};
\node at (0,1) {$\bullet$}; 
\node at (0,0){$$};
\node at (4,2){$$};
\node at (4,0){$$};
\node at (0,-0.3) {$0$};
\node at (4,-0.3) {$0$};
\node at (2,1) {$\leadsto$};
\path[-,font=\scriptsize]
(0,2) edge (0,1)
(0,1) edge (0,0)
(4,2) edge (4,0);
\end{tikzpicture}
\]
(The source of the edge at right can be a vertex or leaf; the target must be the root.)

\vskip 1em
\item We can add the label $1$ to any unlabelled edge not incident to the root: 
\[
\begin{tikzpicture}
\node at (0.2,1.5) {$$};
\node (2) at (0,1) {$\bullet$}; 
\node at (2,1.5) {$\leadsto$};
\node at (4,2){$$};
\node at (4,1){$\bullet$};
\node at (4.4,1.5){$1$};
\path[-,font=\scriptsize]
(0,1) edge (0,2)
(4,1) edge (4,2);
\end{tikzpicture}
\]
(The source of this edge can be a vertex or leaf; the target must be a vertex, not the root.)

\vskip 1em
\item Suppose a vertex $v$ is labelled by $f\cdot \sigma$, where $f\in O_k$ and $\sigma\in S_k$. Then we can permute its children by $\sigma^{-1}$ and substitute the label of $v$ by $f$:
\[
\begin{tikzpicture}
\node at (0,0) {$\bullet$};
\node at (0.5,0) {$f\cdot \sigma$};
\node at (-1,1.35) {$T_1$};
\node at (1,1.35) {$T_k$};
\node at (0,1) {$\cdots$};
\node at (2,0.5) {$\leadsto$};
\node at (4,0) {$\bullet$};
\node at (4.3,0) {$f$};
\node at (3,1.35) {${T_{\sigma^{-1}\negmedspace(1)}}$};
\node at (5,1.35) {${T_{\sigma^{-1}\negmedspace(k)}}$};
\node at (4,1) {$\cdots$};
\path[-,font=\scriptsize]
(0,0) edge (-1,1)
(0,0) edge (1,1)
(0,0) edge (0,-0.8)
(4,0) edge (3,1)
(4,0) edge (5,1)
(4,0) edge (4,-0.8);
\end{tikzpicture}
\]
(The target of the bottom edge can be a vertex or root. Here $T_1,\dots T_k$ denote the `full subtrees ending in $v$', that is,  $T_i$ is the subtree with edge incident to the root the $i$-th child of $v$ and $\inc_{T_i}$ consisting only of leaves and termini of the original tree.)

\end{enumerate}

\noindent Move (1) here corresponds to item (\ref{item:composition_2}) of Theorem \ref{thm:coproduct}.  Move (2) corresponds to item (\ref{item:composition_3}) of that theorem.  Moves (3) corresponds to item (\ref{item:unit_2}) of that theorem.  Moves (4) and (5) correspond to item (\ref{item:unit_3}). Move (6) corresponds to item (\ref{item:symmetry_2}).    Item (\ref{item:symmetry_3}) does not arise, since $O'$ has only unary operations.

If we repeatedly apply moves (1)--(5) in the forward direction, this process eventually terminates.  The resulting $U(O) + U(O')$-tree is independent of which order we apply these moves, thanks to  Newman's Lemma \cite{New}, also called the Diamond Lemma, which says that a terminating abstract rewriting system is confluent if it is locally confluent.  We call a $U(O) + U(O')$-tree \define{reduced} if it is the result of this process.

In our example, we obtain this reduced $U(O) + U(O')$-tree:
\[
\begin{tikzpicture}[scale=1.1]
\node at (3,1.2){$\ell_1\circ \ell_2$}; 

\node at (-3.1,1.2){$1$};
\node at (-1.95,1.2){$1$};
\node at (-1.35,1.2){$1$};
\node at (-0.8,1.2){$1$};
\node at (-0.15,1.2){$1$};
\node at (1.1,1.2){$1$};
\node at (1.85,1.2){$1$}; 

\node at (0.3,1){$$};
\node at (1.3,1){$$};
\node at (-1,0){$\bullet$};  
\node at (-2.2,0){$f_3\circ(f_4,f_1)$};
\node at (1,-1){$\bullet$};
\node at (1.6,-1){$f_5\circ f_2$};
\node at (-4.2,2.3){$4$};
\node at (-3,2.3){$1$}; 
\node at (-2,2.3){$3$}; 
\node at (-1,2.3){$2$}; 
\node at (0,2.3){$8$}; 
\node at (0.8,2.3){$6$}; 
\node at (1.8,2.3){$5$}; 
\node at (3,2.3){$7$}; 
\node at (1,-2){$\bullet$};
\node at (1.3,-1.6){$\ell_3$};
\node at (1,-3.3){$0$};
\node at (0.1,-0.2){$\ell_4$};

\path[-,font=\scriptsize]
(-4,2) edge (-1,0)
(-3,2) edge (-1,0)
(-2,2) edge (-1,0)
(-1,2) edge (-1,1)
(0.8,2) edge (1,-1)
(1.8,2) edge (1,-1)
(1,-1) edge (3,2)
(-1,1) edge (-1,0)
(0,2) edge (-1,0)
(-1,0) edge (1,-1)
(1,-1) edge (1,-2)
(1,-2) edge (1,-3);
\end{tikzpicture}
\]
Here is an example of the reduction process that illustrates subtleties concerning the edge incident to the root:
\[  
\begin{tikzpicture}[scale=0.8]
\node at (0,1){$\bullet$};
\node at (-0.4,1){$f$};
\node at (0,0){$\bullet$};
\node at (0.4,0.5){$1$};
\node at (-1,3){$1$};
\node at (-1,3){$1$};
\node at (0,3){$2$};
\node at (1,3){$3$};
\node at (0,-1.5){$0$};
\node at (3,1){$\stackrel{(4)}{\leadsto}$};
\node at (4,1){$$};
\path[-,font=\scriptsize]
(-1,2.5) edge (0,1)
(0,2.5) edge (0,1)
(1,2.5) edge (0,1)
(0,1) edge (0,-1);
\end{tikzpicture} 
\begin{tikzpicture}[scale=0.8]
\node at (0,1){$\bullet$};
\node at (-0.4,1){$f$};
\node at (0,0){$\bullet$};
\node at (0.4,0.5){$1$};
\node at (-1,3){$1$};
\node at (0,3){$2$};
\node at (1,3){$3$};
\node at (0,-1.5){$0$};
\node at (3,1){$\stackrel{(3)}{\leadsto}$};

\node at (-0.8,1.8){$1$};
\node at (0.1,1.8){$1$};
\node at (0.8,1.8){$1$};

\path[-,font=\scriptsize]
(-1,2.5) edge (0,1)
(0,2.5) edge (0,1)
(1,2.5) edge (0,1)
(0,1) edge (0,-1);
\end{tikzpicture} 
\begin{tikzpicture}[scale=0.8]
\node at (-2,0){};
\node at (0,1){$\bullet$};
\node at (-0.4,1){$f$};
\node at (-1,3){$1$};
\node at (0,3){$2$};
\node at (1,3){$3$};
\node at (0,-1.5){$0$};

\node at (-0.8,1.8){$1$};
\node at (0.1,1.8){$1$};
\node at (0.8,1.8){$1$};
\path[-,font=\scriptsize]
(-1,2.5) edge (0,1)
(0,2.5) edge (0,1)
(1,2.5) edge (0,1)
(0,1) edge (0,-1);
\end{tikzpicture} 
\]

In a reduced $U(O) + U(O')$-tree, the edge incident to the root is unlabelled.  Edges incident to leaves are labelled by operations in $O'_1$.  Edges incident to neither leaves nor the root are labelled by operations in $O'_1$ different from the identity. So, to turn our reduced $U(O) + U(O')$-tree into a $(O,O'_1)$-labelled tree, we apply the following rule.  Only one of these two cases will apply:

\begin{itemize}
\item
Suppose $v$ is the target of a single edge $e \maps u \to v$ and there is an edge $e' \maps v \to 0$. Suppose $e$ is labelled by the operation $\ell \ne 1_{O'}$ in  $ O'_1$.  Then we remove the vertex $v$ and replace the edges $e$ and $e'$ by a single edge $f \maps u \to 0$, which is labelled by $\ell$:
\[
\begin{tikzpicture}
\node at (0.2,1.5) {$\ell$};
\node at (0,1) {$\bullet$}; 
\node at (0,0){$$};
\node at (4.2,1){$\ell$};
\node at (0,-0.4){$0$};
\node at (4,-0.4){$0$};
\node at (2,1) {$\leadsto$};
\path[-,font=\scriptsize]
(0,2) edge (0,1)
(0,1) edge (0,0)
(4,2) edge (4,0);
\end{tikzpicture}
\]
\item 
Suppose $v$ is the target of more than one edge and there is an edge $e \maps v \to 0$.  Then label the edge $e$ by $1$  (that is, $1_{O'}$).  For example:
\[
\begin{tikzpicture}
\node at (0,1) {$\bullet$}; 
\node at (4,1) {$\bullet$}; 
\node at (0,-0.4){$0$};
\node at (4.3,0.5){$1$};
\node at (4,-0.4){$0$};
\node at (2,1) {$\leadsto$};
\path[-,font=\scriptsize]
(0,2) edge (0,1)
(-1,2) edge (0,1)
(1,2) edge (0,1)
(0,1) edge (0,0)
(4,2) edge (4,0)
(4,2) edge (4,1)
(3,2) edge (4,1)
(5,2) edge (4,1);
\end{tikzpicture}
\]
\end{itemize}
The result is no longer a $U(O)+U(O')$-tree since now every edge, even the edge incident to the root, is labelled with an operation of $O'$.  

In our running example, this rule produces the following $(O,O'_1)$-tree:

\[
\begin{tikzpicture}[scale=1.1]
\node at (3,1.2){$\ell_1\circ \ell_2$}; 
\node at (1.9,1.2){$1$}; 
\node at (1.15,1.2){$1$};
\node at (-0.15,1.2){$1$};
\node at (-0.75,1.2){$1$};
\node at (-1.35,1.2){$1$};
\node at (-1.9,1.2){$1$};
\node at (-3,1.2){$1$};
\node at (0.3,1){$$};
\node at (1.3,1){$$};
\node at (-1,0){$\bullet$};  
\node at (-2.2,0){$f_3\circ(f_4,f_1)$};
\node at (1,-1){$\bullet$};
\node at (1.5,-1){$f_5\circ f_2$};
\node at (-4.2,2.3){$4$};
\node at (-3,2.3){$1$}; 
\node at (-2,2.3){$3$}; 
\node at (-1,2.3){$2$}; 
\node at (0,2.3){$8$}; 
\node at (0.8,2.3){$6$}; 
\node at (1.8,2.3){$5$}; 
\node at (3,2.3){$7$}; 
\node at (1.2,-1.6){$\ell_3$};
\node at (1,-2.3){$0$};
\node at (0.1,-0.2){$\ell_4$};

\path[-,font=\scriptsize]
(-4,2) edge (-1,0)
(-3,2) edge (-1,0)
(-2,2) edge (-1,0)
(-1,2) edge (-1,1)
(0.8,2) edge (1,-1)
(1.8,2) edge (1,-1)
(1,-1) edge (3,2)
(-1,1) edge (-1,0)
(0,2) edge (-1,0)
(-1,0) edge (1,-1)
(1,-1) edge (1,-2);
\end{tikzpicture}
\]

One can check that this rule always gives a $(O,O'_1)$-tree with no internal edge labelled by $1$ (that is, $1_{O'}$) and no unary vertex labelled by $1_O$, and that it loses no information. Furthermore, it is easy to see that move (6) generates the equivalence relation in the statement of the lemma. Finally, one can check that every such equivalence class of $(O,O'_1)$-trees arises from a $U(O)+U(O')$-tree via this process.  Thus, we obtain the desired result.
\end{proof}

We are now finally able to prove Theorem \ref{thm:bijection}:


\vskip 1em
\noindent \textbf{Theorem 9.}  \textit{The $n$-ary operations in the phylogenetic operad are in one-to-one correspondence with phylogenetic $n$-trees.}

\begin{proof}  
The statement of this theorem is somewhat inadequate, because we really have a specific bijection between $n$-ary operations in $\Phyl$ and phylogenetic $n$-trees in mind.  By Lemma \ref{thm:generalized_phyl_trees} we know that there is a bijection between operations of $\Com+[0,\infty)$ and equivalence classes of $(\Com,[0,\infty))$-trees without  unary vertices and such that no internal edge is labelled by $0$.  Since the symmetric group action on $\Com$ is trivial, each such equivalence class $[T]$ consists of all the trees obtained from $T$ by varying its planar structure. Hence the bijection of Lemma \ref{thm:generalized_phyl_trees} sends operations of $\Com+[0,\infty)$ to phylogenetic trees.
\end{proof}

\section{The $\wco$ construction and the phylogenetic operad}
\label{section:w}

The $\wco$ construction was introduced by Boardman and Vogt  \cite{BV} to study homotopy invariant algebraic structures on topological spaces.    In their construction, Boardman and Vogt used elements of $[0,1]$ to label edges of trees, using the fact that this space becomes a commutative topological monoid under the operation
\[    x \star y = x + y - xy  \]
However, this topological monoid is isomorphic to the monoid $[0,\infty]$ introduced in the previous section:

\begin{lemma} \label{monoid iso}
There is an isomorphism 
$
\psi \maps ([0,\infty],+)\to ([0,1],\star).
$
\end{lemma}
\begin{proof}
We use an argument due to Trimble \cite{Tri}.   Note that 
\[ x\star y=1-(1-x)(1-y). \]
Thus, there is an isomorphism of topological monoids 
\[ \begin{array}{ccl} 
\phi \maps ([0,1],\star) &\to& ([0,1], \cdot) \\                                                              x & \mapsto& 1-x. 
\end{array}
\]
Further, the topological monoid $([0,1],\cdot)$ is isomorphic to $([0,\infty],+)$ via
the map
\[ \begin{array}{ccl} 
([0,1],\cdot) &\to& ([0,\infty],+) \\
 x&\mapsto& -\ln x.
\end{array}
\qedhere
\]
\end{proof}

\noindent So, we can freely adapt Boardman and Vogt's original construction by using elements of $[0,\infty]$ instead of $[0,1]$. 
For this, we first recall the definition of $O$-trees, which we  introduced in Section \ref{sec:free}. Given two planar $n$-trees with  $k$-ary vertices labelled by $k$-ary operations of $O$, we say that they are isomorphic if there is an isomorphism of their underlying planar $n$-trees such that the labelling of each vertex in the first tree equals the labelling of the corresponding vertex in the second. Then an $O$-n-tree is an isomorphism class of planar $n$-trees with  $k$-ary vertices labelled by $k$-ary operations of $O$.

Given an operad $O$,  we define a new operad $\wco(O)$, where for any natural number $n=0,1,2,\dots $ an element of $\wco(O)_n$ is an equivalence class of pairs $(T,l)$, where
\begin{enumerate} 
\item $T$ is an $O$-$n$-tree
\item  a \define{length map} $\ell \maps E\rightarrow [0,\infty]$ such that external edges are mapped to $\infty$. For any $e\in E$ we call $l(e)$ the \define{length} of $e$.
\end{enumerate} 
The equivalence relation on these pairs is generated by the following moves.  
For any pair $(T,l)$ in $\wco(O)_n$:
\begin{enumerate}
\item \label{item: w identity} any subtree of $T$ consisting of one vertex labelled by $1_O \in O_1$ together with the two adjacent edges labelled by $\ell_1$ and $\ell_2$ can be replaced by an edge labelled by $\ell_1+\ell_2$:
\[
\begin{tikzpicture}
\node (1) at (0,2){};
\node at (-0.3,1.5) {$\ell_1$};
\node (2) at (0,1) {$\bullet$}; 
\node at (0.4,1) {$1_O$};
\node (3) at (0,0){};
\node at (-0.3,0.5) {$\ell_2$};
\node at (2,1) {$\sim$};
\node at (3.3,1){$\ell_1+ \ell_2$};
\path[-,font=\scriptsize]
(0,2) edge (0,1)
(0,1) edge (0,0)
(4,2) edge (4,0);
\end{tikzpicture}
\]

\item \label{item: w symmetry}  any subtree $S$ of $T$ formed by a vertex $v$ in $T$ of arity $r$ labelled by $f\cdot\sigma$, where $\sigma \in S_r$ and $f\in O_r$, 
can be substituted with the subtree obtained from $S$ by permuting $S$ by $\sigma$ and substituting the label of $v$ by $f$:
\[
\begin{tikzpicture}
\node (1) at (-1,1) {$T_1$};
\node  at (0,1) {$\dots$};
\node (2) at (1,1) {$T_r$};
\node (3) at (0,0) {$\bullet$};
\node at (0.8,0) {$f\cdot\sigma$};
\node at (3,0.5) {$\sim$};
\node (4) at (4,1) {$T_{\sigma^{-1}(1)}$};
\node at (5,1) {$\dots$};
\node (5) at (6,1) {$T_{\sigma^{-1}(r)}$};
\node (6) at (5,0) {$\bullet$};
\node at (5.4,0) {$f$};
\path[-,font=\scriptsize]
(1) edge (0,0)
(2) edge (0,0)
(4) edge (5,0)
(5) edge (5,0);
\end{tikzpicture}
\]
\item \label{item: w comp} any edge of length $0$ may be shrunk away by composing the labels of its adjacent vertices  using the composition in $O$.
\end{enumerate}
The space $\wco(O)_n$ inherits a topology from the spaces $O_0,\dots, O_n$ and from $[0,\infty]$.  Let $(T,\ell)$ be a pair of an $O$-$n$-tree and a length map, and denote the underlying isomorphism class of planar $n$-trees of $T$ by $\lambda$.  Given $\lambda$, $(T,\ell)$ is uniquely determined by the labels assigned to edges and vertices of $\lambda$, so we can see it as a point of the set
\[
\prod_{j} O_j^{m_j}\times [0,\infty]^{r}
\]
where $m_j$ is the number of vertices of $\lambda$ having arity $j$ and $r$ is the number of internal edges of $\lambda$.  We endow this set with the product topology. Taking the disjoint union over all isomorphism classes of planar $n$-trees and taking the quotient of the resulting space by the above equivalence relations, an element of $\wco(O)_n$ is a point in the following topological space: 
\[
\begin{fracinline}{\left(\coprod_{\lambda}\prod_{j_\lambda} O_{j_\lambda}^{m_{j_\lambda}}\times [0,\infty]^{r_\lambda}\right)}{\sim}\end{fracinline}
\]
where we take the topology to be the quotient topology. 
In \cite{BV} the space $\wco(O)_n$ is described as 
\[
\begin{fracinline}{\left(\coprod_{\lambda}\prod_{j_\lambda} O_{j_\lambda}^{m_{j_\lambda}}\times [0,\infty]^{r_\lambda}\times S_n\right)}{\sim}\end{fracinline}
\]
where $S_n$ is endowed with the discrete topology. This is because they consider $\lambda$ as being the underlying graph of a tree, while for us $\lambda$ is the underlying isomorphism class of planar $n$-trees, and we define an $n$-tree to have leaves labelled by $1,\dots , n$. 

Given two $O$-trees with length maps $(T,\ell)$ and $(T',\ell')$, we define their partial composite $(T,\ell)\circ_i(T',\ell')$ as the partial composite $T\circ_iT'$ of the underlying $O$-trees, together with the length function that sends every edge to its image under $l$ or $l'$ and the new internal edge that arises from the grafting to $\infty$. The unit for this composition is given by the $1$-tree without vertices and unique edge labelled by $\infty$. Similarly, for  any $O$-$n$-tree $T$ and $\sigma\in S_n$  we define $(T,l)\cdot\sigma=(T\cdot \sigma, l)$. These operations are easily seen to be well-defined on equivalence classes and  to be continuous, and thus endow $\wco(O)$ with the structure of a topological operad. 

The operad $\wco(O)$ is closely related to the coproduct $O+[0,\infty]$. To see this, 
recall that by Lemma \ref{thm:generalized_phyl_trees} operations of $O+[0,\infty]$ can be identified with equivalence classes of $O$-trees with no unary vertex labelled by $1_O$, and edges labelled by numbers in $[0,\infty]$ such that internal edges are not labelled by zero, where the equivalence relation is given by the symmetric group action on operations of $O$.  

Thus an operation of $O+[0,\infty]$ is in $\wco(O)$ if and only if it is an equivalence class of an $O$-$n$-tree with all external edges labelled by $\infty$.  From this we see that the unit of $O+[0,\infty]$ is not in $\wco(O)$, so $\wco(O)$ fails to be a suboperad of $O+[0,\infty]$.  However, $\wco(O)$ is a non-unital suboperad of $O+[0,\infty]$, and its unit is an idempotent of $O+[0,\infty]$:

\begin{theorem}
\label{thm:W-construction}
The inclusions $\iota_n\maps \wco(O)_n\to O+[0,\infty]_n$  induce a morphism of non-unital topological operads. Moreover, the spaces  $W(O)_n$ and $O+[0,\infty]_n$ are homotopy equivalent if $n\ne 1$.
\end{theorem}

\begin{proof}
By the previous discussion, and the remarks at the end of Section \ref{sec:coproduct}, it is easy to see that the inclusion is continuous. The contracting homotopy 
\[
F\maps  O+[0,\infty]_n \times [0,1]\to  O+[0,\infty]_n 
\]
is defined as follows:
\[
((T,l),t)\mapsto (T,\hat{l}_t)
\]
where 
\[
\hat{l}_t\maps E\to [0,\infty]\maps e \mapsto \begin{cases}
l(e), \text{ if $e$ is an internal edge}\\
\alpha((1-t)\alpha^{-1}(l(e)), \text{ otherwise.}
\end{cases}
\]
with $\alpha\maps ([0,1],\star)\to ([0,\infty],+)$  the inverse to the isomorphism of Lemma \ref{monoid iso}.
\end{proof}

On the other hand, the operads $O+[0,\infty]$ and $O+[0,\infty)$ are closely related to $\Com_+$. To see how, we first need to  choose a convenient category of topological spaces, such as the category of compactly generated Hausdorff spaces \cite{Str}. 
Note that the spaces $\Phyl_n$ are metric spaces and are thus compactly generated Hausdorff spaces. We consider the model structure on this category in which weak equivalences are weak homotopy equivalences and fibrations are Serre fibrations. This model structure induces a model structure on the category of operads in which weak equivalences and fibrations are given by  pointwise weak equivalences and fibrations, respectively \cite{BM1}. Berger and Moerdijk proved that for this model structure, if $O$ is a $\Sigma$-cofibrant and well-pointed operad, $\wco(O)$ gives a cofibrant resolution of $O$  \cite[Thm.\ 5.1]{BM2} and further they showed that in this case   algebras over $\wco(O)$ are invariant under homotopy in the sense of Boardman and Vogt \cite[Thm.\ 3.5]{BM1}.    So, we make the following definition:

\begin{defn}
A morphism of topological operads $f\maps O\to O'$ is a \define{weak equivalence} if for every $n$ the map $f_n$ is a weak homotopy equivalence. We say that $f$ is a \define{homotopy equivalence} if there exists $g\maps O'\to O$ such that $g_n$ is a homotopy inverse to $f_n$ for every $n$.
\end{defn}

Now, since the intervals $[0,\infty]$ and $[0,\infty)$ are contractible,  we have:

\begin{proposition}
\label{C:contra}
Suppose that $O$ is an operad in which every space $O_n$ is contractible. Then $O+[0,\infty]$ and $O+[0,\infty)$ are both homotopy equivalent to the terminal operad, $\Com_+$.
\end{proposition}

\begin{proof}
First note that the underlying topological spaces of both operads $O+[0,\infty]$ and $O+[0,\infty)$ are contractible: the constant  map taking $O+[0,\infty)_n$, respectively $O+[0,\infty]_n$, to the one-point space consisting of the equivalence class of the $n$-corolla with all edges labelled by $0$  exhibit this one-point space as a deformation retract of $O+[0,\infty)_n$ and $O+[0,\infty]_n$.  Any operad with a one-point space in every arity is canonically isomorphic to the operad $\Com_+$, and furthermore the constant maps are easily seen to extend to morphisms of operads.  Therefore, both operads are homotopy equivalent to $\Com_+$.
\end{proof}

We thus have the following commutative diagram:
\[
\begin{tikzpicture}
\node (1) at (1,1) {$O+[0,\infty)$};
\node (2) at (4,1) {$O$};
\node (3) at (1, -1) {$O+[0,\infty]$};
\path[->]
(1) edge node[above] {$\alpha$}(2)
(1) edge (3)
(3) edge node[below] {$\beta$} (2);
\end{tikzpicture}
\]
in which the morphisms $\alpha$ and $\beta$ are weak equivalences, and hence so is the inclusion of $O+[0,\infty)$ in $ O+[0,\infty]$, by the $2$-out-of-$3$ property.

In conclusion, suppose we have a Markov process on a finite set $X$.
We saw in Theorem \ref{thm:coalgebra_of_phyl} that there is a unique way to extend this to a coalgebra of the phylogenetic operad.  In Theorem \ref{thm:extension} we saw that this can be further extended to a coalgebra of $\Com+[0,\infty]$.  We thus have the 
following commutative diagram:  
\[ 
\begin{tikzpicture}
\node (1) at (-2,2) {$\Com+[0,\infty) = \Phyl$};
\node (2) at (2,2) {$\Coend(\mathbb{R}^X)$};
\node (3) at (-6,0) {$\wco(\Com)$};
\node (4) at (-2,0) {$\Com+[0,\infty]$};
\node (5) at (-6,2) {$\Com$};
\path[->]
(1) edge node[above]{$$}(2)
(1) edge (4)
(3) edge node[below]{$\iota$}(4)
(1) edge node[above]{$\alpha$}(5)
(3) edge (5)
(4) edge node[below]{$\beta$} (5);
\draw[->](4) -- node[below] {$$} (2);
\end{tikzpicture}
\]
Except for the non-unital inclusion $\iota$, all the arrows are operad homomorphisms, and those in the two triangles at left are weak equivalences.  For further explorations of operads related to the phylogenetic operad, see the work of Devadoss and Morava \cite{DM1,DM2}.

\subsection*{Acknowledgements}

We thank the denizens of the $n$-Category Caf\'e for their help, especially Todd Trimble
for preparing a proof of a generalized version of Lemma \ref{monadic}, and Richard Garner for pointing out that the case we needed could already be found in the work of Boardmann and Vogt. We also thank  the anonymous referee for his/her many helpful comments.  The second author thanks the ETH Z\"urich for supporting her visit to U.\ C.\ Riverside, during which part of this work was carried out.  She also thanks the Mathematics Department of U.\ C.\ Riverside for their hospitality.

\appendix

\section{The topology on $\Phyl_n$}\label{A:topology of phyl}
Here we   provide a proof of Theorem  \ref{thm:homeo}, in which we related the topology on the space of $n$-ary operations of the phylogenetic operad with the topology on the space $\T_n$ of metric $n$-trees introduced in \cite{BHV}:

\begin{customthm}{11} 
\label{thm:homeo app}
For every $n \ne 1$ there is a homeomorphism 
\[           \Phyl_n \cong \T_n  \times [0,\infty)^{n+1}  ,\]
and $\Phyl_1 \cong \T_1 \times [0,\infty)$. 
\end{customthm}  

\noindent To prove this, we give an explicit description of the topology on $\Phyl_n$ in Lemma \ref{lem:topology phylogenetic trees}. We first need to introduce some notation:

\begin{defn}\label{defn:basis phyl}
Let $T$ be any  isomorphism class of $n$-trees with no vertices of arity $0$ or $1$.
We define $U_T$  as follows:
\begin{enumerate}
\item If $T$ is an isomorphism class of $1$-trees, we let  $U_T$  be an open subset of $[0,\infty)$.
\item  If $T$ is an isomorphism class of binary $n$-trees, we have 
\[
U_T=U_1\times \dots \times U_{n-2} \times V_1\times \dots \times V_{n+1}
\]
 with $U_i$ an open subset of $(0,\infty)$ for $i=1,\dots , n-2$  and $V_j$ an open subset of $[0,\infty)$ for $j=1,\dots , n+1$.
\item Otherwise,  $T$ is an isomorphism class of $n$-trees with $k<n-2$ internal edges.
Consider the binary $n$-trees such that by contracting some of their internal edges $e_1,\dots, e_m$ we obtain $T$, where $m=n-2-k$; denote by $\widetilde{T}_1,\dots , \widetilde{T}_a$  the isomorphism classes of such binary trees.   We then define
\[
U_T=U_{T_0}\sqcup U_{\widetilde{T}_1}\sqcup \dots \sqcup U_{\widetilde{T}_a}.
\] 
Here
\[
U_{T_0}=U_1\times \dots \times U_k\times V_1\times \dots \times V_{n+1}
\]
where $U_i$ is an open subset of $(0,\infty)$ for $i=1,\dots , k$ and $V_j$ is an open subset of $[0,\infty)$ for $j=1,\dots, n+1$.  Furthermore, for any $l=1,\dots , a$, 
  \[
 U_{\widetilde{T_l}}=U_{T_0}\times U_{e_1}\times \dots \times U_{e_m}
 \] 
where each set  $U_{e_i}$ is of the form $(0,r_i)$ for some $r_i$ in $(0,\infty)$  for $i=1,\dots , m$. 
\end{enumerate}
\end{defn}

\begin{lemma}\label{lem:topology phylogenetic trees}
The bijection $\psi\maps \Com+[0,\infty)_n\to \Phyl_n$ of Theorem \ref{thm:bijection} endows the space of  phylogenetic $n$-trees with the topology whose basis is given by sets of the  form
\[
\bigsqcup_T U_T
\]
where $T$ ranges over isomorphism classes of $n$-trees with no unary vertices, and the sets $U_T$ are as in Definition \ref{defn:basis phyl}.

\end{lemma}

\begin{proof}
First note that a basis for the topology on $F(\Com+[0,\infty))_n$ is given by sets of the form  
\[
\bigsqcup_{T} W_T
\]
where $T$ ranges over isomorphism classes of planar $n$-trees, and for $T$ having $k$ unary vertices and $m_j$ vertices of arity $j\geq2$, $W_T$ is an open set in 
 \[
 \bigsqcup_{i=0}^k\left([0,\infty)^i\times \Com_1^{k-i}\right)\times \prod_{j}\Com_j^{m_j}.
 \]
The sets of the form $\bigsqcup_T U_T$ as defined in Definition \ref{defn:basis phyl}  are easily seen to satisfy the properties of a basis.

First we show that these sets are open in the quotient topology, namely that for any such set $U=\bigsqcup_T U_T$ its preimage $\epsilon^{-1}\circ \psi^{-1}(U)$ is open in $F(\Com+[0,\infty))_n$, where we write $\epsilon$ instead of $\epsilon_\Com+\epsilon_{[0,\infty)}$.  Let $x\in \epsilon^{-1}\circ \psi^{-1}(U)$. Then there exists $z\in U$ such that $x\in \epsilon^{-1}\circ \psi^{-1}(z)$. 
 Let $\ell_1,\dots , \ell_k$ for ($k\leq n-2$) be the labels of the internal edges of $z$, and $h_1,\dots,h_{n+1}$ the labels of the external edges.  
 To describe the elements of the 
set $\epsilon^{-1}\circ \psi^{-1}(z)$ we introduce the following notation:
 let $\widetilde{z}\in \epsilon^{-1}\circ \psi^{-1}(z)$ be  the unique (up to non-planar isomorphism) $U(\Com)+U([0,\infty))$-tree that is obtained from $z$ by substituting every edge labelled by $\ell$ with the $1$-corolla with its unique vertex labelled by $\ell$ and choosing any planar structure for $\widetilde{z}$.  Recall that we denote by $f_j$ the unique operation of $\Com_j$, for $j$ any natural number. Now  we can obtain all elements of $\epsilon^{-1}\circ \psi^{-1}(z)$ from $\widetilde{z}$ through the following moves:
 
\begin{enumerate}[label=(\alph*)]
\item every subtree of $\widetilde{z}$ consisting of a $2$-ary vertex $v$ labelled by $f_2$ is substituted by a subtree with one vertex of arity $2$ labelled by $f_2$ and $u_v\geq 0$ unary vertices labelled by $f_1$
 \item every subtree of $\widetilde{z}$ consisting of a $j$-ary vertex $v$ labelled by $f_j$ ( for $j>2$) is substituted by a  subtree with $r_v\geq1$ vertices of arity $1\leq j_1,\dots , j_{r_v}\leq j$ labelled by $f_{j_1},\dots ,f_{j_{r_v}}$ with 
 $\sum_{i=1}^{r_v}(j_i -1)+1=j$ and $m_v\geq 0$ unary vertices labelled by $0\in [0,\infty)$
 \item every subtree of $\widetilde{z}$ consisting of a unary vertex $v$ labelled by $\ell\in [0,\infty)$ is substituted by a subtree with $d_\ell\geq 1$ unary vertices labelled by $\ell_1,\dots , \ell_{d_\ell}\in [0,\infty)$ with $\sum_{i=1}^{d_\ell}\ell_i=\ell$ and $m_v\geq 0$ unary vertices labelled by the identity of $\Com$
 \item choose a planar structure for the resulting tree.
 \end{enumerate}

If $y$ is a $U(\Com)+U([0,\infty))$-tree that was obtained from $\widetilde{z}$ through moves (a)--(d), then we also say that its underlying isomorphism class of planar trees $H$ was obtained from $\widetilde{z}$  through moves (a)--(d).  Similarly, if $T$ is the underlying isomorphism class of trees of $z$ we denote by $\widetilde{T}$ the underlying isomorphism class of planar trees of $\widetilde{z}$, and we say that $H$ is obtained from $\widetilde{T}$ through moves (a)--(d).

We next need to introduce some notation.  For $n\ge 1$ let $\phi_n$ be the continuous map
\[
\begin{array}{ccl}		\phi_n\maps [0,\infty)^{n+1} &\to& [0,\infty) \\ (x_1, \dots, x_{n+1})& \mapsto& x_1 + \cdots + x_{n+1} .
\end{array}
\]
We also set $\phi_0$ be the identity on $[0,\infty)$.   If $W$ is open, then $\phi_n^{-1}(W)$ is an open subset of $[0,\infty)^{n+1}$.  

We now consider three cases:
\begin{enumerate}
\item $z\in U_T$, with $U_T$ as in item (1) of Definition \ref{defn:basis phyl}.
\item $z\in U_T$, with $U_T$ as in item (2) of Definition \ref{defn:basis phyl}.
\item $z\in U_T$, with $U_T$ as in  item (3) of Definition \ref{defn:basis phyl}.
\end{enumerate}
In the first case we have
$x\in \phi^{-1}_n(U_T)$, with $U_T$ an open subset of $[0,\infty)$, so $\phi^{-1}_n(U_T)$ is open in $F(\Com+[0,\infty))_1$.

In the second case, $z$ has $n-2$ internal edges. We then have
		\[
		x\in  
		\; \bigsqcup_H
		   \;\prod_{p=1}^{n-2} \phi^{-1}_{d_{\ell_p}}(U_p)\times
		  \prod_{q=1}^{n+1} \phi^{-1}_{d_{h_q}}(V_q) \times
		  \prod_{r=1}^{n-1}\left(\Com_2\times \Com_1^{u_{v_r}}\right).
		\]
Here $H$ ranges over isomorphism classes of planar $n$-trees obtained 
from $\widetilde{z}$ by moves (a), (b), and (d).  
Thus, the disjoint union is taken over all numbers  $d_{\ell_1},\dots , d_{\ell_{n-2}}, d_{h_1},\dots d_{h_{n+1}}\geq 0$ and $u_{v_1},\dots , u_{v_{n-1}}\geq 0$ for $v_1,\dots , v_{n-1}$ the $2$-ary vertices  of $z$. 
This set is open in $F(\Com+[0,\infty))_n$ and is contained in $\epsilon^{-1}\circ \psi^{-1}(U_T)$ and therefore in $\epsilon^{-1}\circ \psi^{-1}(U)$.

For the third case, we have that $z$ has $k<n-2$ internal edges and $k+1$ vertices. Let $b$ be the number of $2$-ary vertices of $z$ and $c=k+1-b$ the number of 		vertices of arity greater than two. We then have:	
	\begin{alignat}{2}
		x \; \in  
		 \;\bigsqcup_H
		 \; &\notag \prod_{p=1}^{k} \phi^{-1}_{d_{\ell_p}}(U_p)\times
		  \prod_{q=1}^{n+1} \phi^{-1}_{d_{h_q}}(V_q) \times
		  \prod_{r=1}^{b}\left (\Com_2\times \Com_1^{u_{v_r}}\right)\times\\
		 &\notag   \prod_{i=1}^{c} \left ( \phi^{-1}_{m_{w_i}}([0,\delta_{w_i}))\times
		 \prod_{s=1}^{r_{w_i}}
		 \Com_{j_{i_s}} \right)
		\end{alignat}				
Here $H$ ranges over isomorphism classes of planar $n$-trees obtained from $\widetilde{z}$ by moves (a)--(d).  Thus, the disjoint union is taken over all numbers  $d_{\ell_1},\dots , d_{\ell_{k}}, d_{h_1},\dots d_{h_{n+1}}\geq 0$ and $u_{v_1},\dots , u_{v_{b}}\geq 0$ for $v_1,\dots ,v_{b}$ the $2$-ary vertices  of $z$, and further for $i=1,\dots , c$ and $w_i$ a vertex of $z$ with arity $2<j_{w_i}$ 
the numbers $r_{w_i}\geq 1$, $m_{w_i}\geq 0$ and $1\leq j_{i_1},\dots, j_{i_{r_{w_i}}}\leq j_{w_i} $ such that $\sum_s (j_{i_s} - 1)+1=j_{w_i}$.
This is an open set in $F(\Com+[0,\infty))_n$, and furthermore by choosing the numbers $\delta_{w_i}$ appropriately, one has that this set is contained in $\epsilon^{-1}\circ \psi^{-1}(U_T)\subseteq \epsilon^{-1}\circ\psi^{-1}(U)$. Therefore the sets of the form $\bigsqcup_T U_T$ as defined in Definition \ref{defn:basis phyl} are open in the quotient topology.

It remains to show that the topology induced by these sets is the quotient topology. So we have to show that if $\epsilon^{-1}\circ \psi^{-1}(U)$ is open in $F(\Com+[0,\infty))_n$, then $U$ is open in $\Phyl_n$. We prove this by contradiction, namely we show that if $U$ is not open in $\Phyl_n$, then $\epsilon^{-1}\circ \psi^{-1}(U)$ is not open in $F(\Com+[0,\infty))_n$.  So suppose that $U\subseteq \Phyl_n$ is not open.
First note that if $T$ is an isomorphism class of  $n$-trees with $k$ internal edges and $m_j$ vertices of arity $j$, then the phylogenetic $n$-trees whose underlying isomorphism class of $n$-trees is $T$ are points in this space:
\[
\mathcal{U}_T =(0,\infty)^k\times [0,\infty)^{n+1}\times \prod_{j}\Com_j^{m_j}.
\]
Thus we can write $\Phyl_n$ as the space
\[
\Phyl_n=\bigsqcup_T \mathcal{U}_T
\]
where $T$ ranges over isomorphism classes of  $n$-trees, and we can write $U$ as 
\[
U=\bigsqcup_T (U \cap \mathcal{U}_T).
\]
Therefore for at least one $T$ the set $U\cap\mathcal{U}_T$ is not open. 
Since the $\Com_j$ are one-point sets, we must have that 
\[
U\cap\mathcal{U}_T=\bigcup_i V_i
\]
with 
\[
V_i=V_{i_1}\times \dots \times V_{i_k}\times\widetilde{V}_{i_1}\times \dots \times \widetilde{V}_{i_{n+1}}\times \prod_j \Com_j^{m_j}
\]
where $V_{i_s}\subseteq (0,\infty)$ for $s=1,\dots , k$ and $\widetilde{V}_{i_t}\subseteq[0,\infty)$ for $t=1,\dots , n+1$, and at least one of the $V_{i_s}$ or $\widetilde{V}_{i_t}$ is not open. 
Furthermore,  supposing that $k=n-2$, we have
\begin{alignat}{2}
\epsilon^{-1}\circ\psi^{-1}(U\cap\mathcal{U}_T)
&\notag= \bigcup_i \epsilon^{-1}\circ \psi^{-1}(V_i)\\  
&\notag= \bigcup_i \bigsqcup_{H}\phi_{d_1}^{-1}(V_{i_1})\times \dots \times \phi_{d_k}^{-1}(V_{i_k})\times \phi_{d_{k+1}}^{-1}(\widetilde{V}_{i_1})\times \dots \times \phi_{d_{k+n+1}}^{-1}(\widetilde{V}_{i_{n+1}})\\
&\notag \quad \times \prod_{r=1}^{n-1}\Com_2\times \Com_1^{u_{v_r}}.
\end{alignat}
Here $H$ ranges over  isomorphism classes of planar  $n$-trees obtained from $\widetilde{T}$ through moves $(a)$, $(b)$ and $(d)$. Thus, the coproduct is taken over all numbers $d_1,\dots, d_{k+n+1}\geq 0$ and $u_{v_1},\dots , u_{v_{n-1}}\geq 0$ for $v_1,\dots , v_{n-1}$ the $2$-ary vertices of $T$.
The case in which $k<n-2$ is similar. Now, suppose that any of the $V_{i_s}$ or $\widetilde{V}_{i_t}$ is not open. Then for $d_s=0$ and $d_t=0$ the set $\phi^{-d_s}(V_{i_s})$ or $\phi^{-d_t}(\widetilde{V}_{i_t})$ is not open, and so $\epsilon^{-1}\circ \psi^{-1}(U\cap \mathcal{U}_T)$ is not open, and hence neither is  $\epsilon^{-1}\circ \psi^{-1}(U)$. This completes the proof.   \end{proof}

\begin{proof}[Proof of Theorem \ref{thm:homeo app}]  
The claim is valid for $n=0$, since there are no metric $0$-trees, nor any phylogenetic $0$-trees.
A phylogenetic $1$-tree must have just one edge, labelled by a number in $[0,\infty)$, while the unique metric $0$-tree has its one edge labelled by zero. Thus there is a bijection between  $\Phyl_1$ and  $\T_1  \times [0,\infty)$, and by the explicit description of the topology on $\Phyl_1$ in Lemma \ref{lem:topology phylogenetic trees} it follows that this bijection is a homeomorphism.

For $n>1$
a phylogenetic $n$-tree gives a metric $n$-tree together with 
an $(n+1)$-tuple of lengths in $[0,\infty)$, namely the lengths labelling the 
external edges of the phylogenetic tree.  Conversely, a metric $n$-tree together 
with an $(n+1)$-tuple of lengths in $[0,\infty)$ gives a phylogenetic tree with these
lengths labelling its external edges.  We thus have a specific bijection between operations of $\Phyl_n$ and elements of  $\T_n  \times [0,\infty)^{n+1}$. We denote this bijection by 
\[
f\maps \Phyl_n\to \T_n  \times [0,\infty)^{n+1}.
\]

We now show that this assignment is a homeomorphism.  By Lemma \ref{lem:topology phylogenetic trees} we know that a basis for the topology on the set of phylogenetic $n$-trees is given by sets of the form $\bigsqcup_T U_T$ where $T$ is an isomorphism class of $n$-trees with no unary vertices and $U_T$ is described by item (2) or (3) in Definition~\ref{defn:basis phyl}. On the other hand, a basis for the topology on $\T_n$ is given by the balls $B(x,\epsilon)=\{y\in \T_n | d(x,y)< \epsilon\}$ for any $\epsilon>0$ and any  $x\in \T_n$. 

We first show that the bijection $f\maps \Phyl_n\to \T_n\times [0,\infty)^{n+1}$ is continuous.  Let $B\subseteq \T_n\times [0,\infty)^{n+1}$ be a basic open set. Then $B$ is of the form $W\times V$ with $W$ open in $\T_n$ and $V$ open in $[0,\infty)^{n+1}$. Let $x\in W$ and $y\in V$. Then there exists  a ball $B(x,\epsilon)$ such that $x\in B(x,\epsilon)\subseteq U$, and an open rectangle $r_y$ such that $y\in r_y \subseteq V$.

First, suppose that $x$ lies in the interior of an $(n-2)$-dimensional orthant.  The orthant corresponds to the isomorphism class of some binary $n$-tree $T$.
Let $\ell_1,\dots , \ell_{n-2}$ be the labels of the internal edges of $x$. Then for $\delta$ small enough, the set
\[
R=(\ell_1-\delta,\ell_1+\delta)\times \dots \times (\ell_{n-2}-\delta,\ell_{n-2}+\delta)\times r_y
\]
is such that 
\[
\begin{array}{ccl}
f^{-1}((x,y))
&\in & \{T\}\times R\\
& \subseteq & f^{-1}(B(x,\epsilon)\times r_y)\\
& \subseteq & f^{-1}(B)
\end{array}
\]
and $\{T\}\times R$ is a basic open set of $\Phyl_n$ satisfying item (2) of Definition \ref{defn:basis phyl}.

Now suppose that $x$ lies on the boundary of one or more $(n-2)$-dimensional orthants. Then this boundary corresponds to the isomorphism class of an $n$-tree $T$ with $k$ internal edges. Let $\ell_1,\dots , \ell_k$ denote the labels of the internal edges of $x$, and let $\widetilde{T}_1,\dots , \widetilde{T}_a$ be the isomorphism classes of binary $n$-trees corresponding to the $a$ neighboring $(n-2)$-dimensional orthants. 

Let $R$ denote the open rectangle
\[
(\ell_1-\delta,\ell_1+\delta)\times \dots \times (\ell_k-\delta,\ell_k+\delta)\times r_y.
\]
Then for $\delta$ small enough the set
\[
Q=\left(\{T\}\times R \right)\;\cup
 \bigcup_{i=1}^a
\left(\{\widetilde{T}_i\}\times R \times \underset{n-2-k\text{ times}}{\underbrace{(0,\delta)\times \dots \times (0,\delta)}}\right).
\]
has
\[
\begin{array}{ccl}
f^{-1}((x,y)) &\in& Q \\
&\subseteq & f^{-1}(B(x,\epsilon)\times r_y) \\
&\subseteq&  f^{-1}(B) .
\end{array}
\]
Furthermore, $Q$ is a basic open set of $\Phyl_n$ satisfying item (3) of Definition \ref{defn:basis phyl}. Therefore $f$ is continuous.

It remains to show that $f$ is open. For this, let $U\subseteq \Phyl_n$ be an open set, and let $z\in U$, and $(x,y)=f(z)$ with $x\in \T_n$ and $y\in [0,\infty)^{n+1}$.   First suppose that $z$ is binary. Let $\ell_1,\dots , \ell_{n-2}$ denote the labels of the internal edges of  $z$, and $h_1,\dots , h_{n+1}$ the labels of the external edges. For  all $i=1,\dots , n+1$ we set $V_i= (h_i-\delta,h_i+\delta)$ if $h_i\ne 0$ and $V_i=[0,\delta)$ otherwise. Then  the set  
\[  
\begin{array}{ccl}
f^{-1}((x,y))
&\in & \{T\}\times R\\ 
& \subseteq & f^{-1}(B(x,\epsilon)\times r_y)\\ 
& \subseteq & f^{-1}(B)
\end{array}
\]
 is a basic open set of $\Phyl_n$ which is a neighborhood of  $z$ and, for $\delta$ small enough, contained in $U$. Thus we have   
\[  
\begin{array}{ccl}
f(z)
&\in & B(x,\begin{frac}{\delta}{2}\end{frac})\times V_1\times \dots \times V_{n+1}\\
&\subseteq & f(R)\\
&\subseteq & f(U).
\end{array}
\]
Now suppose that $z$ is not binary.  That is, suppose the underlying isomorphism class $T$ of $n$-trees of $z$ does not contain binary trees.   Let $\widetilde{T}_1,\dots , \widetilde{T}_a$ be the isomorphism classes of binary $n$-trees corresponding to $T$ (as defined in Definition  \ref{defn:basis phyl}). 
Let $\ell_1,\dots , \ell_k$ denote the labels of the internal edges of $z$, and $h_1,\dots , h_{n+1}$ the labels of the external edges. Similarly as before, we define 
$V_i= (h_i-\delta,h_i+\delta)$ if $h_i\ne 0$ and $V_i=[0,\delta)$ otherwise,
for all $i=1,\dots , n+1$. Denote by $R$ the open rectangle 
\[
(\ell_1-\delta,\ell_1+\delta)\times \dots \times (\ell_k-\delta,\ell_k+\delta)\times V_1\times \dots \times V_{n+1}.\]
 Then the set
\[
Q=\{T\}\times R\; \cup \bigcup_{i=1}^a \left(\{\widetilde{T}_i\}\times  R
 \times \underset{n-2-k \text{ times }}{\underbrace{(0,\delta)\times \dots \times (0,\delta)}}\right)
\]
is a basic open set in $\Phyl_n$ which is a neighborhood  of $z$ and, for $\delta$ small enough, contained in $U$.  Finally, we have  
\[
\begin{array}{ccl}
f(z)
&\in & B(x,\begin{frac}{\delta}{2}\end{frac})\times V_1\times \dots \times V_{n+1}\\
&\subseteq&  f(Q)\\
&\subseteq & f(U).  
\end{array}
\]
Therefore $f$ is open. This completes the proof of Theorem  \ref{thm:homeo}.
\end{proof}


\end{document}